\documentclass{amsart}
\usepackage{amscd}
\usepackage{enumerate,amssymb,amsmath,graphics,epsfig}
\usepackage[colorlinks]{hyperref}
\usepackage{color}
\usepackage{tcolorbox}
\usepackage{mathrsfs}
\usepackage{epstopdf}
\usepackage{booktabs}
\usepackage{subfigure}
\usepackage{bbding}
\usepackage{yfonts}
\usepackage{booktabs}
\usepackage{multirow}

\theoremstyle{plain}
\newtheorem{theorem}{Theorem}
\newtheorem{proposition}[theorem]{Proposition}

\theoremstyle{remark}

\newtheorem{remark}{Remark}

\newcommand{\RE}{\mathbb{R}}
\newcommand{\Hone}{H^1_0(\Omega)}
\newcommand\Hd{H^2_0(\Omega)}

\newcommand\Hdiv{\mathbf{H}(\div;\Omega)}
\newcommand\Hodiv{\mathbf{H}_0(\div;\Omega)}
\newcommand\Pinabla{\Pi_k^\nabla}
\newcommand\Pio{\Pi_k^0}
\newcommand\VhP{V_h^k(P)}
\newcommand\ahP{a_h^P}

\newcommand\m[1]{\mathsf{#1}}
\newcommand\SPa{S^P_a}
\newcommand\SPb{S^P_b}
\newcommand\edges{\mathcal{E}_h}
\newcommand\T{\mathcal{T}_h}
\newcommand\Tn{\mathcal{T}_n}
\newcommand\Vhnc{V_h^{k,nc}}
\newcommand\Hnc{H^{1,nc}(\T;k)}
\newcommand\jump[1]{\left[\!\left[#1\right]\!\right]}
\newcommand\nc{\mathcal{R}_h}
\newcommand\Hu{H^1(\Omega)}
\newcommand\bfsigma{\boldsymbol{\sigma}}
\newcommand\bftau{\boldsymbol{\tau}}
\newcommand\bfeta{\boldsymbol{\eta}}
\newcommand\bfn{\mathbf{n}}
\newcommand\bfu{\mathbf{u}}
\newcommand\bfv{\mathbf{v}}
\newcommand\bff{\mathbf{f}}
\newcommand\bfx{\mathbf{x}}
\renewcommand\div{\operatorname{\mathrm{div}}}
\newcommand\rot{\operatorname{\mathrm{rot}}}
\newcommand\adelta{a_\Delta^P}
\newcommand\adh{a_{\Delta,h}^P}
\newcommand\bdh{b_{\Delta,h}^P}
\newcommand\Pidelta{\Pi^\Delta_2}
\newcommand\symgrad{\operatorname{\underline{\boldsymbol{\varepsilon}}}}
\newcommand\smm{\text{-}}

\begin{document}

\title[]{Virtual element approximation of eigenvalue problems}

\author{Daniele Boffi}
\address{King Abdullah University of Science and Technology (KAUST), Saudi Arabia }
\email{daniele.boffi@kaust.edu.sa}
\urladdr{http://cemse.kaust.edu.sa/people/person/daniele-boffi}
\author{Francesca Gardini}
\address{Dipartimento di Matematica ``F. Casorati'', Universit\`a di Pavia}
\email{francesca.gardini@unipv.it}
\urladdr{http://www-dimat.unipv.it/gardini/}
\author{Lucia Gastaldi}
\address{DICATAM, Universit\`a di Brescia, Italy}
\email{lucia.gastaldi@unibs.it}
\urladdr{http://lucia-gastaldi.unibs.it}

%
\maketitle
%

\begin{abstract}
We discuss the approximation of eigenvalue problems associated with
elliptic partial differential equations using the virtual element method.
After recalling the abstract theory, we present a model problem, describing in
detail the features of the scheme, and highligting the effects of the
stabilizing parameters. We conlcude the discussion with a survey of several
application examples. 
\end{abstract}

\section{Introduction}
\label{se:intro}

In this chapter we discuss the use of the Virtual Element Method (VEM) for the
approximation of eigenvalue problems associated with partial differential
equations.
Eigenvalue problems are present in several applications and are the object of
an appealing and vast research area. It is known that the analysis of
numerical schemes for the approximation of eigenmodes is based on suitable a
priori estimates and appropriate compactness assumptions (see, for
instance~\cite{kato,baos,acta}).

Moreover, a good knowledge of the spectral properties of a discretization
scheme is essential for the stability analysis of transient problems.
It is well known, for instance, that the solution of parabolic and hyperbolic
linear evolution problems (heat/wave equation) can be presented as a Fourier
series in terms of eigenfunctions and eigenvalues of the corresponding
elliptic eigenproblem (Laplace eigenproblem). We refer the interested reader
to~\cite{RT} for more information and to~\cite{evo,borzi} for an example of how
things can go wrong if the eigenvalue problem is not correctly approximated.

The virtual element method has been successfully used for the approximation of
several eigenvalue problems. Starting from the pioneering
works~\cite{Steklov,GV}, where the analysis of the Steklov and the Laplace
eigenproblems was developed,
other applications of VEM to eigenvalue problems include an acoustic
vibration problem~\cite{acoustic}, plates
models~\cite{Kirchhoff,buckling}, linear elasticity~\cite{MRelasticity}, and
a transmission problem~\cite{transmission}.
Moreover, nonconforming, $p$, and $hp$ VEM have been considered
in~\cite{GMV,CGMMV} for the Laplace eigenvalue problem. 

The above mentioned references make use of classical tools for the spectral
analysis, such as the Babu\v ska--Osborn theory~\cite{baos} for compact
operators or the Descloux--Nassif--Rappaz theory~\cite{DNR} for non-compact
operators, typically adopted in connection with a shift procedure.
In the case of the mixed formulation for the Laplace
eigenproblem~\cite{mixed}, the conditions introduced in~\cite{bbg1,bbg2} are
considered.

An important feature of VEM, as compared to standard FEM, is that suitable
stabilizing forms, depending on appropriate parameters, have to be introduced
in order to guarantee consistency and stability of the approximation. Typical
theoretical results state that, for given choices of the stabilization
parameters, the discrete solution converges to the continuous one with optimal
order asymptotically in $h$ or $p$. In the case of eigenvalue problems, the
presence of the stabilizing forms may introduce artificially additional
eigenmodes and we have to make sure that they will not pollute the portion of
the spectrum we are interested in. Some of the above mentioned references
address this issue even if they are not conclusive in this respect.
In~\cite{auto-stab} we presented a systematic study of the eigenvalue
dependence on the parameters. If the discrete eigenvalue can be written in the
form $\m{A}\m{u}=\lambda\m{M}\m{u}$, with $\m{A}$ and $\m{M}$ depending on
$\alpha$ and $\beta$, respectively, then the \emph{quick and easy recipe} is
to pick a sufficiently large $\alpha$ and a small (possibly zero) $\beta$.
Such discussion is summarized in Section~\ref{se:parameter}.

After describing the general setting of a variationally posed eigenvalue
problem in Section~\ref{se:setting}, we recall the basic theory for the VEM
approximation of the Laplace eigenproblem in Section~\ref{se:theory}. The
case of nonconforming, $p$, and $hp$ VEM is treated in
Section~\ref{se:nonconf-hp}. As already mentioned, Section~\ref{se:parameter}
deals with the choice of the stabilizing parameters. Finally,
Section~\ref{se:appl} presents a survey of applications of VEM to the
discretization of eigenproblems of interest.

\section{Abstract setting}
\label{se:setting}

Let $V$ and $H$ be two Hilbert spaces such that $V\subset H$ with dense and
compact embedding, and $a(\cdot,\cdot)$ and $b(\cdot,\cdot)$ be two continuous
and symmetric bilinear forms defined on $V\times V$ and $H\times H$,
respectively.
We consider the following eigenvalue problem in variational form:
find $(\lambda,u)\in\mathbb{R}\times V$ with $u\neq 0$ such that
\begin{equation}
\label{eq:eig}
a(u,v)=\lambda b(u,v)\quad\forall v\in V. 
\end{equation}
The source problem associated with problem~\eqref{eq:eig} reads: given $f\in H$, find $u\in V$ such that 
\begin{equation}
\label{eq:source}
a(u,v)=b(f,v)\quad\forall v\in V.
\end{equation}
We assume that $a(\cdot,\cdot)$ is coercive on $V$, that is there exists $\alpha>0$ such that
$$
a(v,v)\ge\alpha \vert\vert v\vert\vert^2_V\quad\forall v\in V.
$$
Under these assumptions, there exists a unique solution $u\in V$ of problem~\eqref{eq:source} satisfying the following stability estimate
\begin{equation}
\label{eq:stability}
\vert\vert u \vert\vert_V\le C\vert\vert f \vert\vert_H.
\end{equation}
We then consider the solution operator $T: H\to V\subset H$ defined as follows: 
$$
Tf=u,\ u\in V,
$$
with $u$ the unique solution of the source problem~\eqref{eq:source}. Since
problem~\eqref{eq:source} is well posed, $T$ is a well defined, self-adjoint
linear bounded operator. Thanks to the compact embedding of $V$ in $H$, it
turns out that $T$ is also compact.  
The same conclusions apply to the situation when $V$ is not compact in $H$, by
assuming directly the compactness of the operator $T$.
We observe that $(\lambda,u)$ is an eigenpair of problem~\eqref{eq:eig} if and
only if $\big(\frac{1}{\lambda},u\big)$ is an eigenpair of $T$. Thanks to the
spectral theory of compact operators, we have that the
eigenvalues of problem~\eqref{eq:eig} are real and positive and form a
divergent sequence
\begin{equation}
\label{eq:eigvalues}
0 < \lambda_1\le\lambda_2\le\dots\le\lambda_i\le\dots;
\end{equation}
conventionally, we repeat them according to their multiplicities. The
corresponding eigenfunctions $u_i$ are chosen with the following properties 
$$
\vert\vert u_i\vert\vert_H = 1, \quad a(u_i,u_i)=\lambda_i,
$$ 
and so that they form an orthonormal basis both in $V$ and in $H$. 

We introduce a Galerkin type discretization of problems~\eqref{eq:eig}
and~\eqref{eq:source}. 
To this end, let $V_h\subset V$ be a finite dimensional subspace of 
$V$, and $a_h:V_h\times V_h\to\mathbb{R}$ and $b_h:H\times H\to\mathbb{R}$ two
discrete symmetric continuous bilinear forms. Then, the discrete eigenvalue
problem and the discrete source problems are given by: 
find $(\lambda_h,u_h)\in\mathbb{R}\times V_h$ with $u_h\neq 0$ such that
\begin{equation}
\label{eq:eig-discrete}
a_h(u_h,v_h)=\lambda_h b_h(u_h,v_h)\quad\forall v_h\in V_h, 
\end{equation}
and, given $f\in H$, find $u_h\in V_h$ such that
\begin{equation}
\label{eq:source-discrete}
a_h(u_h,v_h)=b_h(f,v_h)\quad\forall v_h\in V_h. 
\end{equation}
In order to guarantee existence, uniqueness, and stability of the solution to the discrete source problem~\eqref{eq:source-discrete}, we assume that the discrete bilinear form $a_h$ is coercive. 

Similarly to the continuous case, we introduce the discrete solution operator 
$T_h: H\to V_h\subset H$ defined by
$$
T_h f=u_h,
$$
with $u_h\in V_h$ the unique solution to the discrete source
problem~\eqref{eq:source-discrete}.  $T_h$ is a self-adjoint operator with
finite range, hence it is compact and has $N_h$ positive discrete eigenvalues
$$
0 < \lambda_{h,1}\le\lambda_{h,2}\le\dots\le\lambda_{h,N_h},
$$ 
where $N_h$ denotes the dimension of $V_h$.  Moreover, the discrete
eigenfunctions $u_{h,i}$ span $V_h$ and can be chosen as an orthogonal basis
satisfying 
$$
\vert\vert u_{h,i}\vert\vert_H = 1, \quad 
a(u_{h,i},u_{h,i})=\lambda_{h,i}\quad\forall i=1,\dots N_h.
$$
We recall some results of the spectral approximation theory for compact
operators~\cite{kato,baos,acta}. In the following we denote by $\delta(E,F)$ 
the gap between the spaces $E$ and $F$, that is 
$$
\delta(E,F)=\max(\hat\delta(E,F),\hat\delta(F,E))\quad \text{with }
\hat\delta(E,F)=\sup_{\substack{u\in E \\ ||u||_H=1}}\inf_{v\in F} ||u-v||_H. 
$$
\begin{theorem}
\label{th:uniform-conv}
Let us assume that 
\begin{equation}
\label{eq:spectral-approx}
\vert\vert T-T_h\vert\vert_{\mathcal{L}(H,H)}\to 0 \text{  as } h\to 0.
\end{equation}
Let $\lambda_i$ be an eigenvalue of problem~\eqref{eq:eig} with multiplicity
equal to $m$, namely $\lambda_i=\lambda_{i+1}=\dots =\lambda_{i+m-1}$. Then,
for $h$ small enough such that $N_h\ge i+m-1$, exactly $m$ discrete eigenvalues
$\lambda_{h,i}=\lambda_{h,i+1}=\dots =\lambda_{h,i+m-1}$ converge to
$\lambda_i$. 

Moreover, let $\mathcal{E}_i$ be the eigenspace of dimension $m$  
associated with $\lambda_i$ and 
$\mathcal{E}_{h,i}=\bigoplus_{j=i}^{i+m-1}\mathrm{span}(u_{h,j})$ be the direct sum of the eigenspaces associated with the discrete eigenvalues. Then 
\begin{equation}
\label{eq:eigfun}
\delta(\mathcal{E}_i,\mathcal{E}_{h,i})\to 0\quad \text{as } h\to 0,
\end{equation}
where $\delta(E,F)$ denotes the gap between the spaces $E$ and $F$.
\end{theorem}

\begin{theorem}
\label{th:abstract-rate-conv}
Using the same notation as in Theorem~\ref{th:uniform-conv}, then there exists
a positive constant $C$ independent of $h$ such that
\begin{equation}
\label{eq:delta}
\delta(\mathcal{E}_i,\mathcal{E}_{h,i})\le 
C ||(T-T_h)|_{\mathcal{E}_i}||_{\mathcal{L}(H,H)}.
\end{equation} 
Let $\phi_1,\ldots,\phi_m$ be a basis of the eigenspace $\mathcal{E}_i$
corresponding to the eigenvalue $\lambda_i$.
Then, for $=j=i,\ldots,i+m-1$
\begin{equation}
\label{eq:eig-rate}
|\lambda_j - \lambda_{j,h}| \le C 
\Big(\,\sum_{l,k=1}^m | ((T-T_{h})\phi_k,\phi_l)_H | + 
|| (T-T_{h})|_{\mathcal{E}_i} ||_{\mathcal{L}(H,H)}\Big),
\end{equation}
where $(\cdot,\cdot)_H$ stands for the scalar product in $H$.
\end{theorem}
\begin{remark}
\label{rm:TV}
The solution operator could also be defined as $T_V: V\to V$ such that 
for all $f\in V$, $T_V f$ is the solution to problem~\eqref{eq:source}. 
If $T_V$ is compact, then one can obtain results similar to those presented above with due modifications.
\end{remark}

\subsection{Model Problem}
Let $\Omega\subset\mathbb{R}^d$ ($d=2,3$) be an open bounded Lipschitz domain.
As a model problem, we consider the Laplace eigenvalue equation with
homogeneous Dirichlet boundary conditions: 
\begin{equation}
\label{eq:Laplace}
\aligned
& -\Delta u=\lambda u && \text{in }\Omega\\
& u=0 && \text{on }\partial\Omega.
\endaligned
\end{equation}
Setting
\begin{equation}
\label{eq:notation}
\aligned
& V=\Hone\quad H=L^2(\Omega)\\ 
& a(u,v)=\int_\Omega \nabla u\cdot\nabla v\,dx\quad b(u,v)=\int_\Omega uv\,dx
\endaligned
\end{equation}
the weak formulation of~\eqref{eq:Laplace} fits the above abstract setting.
Later on we will use the following additional regularity
result~\cite{agmon,grisvard}: if $f\in L^2(\Omega)$ there exists $r\in
\big(\frac{1}{2},1\big]$ such that $u\in H^{1+r}(\Omega)$, with
\begin{equation}
\label{eq:regularity}
\vert\vert u \vert\vert_{1+r}\le C\vert\vert f \vert\vert_0.
\end{equation}
The regularity index $r$ depends on the maximum interior wedge angle of 
$\Omega$. In particular, if the domain is convex $r$ can be taken equal to
$1$.

Here and in what follows, we adopt the usual notation
$\vert\vert\cdot\vert\vert_{s,D}$ and $\vert\cdot\vert_{s,D}$ for norm and
seminorm in the Sobolev space $H^s(D)$.  When $D=\Omega$ we omit the subindex
$\Omega$.  Moreover, we set $a^D(u,v)=\int_D\nabla u\cdot \nabla v\,dx$ and
$b^D(u,v)=\int_D u v\,dx$.  
\section{Virtual Element approximation of the Laplace eigenvalue problem}
\label{se:theory}
In this section, we consider the Laplace eigenproblem and, after recalling 
the definition and the approximation properties of VEM, we present the
discretization of the eigenvalue problem~\eqref{eq:eig} and the relative
convergence analysis. We treat in detail only the two dimensional case; for
the three dimensional one, which can be analyzed with the same arguments, we
refer to~\cite{GV,GMV}.
\subsection{Virtual Element Method}
\label{se:VEM}
We denote by $\mathcal{T}_h$ a family of polygonal decomposition of 
$\Omega$, by  $h_P$ the diameter of the element $P\in\mathcal{T}_h$, by
$h$ the maximum of such diameters, and by $\edges$ the set of the edges of the
mesh. Moreover, $\edges^0$ and $\edges^\partial$ stand for the subset of
internal and boundary edges, respectively.

We suppose that for all $h$, there exists a constant $\rho$ such that each
element $P\in\mathcal{T}_h$ is star-shaped with respect to a disk  with radius
greater than $\rho h_P$, and for every edge $e\in\partial P$ it holds that
$h_e\ge\rho h_P$ (see~\cite{BBCMMR2013}).
We observe that the scaling assumption implies that the number of edges in the
boundary of each element is uniformly bounded over the whole mesh family
$\mathcal{T}_h$. However, there is no restriction on the interior angles of the
polygons which can be convex, flat, or concave.

Let $k\ge1$ be a given integer, we denote by $\mathbb{P}_k(P)$ the space of
polynomials of degree at most equal to $k$. We introduce the {\itshape
enhanced} virtual element space $V_h$ of order $k$ (see~\cite{AABMR}). We
consider the local space
\begin{equation}
\label{eq:vemP}
V_h^k(P)=\left\{v\in\widetilde{V}_h^k(P):\int_P (v-\Pinabla v)p\,dx=0
\quad\forall p\in(\mathbb{P}_k\setminus\mathbb{P}_{k-2})(P) \right\},
\end{equation}
with the following definitions:
\begin{equation}
\label{eq:pinabla}
\aligned
& \widetilde{V}_h^k(P)=\left\{v\in H^1(P):v|_{\partial P}\in C^0(\partial P), 
v|_e\in\mathbb{P}_k(e)\ \forall e\in\partial P,
\Delta v\in\mathbb{P}_k(P)\right \} \\
& \Pinabla: \widetilde{V}_h^k(P)\to\mathbb{P}_k(P) \text{ is a projection operator defined by } \\
& \qquad\qquad a^P(\Pinabla v-v,p)=0\quad\forall p\in\mathbb{P}_k(P)\\
& \qquad\qquad \int_{\partial P} (\Pinabla v-v)\,ds=0. 
\endaligned
\end{equation}
In~\eqref{eq:vemP}, $(\mathbb{P}_k\setminus\mathbb{P}_{k-2})(P)$ stands for the
space of polynomials in $\mathbb{P}_k(P)$ $L^2$-othogonal to
$\mathbb{P}_{k-2}(P)$.

As shown in~\cite{AABMR}, it is possible to compute $\Pinabla v_h$ for all 
$v_h\in \VhP$ using the following degrees of freedom:
the values $v(V_i)$ at the vertices $V_i$ of $P$, the scaled moments up to order $k-2$ on each edge $e\subset\partial P$, and on the element $P$.

The global virtual element space is obtained by gluing together the local spaces, that is 
$$
V_h^k=\left\{ v\in\Hone:v|_P\in\VhP\ \forall P\in\mathcal{T}_h \right\}.
$$
From now on, we denote by $N_h$ the dimension of $V_h^k$. 

We highlight that, differently form finite elements, we do not have at hand the explicit 
knowledge of the basis functions of $V_h^k$; in this sense the basis functions are  {\itshape virtual}. On the other hand, polynomials of degree at most $k$ are in the VEM space; this will guarantee the optimal order of accuracy. Actually, the following approximation results hold true, 
see~\cite{brennerscott,BBCMMR2013,CGPS}.
\begin{proposition}
\label{prop:approx}
There exists a positive constant $C$, depending only on the polynomial degree
$k$ and the shape regurality $\rho$, such that for every $s$ with $1\le s \le
k+1$ and for every $v\in H^s(\Omega)$ there exists $v_{\pi}\in\mathbb{P}_k(P)$
such that
\begin{equation}
\|v-v_{\pi}\|_{0,P}+h_P |v-v_{\pi}|_{1,P}\le Ch_P^s|v|_{s,P}.
\end{equation}
Moreover, there exists a constant $C$, depending only on the polynomial degree
$k$ and the shape regurality $\rho$, such that for every $s$ with $1\le s \le
k+1$, for every $h$,  and for every $v\in H^s(\Omega)$ there exists $v_I\in
V_h^k$ such that
\begin{equation}
\|v-v_I\|_{0}+h_P |v-v_I|_{1}\le Ch_P^s|v|_{s}.
\end{equation}
\end{proposition}

Notice that $v_\pi$ is defined element by element, and does not belong to the
space $H^1(\Omega)$. We shall denote its broken $H^1$-seminorm by $\vert
v_\pi\vert_{1,h}$.

\subsection{The VEM discretization of the Laplace eigenproblem}
\label{se:discrete-Laplace}
In the VEM context, the discrete bilinear forms $a_h(\cdot,\cdot)$ and
$b_h(\cdot,\cdot)$ are written as sum over the elements $P\in\mathcal{T}_h$ of
local contributions $a_h^P(\cdot,\cdot)$ and $b_h^P(\cdot,\cdot)$ for all
$P\in\mathcal{T}_h$, that is 
$$
a_h(\cdot,\cdot)=\sum_{P\in\mathcal{T}_h}a_h^P(\cdot,\cdot),\quad
b_h(\cdot,\cdot)=\sum_{P\in\mathcal{T}_h}b_h^P(\cdot,\cdot).
$$
For the readers' convenience, we recall here the discrete source
problem~\eqref{eq:source-discrete}:
given $f\in L^2(\Omega)$, find $u_h\in V_h^k$ such that
\begin{equation}
\label{eq:source-discrete1}
a_h(u_h,v_h)=b_h(f,v_h)\quad\forall v_h\in V_h^k. 
\end{equation}
It is well known that solving this problem is equivalent to solve a linear system of equations 
$\m{A}\m{u}=\m{f}$, where $\m{A}$ is the $N_h\times N_h$ matrix
corresponding to the form $a_h$ and $\m{f}$ is the vector with components
$b_h(f,\varphi_h)$ with 
$\varphi_h$ the basis functions of $V_h^k$.

We now address the issue of the construction of the discrete bilinear forms along the lines of~\cite{BBCMMR2013}. 
First of all, they are required to be similar to the continuous ones, namely 
$$
a_h^P(v_h,v_h)\approx a^P(v_h,v_h),\quad 
b_h^P(v_h,v_h)\approx b^P(v_h,v_h).
$$
We underline that the continuous forms $a(\cdot,\cdot)$ and $b(\cdot,\cdot)$
in general are not computable on the basis functions, except for those which
are polynomials. Hence we need to project the elements of $V_h^k(P)$ onto 
$\mathbb{P}_k(P)$.
This could be done locally as follows:
\begin{equation}
\label{eq:ah1}
a_h^P(u_h,v_h)=a^P(\Pinabla u_h,\Pinabla v_h)\quad
\forall u_h,v_h\in\mathcal{T}_h,
\end{equation}
but it might happen that a non vanishing element $u_h\in V_h^k(P)$ is such that
$\Pinabla u_h=0$, and so the form $a_h^P(u_h,v_h)=0$ for all $v_h\in \VhP$. In
this case the local contribution to the matrix $\m{A}_h$ results to be singular
and the global matrix might lack control on some components of $V_h^k$. To
avoid this eventuality,  we add a stability term as follows:
\begin{equation}
\label{eq:ah}
a_h^P(u_h,v_h)=a^P(\Pinabla u_h,\Pinabla v_h) + 
\SPa((I-\Pinabla)u_h,(I-\Pinabla)v_h),
\end{equation}
where $\SPa(\cdot,\cdot)$ is a symmetric positive definite bilinear form
defined on $\VhP\times \VhP$ scaling like $a^P(\cdot,\cdot)$.

For proper choices of $\SPa$, it turns out that $a_h^P(\cdot,\cdot)$ 
fulfils the following local {\itshape stability} and {\itshape consistency}
properties for all $P\in\mathcal{T}_h$.
\begin{description}
\item[Stability:]{there exists two positive constants $a_*$ and  
$a^*$ such that for all $v_h\in\VhP$
\begin{equation}
\label{eq:staba}
a_*a^P(v_h,v_h)\le a_h^P(v_h,v_h)\le a^*a^P(v_h,v_h).
\end{equation}
}
\item[Consistency:]{ for all $v_h\in\VhP$ and for all $p_k\in\mathbb{P}_k(P)$ it holds
\begin{equation}
\label{eq:consa}
a_h^P(v_h,p_k)=a^P(v_h,p_k).
\end{equation}
}
\end{description}
In order to deal with the right hand side in~\eqref{eq:source-discrete1}, we define another projection operator $\Pio$ from $L^2(P)$ onto $\mathbb{P}_k(P)$ as 
$$
b^P(\Pio f,p_k)=b^P(f,p_k)\quad\forall p_k\in\mathbb{P}_k(P),\ \forall P\in
\mathcal{T}_h.
$$
This projection operator is computable starting from the degrees of freedom as well. 
Thus the local contribution to the right hand side is given by
\begin{equation}
\label{eq:rhs}
b_h^P(f,v_h)=b^P(\Pio f,v_h)\quad\forall v_h\in\VhP,\ \forall P\in\mathcal{T}_h.
\end{equation}
With the above definitions, the discrete source problem is well-posed and its solution converges to the continuous one with optimal order, (see~\cite{BBCMMR2013,AABMR}): there exists a positive constant C, independent of $h$ such that 
\begin{equation}
\label{eq:error}
\vert\vert u-u_h \vert\vert_1\le C\left ( \vert u-u_I\vert_1 + \vert u - u_{\pi}\vert_{1,h} 
+ \sup_{v_h\in V_h^k}\frac{\vert b(f,v_h)-b_h(f,v_h)
\vert}{\vert v_h\vert_1}\right ),
\end{equation}
where $u_I\in V_h^k$ and $u_{\pi}$ are defined in Proposition~\ref{prop:approx}. We observe that the last term in~\eqref{eq:error} is a consistency term in the spirit of the 
\emph{first Strang Lemma}~\cite[Th. 4.1.1]{ciarlet}.  Indeed, the VEM discretization provides a conforming discrete subspace of $V$, but the 
bilinear forms $a_h$ and $b_h$ differ from the continuous ones. The consistency error generated by the difference between $a$ and $a_h$ is 
already incorporated in the first two terms on the right hand side of~\eqref{eq:error}, while the one related to $b$ depends on the properties 
of the source term $f$. In particular, if 
$f\in L^2(\Omega)$ we obtain that the error converges to zero with optimal order depending on the regularity of the solution $u$:
$$
\vert\vert u-u_h \vert\vert_1\le C \left( h^r \vert u \vert_{1+r} + h\vert\vert f\vert\vert_0 \right).
$$

We now turn to the virtual element discretization of the eigenvalue
problem~\eqref{eq:eig}. Let us recall the discrete formulation: find
$(\lambda_h,u_h)\in\mathbb{R}\times V_h^k$ with $u_h\neq 0$ such that
\begin{equation}
\label{eq:eig-discrete1}
a_h(u_h,v_h)=\lambda_h b_h(u_h,v_h)\quad\forall v_h\in V_h^k, 
\end{equation}
where, using definition~\eqref{eq:rhs} 
\begin{equation}
\label{eq:bh0}
b_h^P(u_h,v_h)=b^P(\Pio u_h,v_h)=b^P(\Pio u_h,\Pio v_h).
\end{equation}
The following generalized algebraic eigenvalue problem corresponds to the
discrete eigenvalue problem:
\begin{equation}
\label{eq:algebraic}
\m{A}\m{u}=\lambda_h \m{M}\m{u},
\end{equation}
where $\m{A}$ and $\m{M}$ are the $N_h\times N_h$ matrices associated with the
bilinear forms $a_h$ and $b_h$, respectively. 

As observed for the form $\ahP$, it might happen that $\Pio u_h=0$ for a non
vanishing $u_h$ and thus  $b_h^P(u_h,v_h)=0$ for all $v_h\in V_h^k$. This means
that the local contribution to the matrix $\m{M}$ has a non trivial kernel
similarly to what we have observed for the form $\ahP$ and some components of
$V_h^k$ may not be controlled by the global mass matrix. It is then natural to
add a stabilization term as follows:
\begin{equation}
\label{eq:bh1}
b_h^P(u_h,v_h)=b^P(\Pio u_h,\Pio v_h) + \SPb((I-\Pio)u_h,(I-\Pio)v_h),
\end{equation}
where $\SPb$ is a symmetric positive definite bilinear form defined on
$\VhP\times \VhP$ scaling as $b^P(\cdot,\cdot)$, that is there exist two
positive constants $\beta_*$ and $\beta^*$ such that 
\begin{equation}
\label{eq:stabb}
\beta_*b^P(v_h,v_h)\le \SPb(v_h,v_h)\le\beta^*b^P(v_h,v_h)\quad\forall 
v_h\in \VhP, \ \forall P\in\mathcal{T}_h.
\end{equation}

\begin{remark}

It is out of the aims of this paper to discuss the choice of algorithms for
the solution of the algebraic eigenvalue problem
$\m{A}\m{u}=\lambda\m{M}\m{u}$. Nevertheless, the presence of the stabilizing
parameters may influence such choice. For instance, if the matrix $\m{M}$ is
singular and $\m{A}$ is not, then it might be convenient to solve the
reciprocal system $\m{M}\m{u}=\lambda^{-1}\m{A}\m{u}$. The definition itself
of solution for problems like this can be difficult to give, in particular,
when both $\m{A}$ and $\m{M}$ are singular and their kernels have a non
trivial intersection. On the other hand if $\m{u}$ belongs to the kernel of
$\m{M}$ and not to that of $\m{A}$, we may conventionally say that it is an
eigenfunction corresponding to the eigenvalue $\lambda=\infty$.

\end{remark}

From now on, unless explicitly stated, we shall consider the discrete eigenvalue problem~\eqref{eq:eig-discrete1} with the local form $b_h^P$ defined as in 
\eqref{eq:bh1}. 
\subsection{Convergence analysis}
\label{se:analysis}

In this section we show the convergence of the discrete eigenpairs to the
continuous ones by applying Theorem~\ref{th:uniform-conv} and prove 
a priori error estimates. 

In order to apply Theorem~\ref{th:uniform-conv} we need to prove $L^2$-error
estimates for the source problems~\eqref{eq:source}
and~\eqref{eq:source-discrete}. For the sake of completeness, we report here
the proof. 

\begin{theorem}
\label{th:L2-error}
Given $f\in L^2(\Omega)$, let $u\in\Hone$ and $u_h\in V_h^k$ denote the solutions to~\eqref{eq:source} and to~\eqref{eq:source-discrete}, respectively. Then there exists a constant $C$ independent of $h$ such that 
\begin{equation}
\label{eq:L2-error}
||u-u_h||_0\le C h^t \Big(|u-u_I|_1 + |u-u_\pi|_{1,h} + ||f-\Pio f||_0  \Big ),
\end{equation}
where $t=\min(r,1)$, being $r$ the regularity index of the solution $u$,
see~\eqref{eq:regularity} and $u_I$ and $u_\pi$ are defined in
Proposition~\ref{prop:approx}.
\end{theorem}

\begin{proof}
We use a duality argument and denote by $\psi\in\Hone$ the solution to 
\begin{equation}
\label{eq:duality}
a(\psi,v)=b(u-u_h,v)\quad\forall v\in\Hone.
\end{equation}
Since $u-u_h\in L^2(\Omega)$, then $\psi\in H^{1+r}(\Omega)$ with 
$||\psi||_{1+r}\le C ||u-u_h||_0$. Let $\psi_I\in V_h^k$ be the interpolant of 
$\psi$ given by Proposition~\ref{prop:approx}, then for $t=\min(r,1)$ it holds
\begin{equation}
\label{eq:stima-psi}
||\psi-\psi_I||_0 + h |\psi-\psi_I|_1\le C h^{1+t} ||u-u_h||_0.
\end{equation}
We have that
\begin{equation}
\aligned
||u-u_h||_0^2 & = b(u-u_h,u-u_h)=a(u-u_h,\psi)\\
& = a(u-u_h,\psi-\psi_I) + a(u-u_h,\psi_I)= I + I\!I.
\endaligned
\end{equation}
We first estimate the term $I$ as follows
\begin{equation}
		\label{eq:28}
\aligned
I = a(u-u_h,\psi-\psi_I)\le C \big (|u-u_h|_1 h^t ||u-u_h||_0 \big).
\endaligned
\end{equation}
Then 
$$
\aligned
I\!I = a(u-u_h,\psi_I) & = b(f,\psi_I) - a_h(u_h,\psi_I) + a_h(u_h,\psi_I) 
-a(u_h,\psi_I)\\ 
& = [b(f,\psi_I)-b_h(f,\psi_I)] + [a_h(u_h,\psi_I) - a(u_h,\psi_I)]\\
& = I\!I\!I + I\!V.
\endaligned
$$
We have that
\begin{equation}
\label{eq:3}
I\!I\!I = \sum_P\Big ( b^P(f,\psi_I) - b^P(\Pio f,\Pio\psi_I)
-S^P_b((I-\Pio)f,(I-\Pio)\psi_I) \Big).
\end{equation}
It holds
\begin{equation}
\aligned
b^P(f,\psi_I) - b^P(\Pio f,\Pio\psi_I) & = b^P(f-\Pio f,\psi_I-\Pio\psi_I)\\
& \le Ch^t||f-\Pio f||_0||u-u_h||_0
\endaligned
\end{equation}
and
\begin{equation}
\aligned
S_b^P((I-\Pio)f,(I-\Pio)\psi_I) \Big) & \le \beta^*||(I-\Pio)f ||_0 ||(I-\Pio)\psi_I ||_0 \\
& \le C h^t ||(I-\Pio)f ||_0 ||u-u_h||_0.
\endaligned
\end{equation}
Finally, estimating the term $IV$ with standard VEM argument, we obtain
$$
I\!V\le C \Big(  ||u-u_h||_1 + ||u - \Pio u||_{1,h} \Big) h^t ||u-u_h||_0.
$$
Putting together all the estimates, we conclude the proof.
\end{proof}


The $L^2$-error estimate proved above implies the uniform convergence stated in Theorem~\ref{th:uniform-conv}. 

\begin{theorem}
\label{thm:thm1}
Let $T_{h}$ and $T$ be the families of operators associated with
problems~\eqref{eq:source-discrete} and ~\eqref{eq:source}, respectively. Then
the following uniform convergence holds true:
\begin{equation*}
||T-T_{h}||_{\mathcal{L}(L^2(\Omega), L^2(\Omega))}\to 0
\quad\textrm{for}\quad h\to 0.
\end{equation*}
\end{theorem}
\begin{proof}
Given $f\in L^2(\Omega)$, let $Tf$ and $T_hf$ be the solutions to
problems~\eqref{eq:source} and~\eqref{eq:source-discrete}, respectively.
Then, using the $L^2$-estimate of Theorem~\ref{th:L2-error}, the interpolation
and approximation results in Proposition~\ref{prop:approx}, and the stability
condition~\eqref{eq:regularity}, we obtain 
\[
||Tf-T_hf||_{0} \le C h^t ||f||_{0},
\]
where $t=\min(k,r)$, $k\geq 1$ is the order of the method, $r$ the regularity
exponent in~\eqref{eq:regularity}, and $C$ a constant independent of $f$ and
$h$.
From this inequality it follows that
\[
||T-T_{h}||_{\mathcal{L}(L^2(\Omega),L^2(\Omega))} 
= \sup_{f\in L^2(\Omega)}\frac{||Tf-T_{h}f||_0}{||f||_0} \le C h^t,
\]
which implies the uniform convergence.
\end{proof}

We now turn to error estimates for the approximation of the eigenvalues and 
the eigenfunctions. First of all we state a result on the rate of convergence 
for the solutions to the source problems when $f$ is smooth.
\begin{proposition}
\label{co:stima-L2}
Assume that $f\in H^{1+s_1}(\Omega)$ and that the solution $u$ of
problem~\eqref{eq:source} belongs to $H^{1+s_2}(\Omega)$ for some $s_1$ and
$s_2$ greater than $0$, then the estimate~\eqref{eq:L2-error} implies
\begin{equation}
\label{eq:stima-L2-detail}
||u-u_h||_0\le C h^t \Big(h^{\min(k,s_2)}|u|_{1+s_2} + h^{1+\min(k,s_1)} 
|f|_{1+s_1} \Big) 
\end{equation}
where $t=\min(1,r)$.
\end{proposition}
\begin{remark}
We point out that an eigenfunction $u$ can be seen as the solution of a source
problem with right hand side equal to $\lambda u$, which belongs at least to
$\Hone$. Hence the second term in the right side of~\eqref{eq:stima-L2-detail}
is at least of first order. In general, if the domain is non convex, the
solution $u$ to~\eqref{eq:source} belongs at least  to $H^{1+r}(\Omega)$ with
$r\in (\frac{1}{2},1]$, even if the right hand side is more regular,
see~\eqref{eq:regularity}. However, it might happen that $u$ is more regular,
for example that it belongs to the space $H^{1+s}(\Omega)$ with $s\ge r$, then
estimate~\eqref{eq:stima-L2-detail} reduces to 
\begin{equation}
\label{eq:stima-regolare}
||u-u_h||_0\le C h^{t+\min(k,s)} |u|_{1+s}.
\end{equation}
\end{remark}

Taking into account the above $L^2$-estimates and using the abstract Theorem~\ref{th:abstract-rate-conv}, we obtain the optimal rate of convergence for the approximation of the eigenpairs. 
\begin{theorem}
\label{th:stima-gap}
Let $\lambda_i$ be an eigenvalue of problem~\eqref{eq:eig}, with multiplicity
$m$ (that is $\lambda_i=\dots=\lambda_{i+m-1}$) and let $\mathcal{E}_i$ be the
corresponding eigenspace. We assume that all the elements in $\mathcal{E}_i$
belong to $H^{1+s}(\Omega)$ for some $s\ge r$ (see~\eqref{eq:regularity}).
Moreover, let $\mathcal{E}_{i,h}=\bigoplus_{j=i}^{i+m-1}\mathrm{span}(u_{h,j})$,
where $u_{h,j}$ is the discrete eigenfunction associated with $\lambda_{h,j}$.  
Then, for $h$ small enough, there exists a constant $C$ independent of $h$ such
that
\begin{equation}
\delta(\mathcal{E}_{i}, \mathcal{E}_{i,h})\leq C h^{t+\min(k,s)}
\end{equation}
with $t=\min(1,r)$.
\end{theorem}
\begin{proof}
The result directly stems from~\eqref{eq:delta}. Indeed, 
\begin{equation}
\label{eq:T-Th}
\delta(\mathcal{E}_{i}, \mathcal{E}_{i,h})\leq C
||(T-T_h)|_{\mathcal{E}_i}||_{\mathcal{L}(L^2(\Omega), L^2(\Omega))} 
= C \sup_{\substack{u\in\mathcal{E}_i \\ ||u||_0=1}} 
||(T-T_h)u||_0.
\end{equation}
Since $u\in\mathcal{E}_i$, it holds that $u\in H^{1+s}(\Omega)$ with 
$||u||_{1+s}\le C ||u||_0$ (see~\eqref{eq:regularity}). Hence the result follows 
from estimate~\eqref{eq:stima-regolare}.
\end{proof}
Now we state the \emph{a priori} error estimates for the eigenvalues. 
\begin{theorem}
\label{th:eig-rate}
Using the notation of Theorem~\ref{th:stima-gap}, there exists a constant 
$C$ independent of $h$, but depending on $\lambda_i$ such that 
the following error estimate for the eigenvalues holds true:
$$
|\lambda_j-\lambda_{j,h}|\le C h^{2\min(k,s)}\quad\text{ for } j=i,\dots, i+m-1.
$$ 
\end{theorem}
\begin{proof}
We apply Theorem~\ref{th:abstract-rate-conv}.  
Thanks to~\eqref{eq:T-Th}, it remains to bound the first term 
in~\eqref{eq:eig-rate}. 
This is equivalent to estimate $b((T-T_{h})w,z)$ for all $w$ and $z$ 
in $\mathcal{E}_i$. After some computations we get
$$
\aligned
b((T-T_{h})w,z) & = a(Tz,(T-T_h)w) \\
& = a((T-T_h)z,(T-T_h)w) + b(T_hz,w)-b_h(T_hz,w) \\
& \qquad - a(T_hz,T_hw) + a_h(T_hz,T_hw) \\
& = I' + I\!I' + I\!I\!I'. 
\endaligned
$$ 
We estimate separately the three terms. It is straightforward to obtain 
\begin{equation}
\label{eq:I}
I' \le ||(T-T_h)z||_1 ||(T-T_h)w||_1 \le C h^{2\min(k,s)}.
\end{equation}
In order to estimate the second term, we proceed as in~\eqref{eq:3} by writing it as a sum over the elements $P\in\mathcal{T}_h$,
\begin{equation}
\label{eq:II}
\aligned
I\!I' & = \sum_P \Big( b^P(T_hz - \Pio (T_hz),w - \Pio w) - 
S_b^P((I-\Pio)T_hz,(I-\Pio)w)\Big) \\
& \le C \sum_P || (I-\Pio)T_hz ||_{0,P} || I-\Pio)w ||_{0,P} \le C h^{2\min(k,s)}.
\endaligned
\end{equation}
Finally,
\begin{equation}
\label{eq:III}
\aligned
I\!I\!I' & =\sum_P \Big(a_h^P(T_hz - \Pio Tz,T_hw - \Pio Tw) 
- a^P(T_hz - \Pio Tz,T_hw - \Pio Tw) \Big) \\
& \le C \sum_P \Big( |T_hz-Tz|_{1,P} +  |(I-\Pio) Tz|_{1,P} \Big) \\
& \qquad\qquad\qquad\Big( |T_hw-Tw|_{1,P} +  |(I-\Pio) Tw|_{1,P} \Big)
\le C h^{2\min(k,s)}.
\endaligned
\end{equation}
Collecting estimates~\eqref{eq:I}-\eqref{eq:III} yields the required estimate. 
\end{proof}
Bearing in mind Remark~\ref{rm:TV}, we state the rate of convergence for 
the error in the approximation of the eigensolutions with respect to the
$H^1$-norm.
To this aim we denote by $\delta_1(E,F)$ the gap between the spaces $E$ and $F$
measured in the $H^1$-norm.
\begin{theorem}
\label{th:delta1}
Using the same notation as in Theorem~\ref{th:stima-gap}, and assuming that all the elements in $\mathcal{E}_i$ belong to $H^{1+s}(\Omega)$ for some $s\ge r$, then for $h$ small enough, there exists a constant $C$ independent of $h$ such that
$$
\delta_1(\mathcal{E}_{i}, \mathcal{E}_{i,h})\leq C h^{\min(k,s)}.
$$
\end{theorem}
\begin{remark}
\label{rm:nonstab}
The above analysis has been carried on considering the stabilized discrete 
bilinear form $b_h$, whose local counterpart is given in~\eqref{eq:bh1}. However, in~\cite{GV} it has been shown that the results stated in Theorems~\ref{th:stima-gap}-\ref{th:delta1} hold true also when $b_h$ is defined using the \emph{non-stabilized} local version in~\eqref{eq:bh0}.
\end{remark}

\subsection{Numerical results}
\label{se:test-conf}

Here and in the following sections devoted to the numerical experiments, we
focus on the solution to problem~\eqref{eq:Laplace} on a square and an
\emph{L-shaped} domain. It is well known that in the first test case the
eigenfunctions are analytic, while in the latter they can be singular due to
the presence of a reentrant corner in the domain. All the results that we are
going to present are not new, including the figures which have been already
published in~\cite{GV,GMV,CGMMV}.

The implementation of the method requires a precise definition for the
stabilization terms. Here we are going to consider one possible choice,
other possibilities are reported in~\cite{GV}. For $v_h$ and $w_h$ in $\VhP$, the
stabilizing bilinear forms $S^P_a(v_h,w_h)$ and $S^P_b(v_h,w_h)$
(see~\eqref{eq:ah} and~\eqref{eq:bh1}) are defined using the local degrees of
freedom associated to $\VhP$. Let $N_P$ be the dimension of 
$\VhP$, then $\m{v}$ and $\m{w}$ stand for the vectors in $\mathbb{R}^{N_P}$
with components the local degrees of freedom associated to $v_h$ and $w_h$.
More precisely, we set 
\begin{equation}
\label{eq:def-stab}
S_a^P(v_h,w_h) = \alpha_P\m{v}^\top\m{w}\qquad 
\text{ and }\qquad S_b^P(v_h,w_h) = \beta_Ph_P^2\m{v}^\top\m{w}, 
\end{equation}
where $\alpha_P$ and $\beta_P$ are constants independent of $h$.  The numerical
results we report here are obtained choosing $\alpha_P$ as the mean value of
the eigenvalues of the local matrix associated to the consistency term
$a^P(\Pinabla v_h,\Pinabla w_h)$. The parameter $\beta_P$ is given by the mean
value of the eigenvalues of the local mass matrix associated to
$\frac{1}{h_P^2}b^P(\Pio v_h,\Pio w_h)$. With this definition $\alpha_P$ and
$\beta_P$ depend only on the shape of $P$.

{\bf Square domain} Let $\Omega$ be the square $(0,\pi)^2$, then the
eigensolutions of~\eqref{eq:Laplace} are given by 
\begin{equation}
\label{eq:square-eig}
\aligned
& \lambda_{i,j} = (i^2 + j^2)\quad \text{for } i,j\in\mathbb{N},
\text{ with } i,j\neq 0\\
& u_{i,j}(x,y) = \sin(ix) \sin(jy)\quad \text{for } (x,y)\in\Omega.
\endaligned
\end{equation}
We consider different refinements of a \emph{Voronoi} mesh, an example of them
is shown in Figure~\ref{fg:voronoi} left. 
\begin{figure}
\begin{center}
\includegraphics[width=.45\textwidth]{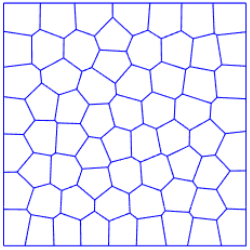}
\hspace{.5cm}
\includegraphics[width=.45\textwidth]{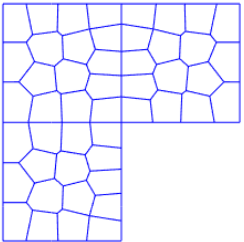}
\end{center}
\caption{Examples of Voronoi mesh of the unit square and the L-shaped domain}
\label{fg:voronoi}
\end{figure}
Figure~\ref{fg:rate-square} reports the error for the first six eigenvalues
obtained with different polynomial degrees $k=1,2,3,4$ on meshes with size
$h=\frac{\pi}{8},\frac{\pi}{16}, \frac{\pi}{32}, \frac{\pi}{64}$. The slopes
of the lines corresponding to the error for each eigenvalue clearly reflects
the theoretical rate of convergence stated in Theorem~\ref{th:eig-rate}. For
$k=4$ the error is close to the machine precision for the two last refinements,
therefore the effect of propagation of rounding error prevents further
decreasing. 
\begin{figure}
\begin{center}
\includegraphics[width=\textwidth]{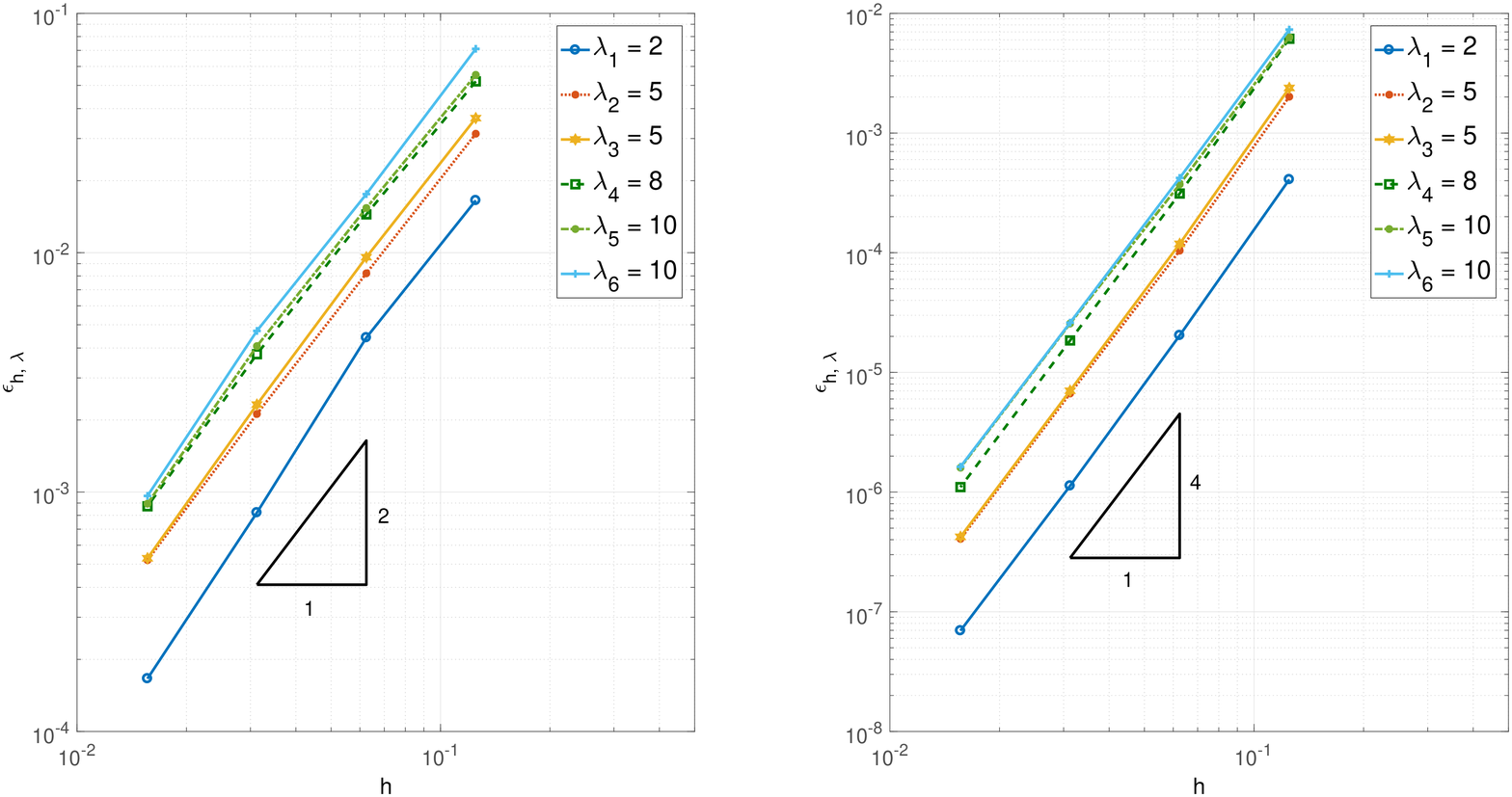} \\
\includegraphics[width=\textwidth]{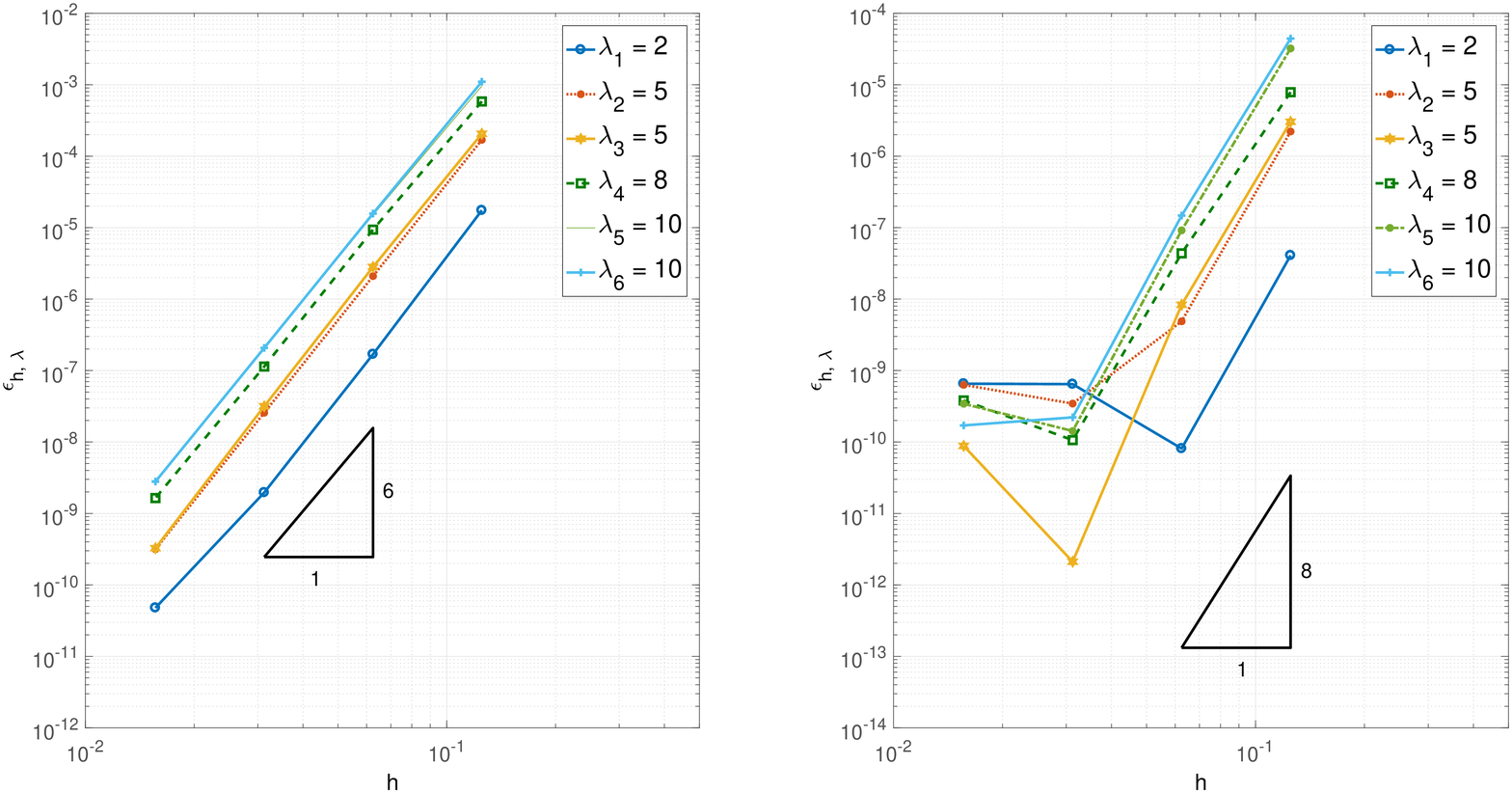}
\end{center}
\caption{Rate of convergence for the first six eigenvalues on the square
domain. Top left: $k=1$, top right: $k=2$, bottom left: $k=3$, bottom right: 
$k=4$}
\label{fg:rate-square}
\end{figure}
Remark~\ref{rm:nonstab} points out that the stabilization on the right hand
side is not necessary at the theoretical level. 
In order to understand if it might impact the numerical results, we report in
Table~\ref{tb:paragone} a comparison between the errors obtained using the
stabilized or non-stabilized bilinear form $b_h$, corresponding to the first
four distinct eigenvalues for $k=1$ and $k=4$. It might be appreciated that the
magnitude of the error is actually the same. A more detailed discussion on the
role of the stabilizing parameters will performed in
Section~\ref{se:parameter}.
\begin{table}[!h]
\centering
\begin{tabular}{c|cc|cc|cc|cc}
\toprule
\multicolumn{9}{c}{$k=1$}\\
\midrule
$h$ &\multicolumn{2}{|c|}{$\lambda_1=2$}
    &\multicolumn{2}{c|}{$\lambda_2=5$}
    &\multicolumn{2}{c|}{$\lambda_4=8$}
    &\multicolumn{2}{c}{$\lambda_5=10$}\\
\midrule
&stab& non stab&stab& non stab&stab& non stab&stab& non stab\\
\midrule
$\pi/8$
     & $1.654e\smm2$ & $1.766e\smm2$ & $3.650e\smm2$ & $4.098e\smm2$
     & $5.196e\smm2$ & $6.384e\smm2$ & $7.107e\smm2$ & $8.488e\smm2$\\
$\pi/16$
     & $4.418e\smm3$ & $4.611e\smm3$ & $9.578e\smm3$ & $1.026e\smm2$
     & $1.446e\smm2$ & $1.558e\smm2$ & $1.758e\smm2$ & $1.926e\smm2$ \\
$\pi/32$
     & $8.206e\smm4$ & $8.784e\smm4$ & $2.320e\smm3$ & $2.465e\smm3$
     & $3.768e\smm3$ & $4.003e\smm3$ & $4.715e\smm3$ & $5.038e\smm3$ \\
$\pi/64$
     & $1.662e\smm4$ & $1.808e\smm4$ & $5.289e\smm4$ & $5.663e\smm4$
     & $8.722e\smm4$ & $9.340e\smm4$ & $9.643e\smm4$ & $1.042e\smm3$ \\
\midrule
\multicolumn{9}{c}{$k=4$}\\
\midrule
$h$ &\multicolumn{2}{|c|}{$\lambda_1=2$}
    &\multicolumn{2}{c|}{$\lambda_2=5$}
    &\multicolumn{2}{c|}{$\lambda_4=8$}
    &\multicolumn{2}{c}{$\lambda_5=10$}\\
\midrule
&stab& non stab&stab& non stab&stab& non stab&stab& non stab\\
\midrule
$\pi/8$
    & $4.087e\smm08$ & $4.094e\smm08$ & $3.033e\smm06$ & $3.046e\smm06$
    & $7.856e\smm06$ & $7.915e\smm06$ & $4.428e\smm05$ & $4.470e\smm05$\\
$\pi/16$
    & $8.101e\smm11$ & $8.098e\smm11$ & $8.329e\smm09$ & $8.340e\smm09$
    & $4.364e\smm08$ & $4.373e\smm08$ & $1.475e\smm07$ & $1.475e\smm07$ \\
$\pi/32$
    & $6.451e\smm10$ & $6.449e\smm10$ & $2.109e\smm12$ & $1.283e\smm12$
    & $1.062e\smm10$ & $1.060e\smm10$ & $2.211e\smm10$ & $2.223e\smm10$ \\
$\pi/64$
    & $6.541e\smm10$ & $6.542e\smm10$ & $8.818e\smm11$ & $8.829e\smm11$
    & $3.815e\smm10$ & $3.817e\smm10$ & $1.706e\smm10$ & $1.707e\smm10$ \\
\bottomrule
\end{tabular}
\caption{Comparison of the errors for the first eigenvalues between stabilized
and non stabilized bilinear form $b_h$}
\label{tb:paragone}
\end{table}

{\bf L-shaped domain}
The Laplace eigenvalue problem with Neumann boundary conditions on the non
convex L-shaped domain $\Omega=(-1,1)^2\setminus(0,1)\times(-1,0)$ (displayed
in Figure~\ref{fg:voronoi} right with an example of the used Voronoi mesh)
is a well-known benchmark test to check the capability of the numerical scheme
to approximate singular eigensolutions.
The analysis presented above for the Dirichlet eigenvalue problem extends
analogously to this situation.
In Figure~\ref{fg:L-shaped}, we show the convergence history for the first and
the third eigenvalues computed with $k=1,2,3$ and the stabilization parameter
chosen as in~\eqref{eq:def-stab}. The reference value for the eigenvalues has
been taken form the benchmark solution set~\cite{benchmark}. Due to the
presence of the reentrant corner, the first eigensolution is singular
belonging to $H^{1+r}(\Omega)$ with $r=2/3-\varepsilon$, hence the rate of
convergence for the first eigenvalue results to be $4/3$. On the other hand,
the third eigenfunction is analytic and the scheme provides the correct
$O(h^{2k})$ rate of convergence.
\begin{figure}
\begin{center}
\includegraphics[width=\textwidth]{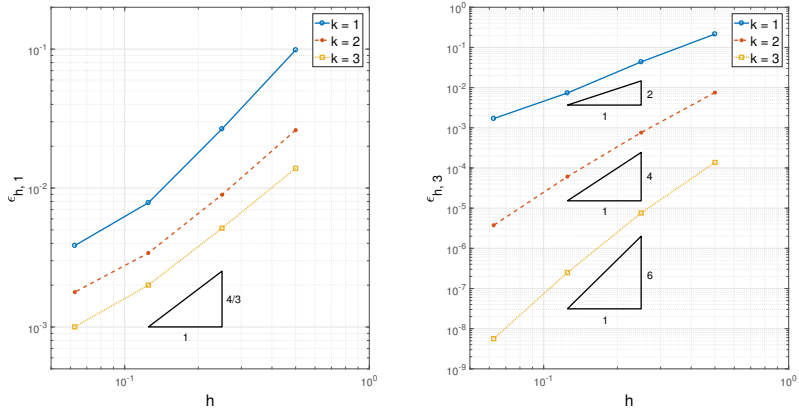}
\end{center}
\caption{Rate of convergence for the first eigenvalue (left) and the
third one (right) on the L-shaped domain}
\label{fg:L-shaped}
\end{figure}

\section{Extension to nonconforming and \emph{hp} version of VEM} 
\label{se:nonconf-hp}
This section is devoted to the presentation of the nonconforming, $p$ and $hp$
versions of the discrete virtual space and reports the main results relative
to the approximation of the eigenvalue problem. In particular, we are going to
report on the results of~\cite{GMV,CGMMV}.

\subsection{Nonconforming VEM}
\label{se:nc}
We start by treating the nonconforming case following~\cite{GMV}. We introduce
the broken Sobolev space 
$$
H^s_h=\left\{v\in L^2(\Omega): v|_P\in H^s(P)\ \forall P\in\T\right\}
$$ 
and the corresponding broken norm
$$
\vert\vert v\vert\vert_{s,h}=\left (\sum_{P\in\T}\vert\vert v \vert\vert_{s,P}^2\right )^{\frac{1}{2}}.
$$
With the aim of constructing the nonconforming space, we need some definitions. Let $P^+$ and $P^-$ be two elements sharing the edge $e$ and 
$\mathbf{n}_e^\pm$ the outward unit normal vector to $\partial P^\pm$. The jump of $v$ across any internal edge $e$ is 
$\jump{v}=v^+\mathbf{n}_e^+ + v^-\mathbf{n}_e^-$, where $v^\pm$ is the 
restriction of $v$ to $P^\pm$. If $e\in\edges^\partial$ we set 
$\jump{v}=v_e\mathbf{n}_e$.

The construction of the nonconforming virtual element space is based on a different definition of the local space $\VhP$. More precisely, instead of imposing that the  functions have a continuous trace along the element boundary, we prescribe the normal derivative on each edge of $P$ to be a polynomial of degree $k-1$, that is $\VhP$ is given by~\eqref{eq:vemP} and~\eqref{eq:pinabla} with 
\begin{equation}
\widetilde{V}_h^k(P) = 
\left\{ v_h\in H^1(P):\ \frac{\partial v_h}{\partial\mathbf n}\in\mathbb{P}_{k-1}(P)\ 
\forall e\subset\partial P,\ \Delta v_h\in\mathbb{P}_k(P)\right\}. 
\end{equation}
The degrees of freedom are given by the moments of $v_h$ of order up to $k-1$ on each edge $e$ of $P$ and the moments up to order $k-2$ on $P$. As in the conforming case, 
$\Pinabla v_h$ can be exactly evaluated by means of these degrees of freedom.  

The \emph{global nonconforming virtual element space} is then defined for, 
$k\ge 1$, as 
\begin{equation}
\label{eq:nonconf}
\Vhnc=\left\{ v_h\in\Hnc: v_h|_P\in\VhP\ \forall P\in\T \right\},
\end{equation}
where as in the finite element framework, 
\begin{equation} 
\label{eq:H1nc}
\Hnc=\left\{ v\in H^1_h: \int_e \jump{v}\cdot \mathbf{n}_e p\,ds=0\ 
\forall p\in \mathbb{P}_{k-1}(e),\ \forall e\in\edges\right\},
\end{equation}
being $\mathbf{n}_e$ the outward normal unit vector to $e$ with orientation fixed once and for all. It is clear that $\Vhnc$ is not a subspace of $\Hone$, this will imply that in the error analysis we have to take into account a \emph{consistency error}.

Proposition~\ref{prop:approx} still holds true with $v_I$ being the interpolant of $v$ constructed using the above degrees of freedom, see for example~\cite{CGPS}.

The discrete eigenvalue problem reads:
find $(\lambda_h,u_h)\in\mathbb{R}\times\Vhnc$ with $u_h\neq 0$ such that
\begin{equation}
\label{eq:eig-nc}
a_h(u_h,v_h)=\lambda_h b_h(u_h,v_h)\quad\forall v_h\in\Vhnc, 
\end{equation}
where the bilinear forms $a_h$ and $b_h$ are defined in~\eqref{eq:ah} and in~\eqref{eq:bh1}. We recall also the discrete source problem: 
given $f\in L^2(\Omega)$ find $u_h\in\Vhnc$ such that
\begin{equation}
\label{eq:sourcenc}
a_h(u_h,v_h)=b_h(f,v_h)\quad\forall v_h\in\Vhnc. 
\end{equation}
We define, as usual, the resolvent operator 
$T_h: L^2(\Omega)\to L^2(\Omega)$ with $T_hf=u_h\in\Vhnc$ solution
to~\eqref{eq:sourcenc} . We notice that in this case $T_h$ cannot be defined
with range in $\Hone$, therefore we shall obtain a result similar to
Theorem~\ref{th:delta1} with the gap measured in the broken $H^1$-norm. 

The analysis of the nonconforming case can be carried on with the same
arguments used in the conforming case, stemming from the broken $H^1$-norm
estimate for the error of the solution of the source problem. We start with the
following result proved in~\cite{ALM}: there exists a positive constant $C$
independent of $h$ such that 
\begin{equation}
\label{eq:errornc}
\aligned
\vert\vert u-u_h \vert\vert_{1,h} & \le C\Big ( \vert u-u_I\vert_{1,h} + \vert u - u_{\pi}\vert_{1,h}  \\
& + \sup_{v_h\in \Vhnc}\frac{\vert b(f,v_h)-b_h(f,v_h)\vert}{\vert v_h\vert_{1,h}}
+  \sup_{v_h\in\Vhnc}\frac{\nc(u,v_h)}{|v_h|_{1,h}}
\Big ),
\endaligned
\end{equation}
where 
\begin{equation}
\label{eq:nc}
\nc(u,v_h)=\sum_{P\in\T}a^P(u,v_h)-b(f,v_h)=
\sum_{e\in\edges^0}\int_e \nabla u\cdot\jump{v_h}\,ds.
\end{equation}
The term $\nc(u,v_h)$ represents an additional \emph{consistency error} which,
in the spirit of the \emph{second Strang Lemma}~\cite[Th. 4.2.2]{ciarlet}, is
due to the fact that the discrete functions are not continuous across
inter-element edges. With standard arguments using the degrees of freedom of
the nonconforming virtual elements on the element boundary and a Poincar\'e
inequality (see~\cite{brenner}), one can prove that for
$u\in H^{1+r}(\Omega)$, with $r>1/2$,
\begin{equation}
\label{eq:stimaR}
\nc(u,v_h)\le C h^r ||u||_{1+r}|v_h|_{1,h}\quad\forall v_h\in\Vhnc. 
\end{equation}

In order to obtain the uniform convergence of $T_h$ to $T$, we first show 
the $L^2$-error estimate for the solution to the source problem. 
\begin{theorem}
\label{th:L2nc}
Let $f\in L^2(\Omega)$, $u\in\Hone$ and $u_h\in\Vhnc$ be the solutions 
to~\eqref{eq:source} and~\eqref{eq:sourcenc}, respectively. Then there exists a constant $C$, independent of $h$, such that for $t=\min(1,r)$
\begin{equation}
\label{eq:L2nc}
\aligned
||u-u_h||_0 & \le C h^t \Big ( |u-u_h|_{1,h} + |u-u_{\pi}|_{1,h} \\
& \qquad + ||f-\Pio f||_0 + \sup_{v_h\in\Vhnc}\frac{\nc(u,v_h)}{|v_h|_{1,h}}
\Big).
\endaligned
\end{equation} 
\end{theorem}
\begin{proof}
The proof follows the same lines as that of Theorem~\ref{th:L2-error}. We
have to take into account the consistency error. Let $\psi\in\Hone$ be the
solution to~\eqref{eq:duality}, then we have
\[
\|u-u_h\|^2_0=\sum_P\left(a^P(u-u_h,\psi-\psi_I)+a^P(u-u_h,\psi_I)\right)
+\nc(\psi,u_h).
\]
Thanks to~\eqref{eq:28} and~\eqref{eq:stimaR}, it remains to estimate the
second term in the sum, which can be split as follows:
\[
\aligned
& \sum_P a^P(u,\psi_I)=b(f,\psi_I)+\nc(u,\psi_I)\\
& \sum_P a^P(u_h,\psi_I)=b_h(f,\psi_I)
+\sum_P\left(a^P(u_h,\psi_I)-a_h^P(u_h,\psi_I)\right).
\endaligned
\]
We use again~\eqref{eq:stimaR} and observe that the other terms can be
estimated as the terms III and IV in the proof of Theorem~\ref{th:L2-error}. 
Putting together all the estimates yields the required bound.
\end{proof}

Assuming that $f\in H^{1+s_1}(\Omega)$ and that the solution $u$ to
problem~\eqref{eq:source} belongs to $H^{1+s_2}(\Omega)$ for some $s_1>0$ and
$s_2 >\frac{1}{2}$, then we obtain again the bound~\eqref{eq:stima-L2-detail}.
This optimal rate of convergence yields the rate convergence for the gap
between the eigenspaces and for the eigenvalue error as in
Theorems~\ref{th:stima-gap} and~\ref{th:eig-rate}.  
The analogous result as the one in Theorem~\ref{th:delta1} can be obtained 
if the gap is measured using the broken $H^1$-norm. 

Numerical results confirming the theory are presented in~\cite{GMV} where the
domain is a square or L-shaped.
The nonconforming VEM has been tested on four different types of meshes
including also non convex elements providing always optimal rate of
convergence in the case of analytic eigenfunctions. In particular,
in~\cite{GMV} a comparison of the rate of convergence between conforming and
nonconforming approximation has been presented showing very close behavior.
The benchmark with the L-shaped domain has been also considered and again the
results agree with those presented in Figure~\ref{fg:L-shaped}.

\subsection{$hp$ version of VEM} 
\label{se:hp}
Let us now recall briefly the $p$ and $hp$ version of the virtual
element method and show how these methods can be applied to the discretization
of the eigenvalue problem. We refer, in particular, to~\cite{CGMMV} although we
are considering the simpler model problem~\eqref{eq:Laplace}, and
to~\cite{BCMR} for the basic principles of $hp$ virtual elements.

In this case the spaces we are considering, are labeled with a subindex
$n\in\mathbb{N}$ and refer to conforming polygonal decompositions of 
$\Omega$ denoted by $\Tn$.  Given an element $P\in\Tn$ and the polynomial
degree $p\in\mathbb{N}$, $V_n^p(P)$ is the local virtual space and coincides
with the space $\VhP$, with the due changes in the notation.
Hence the global space $V_n^p$ is the subspace of $\Hone$ defined as 
$V_n^p=\{v_n\in\Hone :\ v_n|_P\in V_n^p(P)\ \forall P\in\Tn\}$. Analogously, 
we shall use the subindex $n$ instead of $h$ in all the definitions involving
discrete quantities.

The results of approximation in the space $V_n^p$ are similar to those in Proposition~\ref{prop:approx}, but trace the dependence on the degree $p$, so that the exponential convergence of the error can be deduced. 
\begin{proposition}
\label{prop:approxp}
Given $P\in\Tn$ and $u\in H^{1+s}(P)$, $s>0$, for all $p\in\mathbb{N}$, there exists $u_\pi\in\mathbb{P}_p(P)$ and a constant $C$ independent of $p$, such that 
\begin{equation*}
|u-u_\pi|_{\ell,P}\le C \frac{h_P^{1+\min(p,s)-\ell}}{p^{1+s-\ell}} ||u||_{1+s,P}\quad\text{for } 0\le\ell\le s.
\end{equation*}
Moreover, given $u\in\Hone$ with $u|_P\in H^{1+s}(P)$ for all $P\in\Tn$ and for some $s\ge 1$, there exists $u_I\in V_n^p$ such that 
$$
|u-u_I|_1\le C \frac{h^{\min(p,s)}}{p^{s-1}}
\left ( \sum_{P\in\Tn} ||u||^2_{1+s,P}\right)^\frac{1}{2}.
$$ 
\end{proposition}
We point out that the use of high degree $p$ is convenient in the approximation of smooth (piecewise smooth) functions. 

The approximation of the eigenvalue problem~\eqref{eq:Laplace} can be written
introducing the local discrete bilinear forms already used in the
$h$-conforming case~\eqref{eq:ah} and~\eqref{eq:bh1}. Then the discrete
eigenvalue problem reads:
find $(\lambda_n,u_n)\in\mathbb{R}\times V_n^p$ with $u_n\neq 0$ such that
\begin{equation}
\label{eq:eig-discretep}
a_n(u_n,v_n)=\lambda_n b_n(u_n,v_n)\quad\forall v_n\in V_n^p. 
\end{equation}
As it is standard for eigenvalue problems, we associate the discrete source
problem: given $f\in L^2(\Omega)$ find $u_n\in V_n^p$ such that
\begin{equation}
\label{eq:source-discretep}
a_n(u_n,v_n)=b_n(f,v_n)\quad\forall v_n\in V_n^p. 
\end{equation}
In order to take advantage of the approximation properties listed in
Proposition~\ref{prop:approxp}, we define the resolvent operator $T$ on 
$\Hu$ in the spirit of Remark~\ref{rm:TV}. Hence for $f\in\Hu$, $Tf=u$ is the
solution to the continuous source problem~\eqref{eq:source}, while in the
discrete case, $T_nf=u_n\in V_n^p$ stands for the solution
to~\eqref{eq:source-discretep}. 

The convergence of eigensolutions derives from the uniform convergence of 
$T_n$ to $T$, which in turn is a consequence of the error estimates for the
source problem. 
\begin{theorem}
\label{th:errorep}
Given $f\in\Hu$, let $u\in\Hone$ and $u_n\in V_n^p$ be the solutions
to~\eqref{eq:source} and~\eqref{eq:source-discretep}, respectively. Assume that
the restrictions of $f$ and $u$ to every element $P\in\Tn$ belong to 
$H^{1+s}(P)$ for some $s\ge 0$, then there exists a constant $C$ independent of
$h$ and $p$ such that 
$$
|u-u_n|_1\le C \frac{h^{\min(p,s)}}{p^{s-1}}
\left (h^2 ||f||_{1+s,n} + ||u||_{1+s,n} \right ).
$$
\end{theorem}
For the proof it is enough to combine Proposition~\ref{prop:approxp} with
standard VEM arguments. In~\cite{CGMMV}, a precise dependence of $C$ on the
ellipticity constant of the form $a$ and the constants in~\eqref{eq:staba},
and~\eqref{eq:stabb} is obtained. In particular, 
it is shown that the constants in~\eqref{eq:staba} and~\eqref{eq:stabb}
might depend on $p$, as we shall see later. 

The result in Theorem~\ref{th:errorep}, together with the arguments
of~\cite[Section 5]{BCMR}, yields the following exponential convergence:
\begin{theorem}
\label{th:exp}
Let $u$ and $u_n$ be as in Theorem~\ref{th:errorep}, and assume that 
$u$ is the restriction on $\Omega$ of an analytic function defined on an extension of $\Omega$. Then there exist two positive constants $C$ and  $c$ independent of the discretization parameters such that
$$
|u-u_n|_1\le C \mathrm{exp}(-cp).
$$
\end{theorem}
Theorem~\ref{th:errorep} implies the uniform convergence of
$||T-T_n||_{\mathcal{L}(\Hu,\Hu)}$ to zero, more precisely one can show that
under the same assumptions as in Theorem~\ref{th:exp} it holds true: 
$$
||T-T_n||_{\mathcal{L}(\Hu,\Hu)}\le C\mathrm{exp}(-cp).
$$
Then, applying Theorems~\ref{th:eig-rate} and~\ref{th:delta1}, we obtain 
the spectral convergence both for the error in the approximation of the
eigenvalues, and the gap between the continuous and the discrete eigenspaces,
provided the eigenfunctions are analytic. For the details, we refer
to~\cite{CGMMV}. Figure~\ref{fg:p} shows the convergence of the $p$-version of
VEM in comparison with the $h$-version with fixed $p=1,2,3$ for the test case on
the square $(0,\pi)^2$ with exact values given by~\eqref{eq:square-eig}. 
The domain is partitioned into a family of Voronoi meshes, and for the
$p$-version the coarsest one employed for the $h$ version has been used. 
These results has been obtained using the so called \emph{diagonal} recipe for
the stabilization of $a$ defined as 
\[
\widetilde{S}_a^P(\varphi_i,\varphi_j)=
\max(1,a^P(\Pinabla\varphi_i,\Pinabla\varphi_j))\delta_{ij}
\]
where $\varphi_i$ is the canonical basis in $V_n(P)$ and $\delta_{ij}$ is the
Kroenecker index, and the following choice for the stabilization of $b$
\[
S_b^P(u_n,v_n)=\frac{h_P}{p^2}\int_{\partial P} u_n v_n\,ds,
\]
which satisfies~\eqref{eq:stabb} with $\beta_*(p)\gtrsim p^{-6}$ and
$\beta^*(p)\lesssim 1$.

Although these values inserted in the estimates reported in
Theorem~\ref{th:errorep} might seem pessimistic, in reality the numerical
results seem to be not affected by the behavior of the stability constants.

\begin{figure}[h]
\centering
\includegraphics[width=0.48\textwidth]{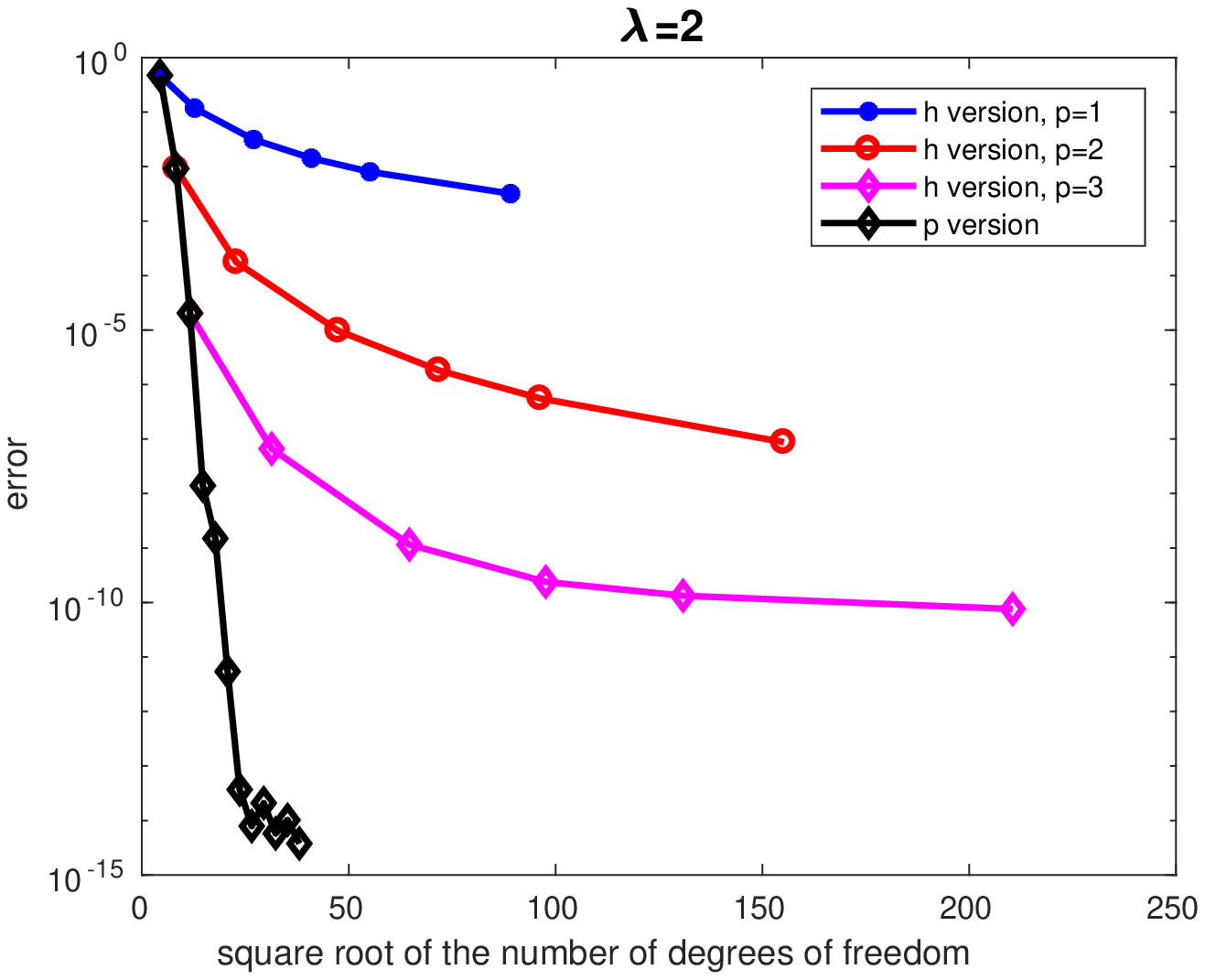}
\includegraphics[width=0.48\textwidth]{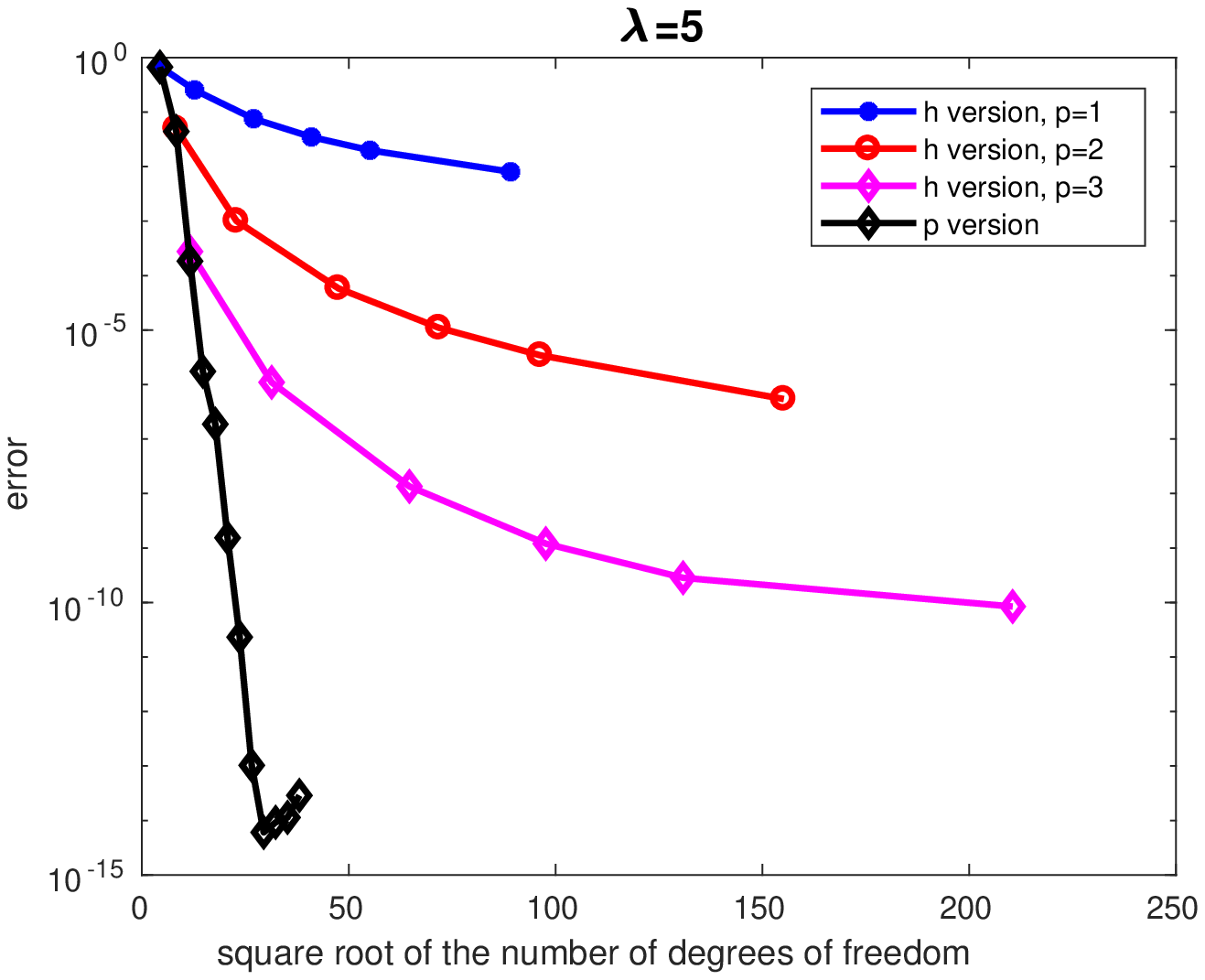}
\includegraphics[width=0.48\textwidth]{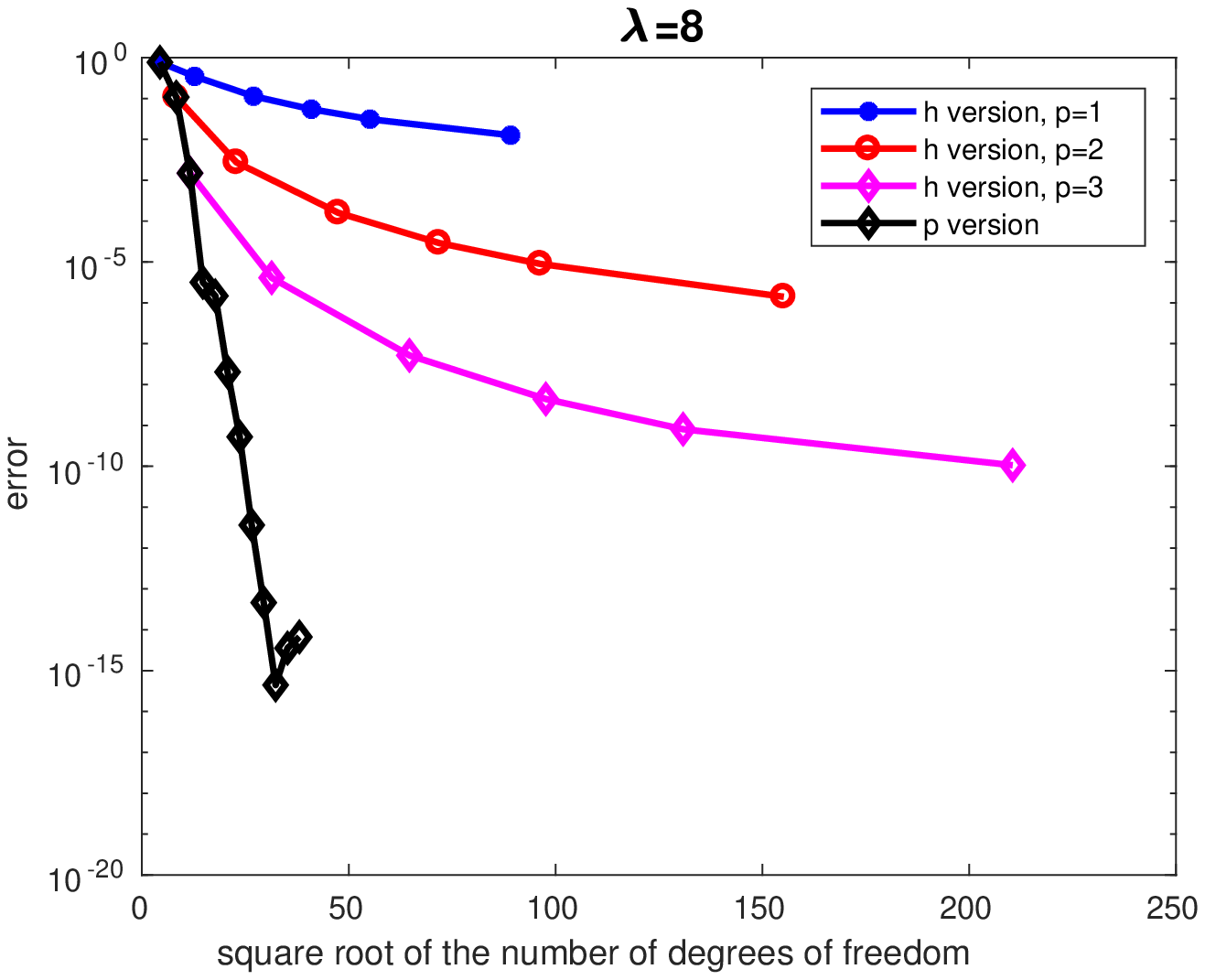}
\includegraphics[width=0.48\textwidth]{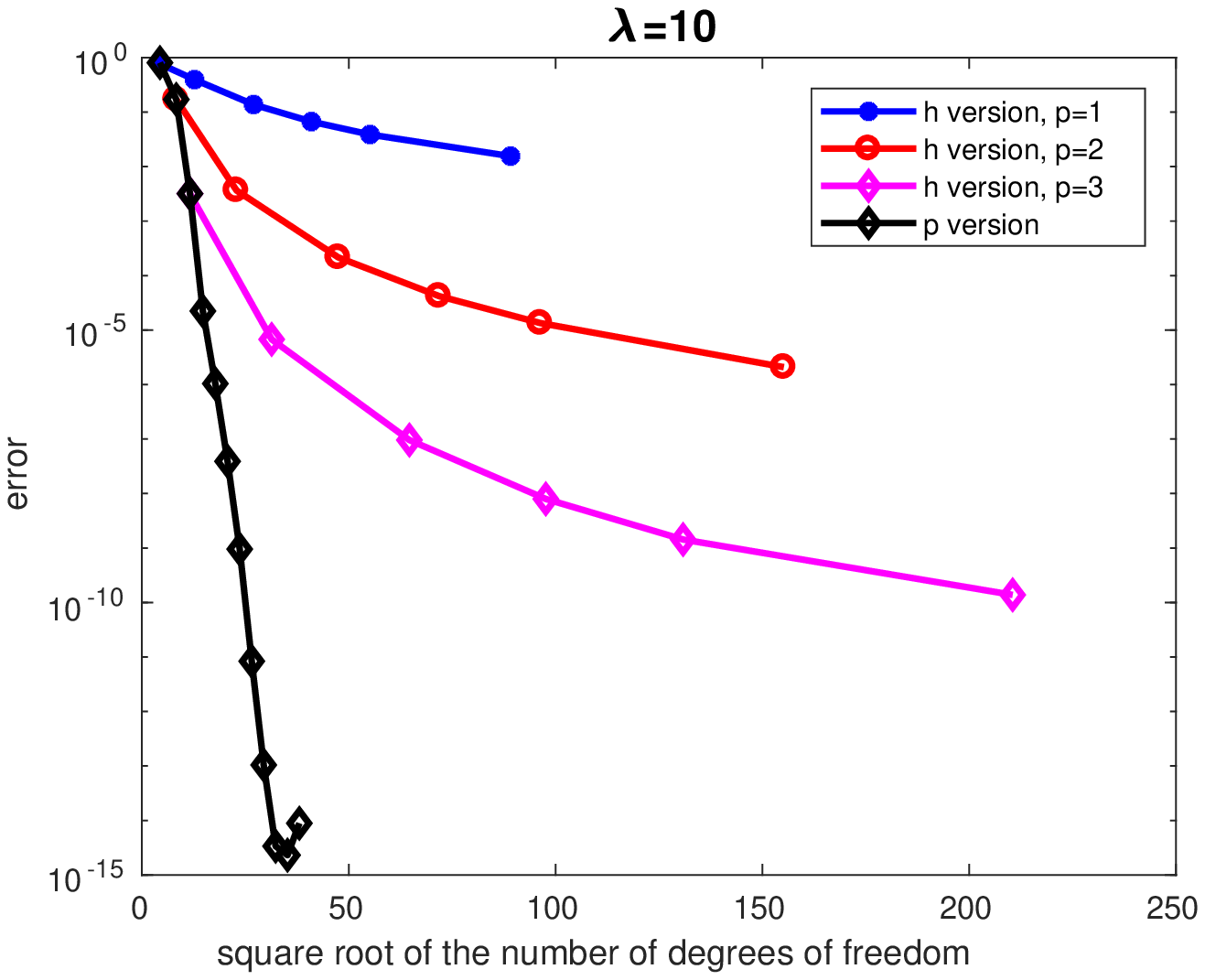}
\caption{Convergence of the error for the first four distinct Dirichlet
eigenvalues of the Laplace operator on the unit square domain.
On the $x$-axis, the square root of the number of degrees
of freedom is reported}
\label{fg:p}
\end{figure}
If the eigenfunctions are not smooth enough, the spectral convergence of $p$
version of VEM cannot be achieved. However, as in the finite element method,
one can resort to the $hp$ approach which combines the $p$ and the $h$ versions
using locally geometrical refined meshes where the solution is singualar.
For the source problem, the analysis
of the exponential convergence of the $hp$ VEM in presence of corner
singularities has been tackled in~\cite{BCMR}. The paper~\cite{CGMMV} merely
investigates numerically the convergence of the $hp$ approximation of the
eigenvalues, showing the exponential decay in terms of the cubic root of the
number of degrees of freedom, in the case of eigenfunctions with finite Sobolev
regularity.
We report in Figure~\ref{fg:hp} the convergence of the error for the first
four distinct
eigenvalues of the Laplace eigenproblem with Neumann boundary condition on the
L-shaped domain obtained in~\cite{CGMMV}.
It is well-known that the first eigenfunction
is singular. In order to recover the spectral convergence the
$hp$-version of VEM has been applied on a graded mesh with a non uniform
distribution of $p$. The domain has been divided into layers around the
reentrant corner and the local polynomial degree of $V_n(P)$ increases as the
layer is farer from the singularity. The results show spectral
convergence with respect to the coobic root of the number of degrees of freedom
for the $hp$-version.
\begin{figure}[h]
\centering
\includegraphics[width=0.48\textwidth]{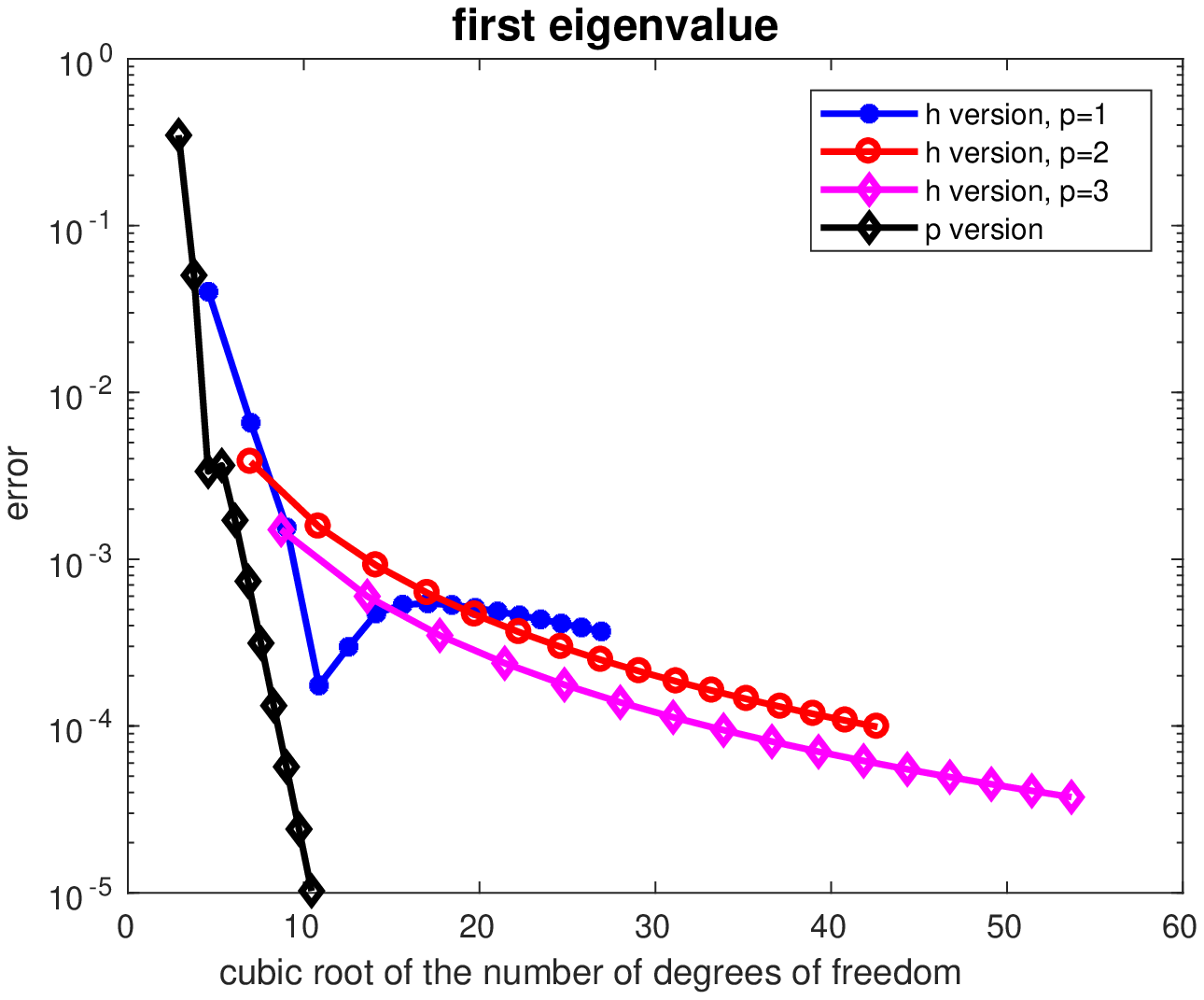}
\includegraphics[width=0.48\textwidth]{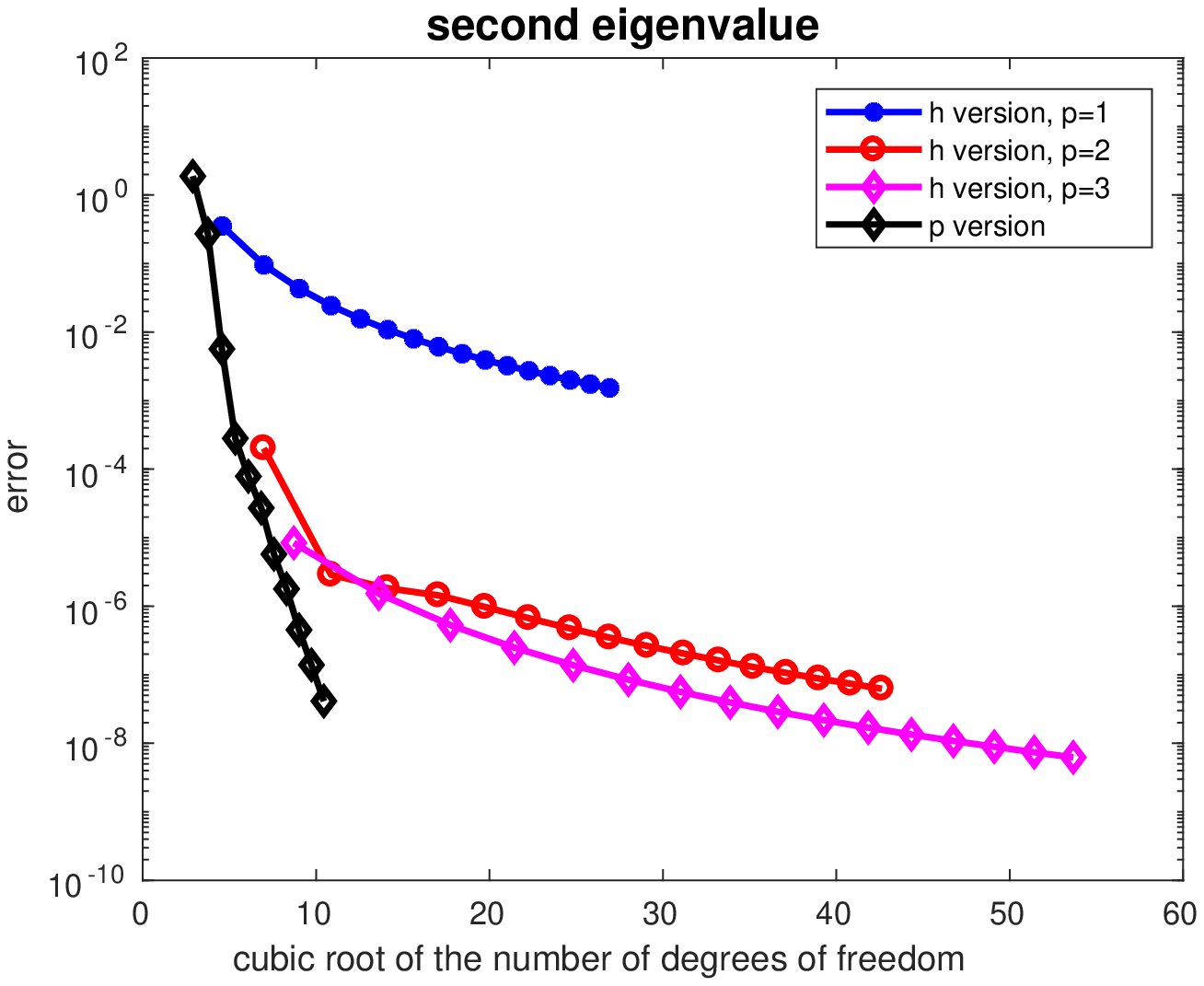}
\includegraphics[width=0.48\textwidth]{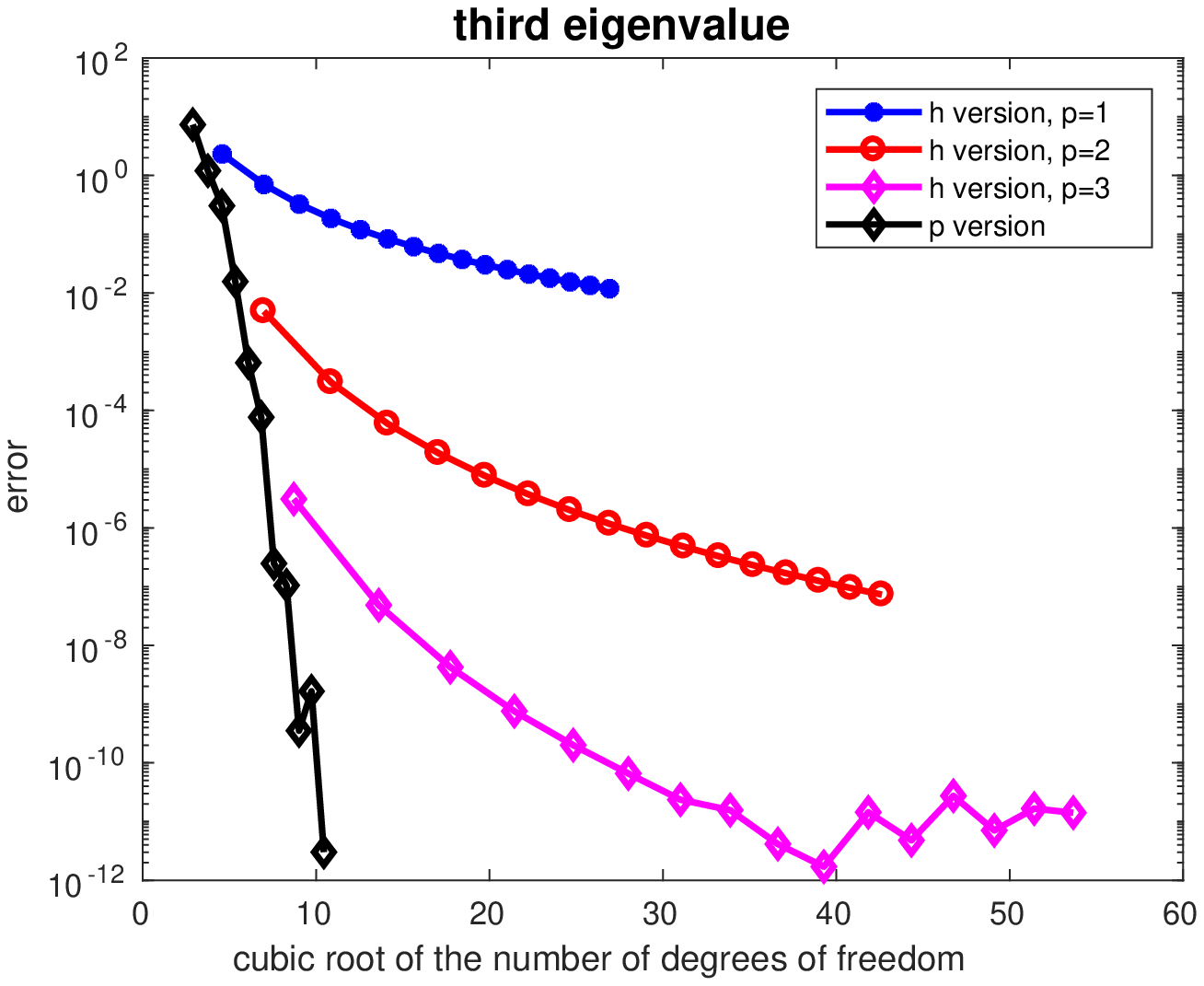}
\includegraphics[width=0.48\textwidth]{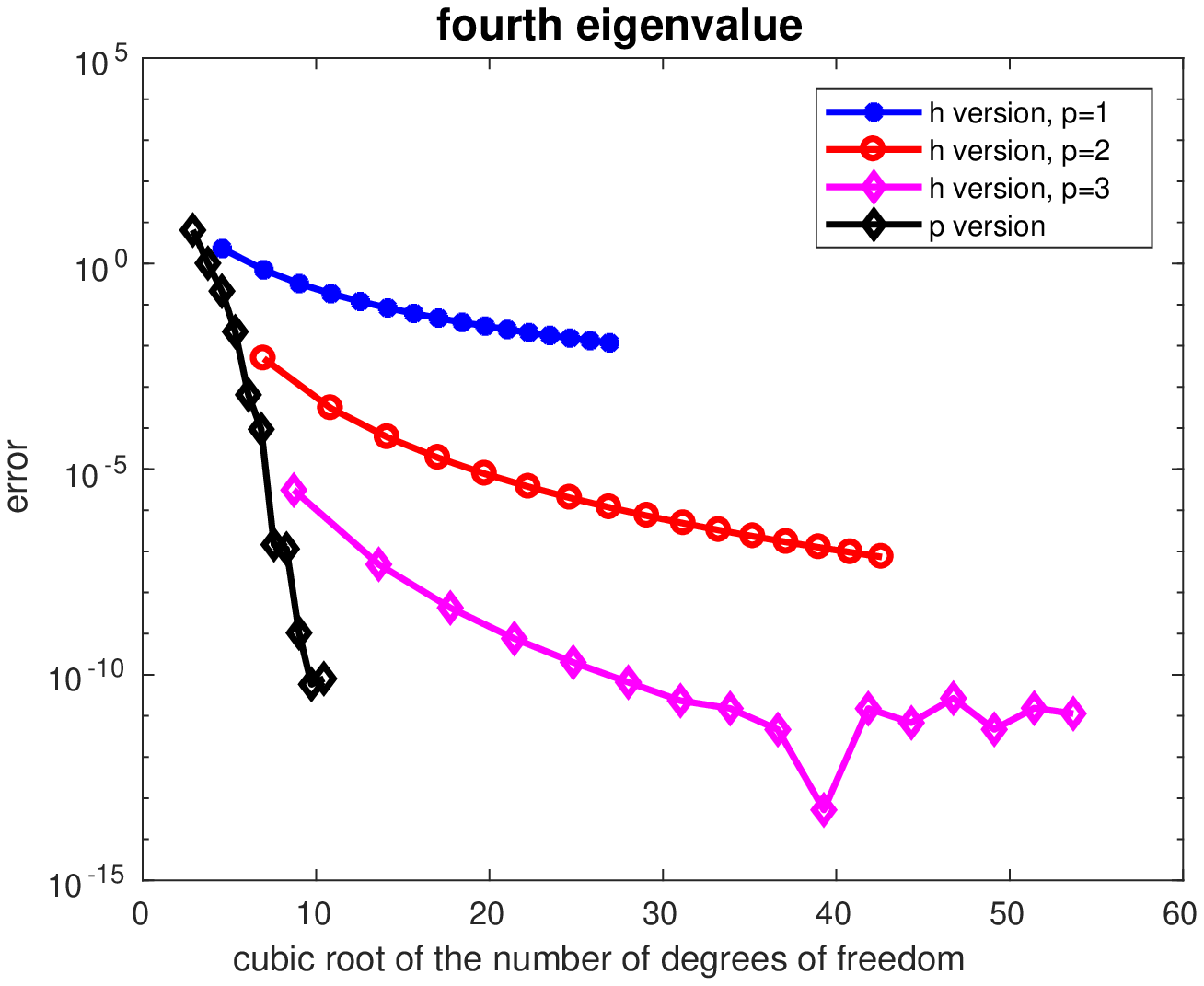}
\caption{Convergence of the error for the first four distinct Neumann
eigenvalues of the Laplace operator on the L-shaped domain. Comparison of the
$hp$-version with the $h$-version with $p=1,2,3$. On the $x$-axis, we report
the cubic root of the number of degrees of freedom}
\label{fg:hp}
\end{figure}
\section{The choice of the stabilization parameters}
\label{se:parameter}

In Section~\ref{se:theory} we have seen that an effective VEM approximation of
eigenvalue problems requires the introduction of appropriate stabilization
parameters. In practice, the discrete problem has the form of the following
algebraic generalized eigenvalue problem
\begin{equation}
\m{A}\m{u}=\lambda\m{M}\m{u}
\label{eq:geneig}
\end{equation}
where both the involved matrices $\m{A}$ and $\m{M}$ depend on the
stabilization  parameters introduced for the VEM discretization.

The results of the previous sections assume that such parameters have been
fixed and show the convergence of the method when the meshsize goes to zero
and/or $p$ goes to $\infty$.
By doing so, it is implicitly understood that the convergence behavior depends
on the choice of the parameters. In this section we want to discuss this
dependence and see how the discrete solution is influenced by such choice. We
would like to make it clear from the very beginning that the optimal strategy
for the choice of the parameters is still the object of ongoing research and
that the aim of our presentation is more related to the description of the
phenomenon than to an ultimate answer to the open questions.
Nevertheless, the \emph{easy and quick recipe} will be that the parameter
related to $\m{A}$ should be large enough, while the parameter related to
$\m{M}$ should be small enough and possibly equal to zero.

The results presented in this section are a condensed version of what is
published in~\cite{auto-stab}. The interested reader will find there a more
general theory and more specific numerical tests.

Although the general picture of VEM approximation of eigenvalue problems is
more complicated, the following simplified setting proves useful in order to
understand its main features.

\subsection{A simplified setting}

We assume that all matrices are symmetric and that
$\m{A}=\m{A_1}+\alpha\m{A_2}$ and $\m{M}=\m{M_1}+\beta\m{M_2}$,
where $\alpha\ge0$ and $\beta\ge0$ play the role of the stabilization
parameters. More precisely, we assume that

\begin{enumerate}[i)]

\item $\m{A_1}$ and $\m{M_1}$ are positive semidefinite;
\item $\m{A_2}$ and $\m{M_2}$ are positive semidefinite and positive definite
on the kernel of $\m{A_1}$ and $\m{M_1}$, respectively;
\item $\m{A_2}$ and $\m{M_2}$ vanish on the orthogonal complement of the
kernel of $\m{A_1}$ and $\m{M_1}$, respectively.

\end{enumerate}

In~\cite{auto-stab} we have studied several examples and described the
spectrum of the generalized eigenvalue problem~\eqref{eq:geneig} as a
function of $\alpha$ and $\beta$. The main features are easily understood with
a simple example when $\alpha$ or $\beta$ are constant.

When $\beta>0$ is constant, then $\m{M}$ is positive definite and the
eigenvalues of~\eqref{eq:geneig} are all positive, possibly vanishing for
$\alpha=0$ (see assumption i) above), and are formed by two families: the
first one is independent on $\alpha$ and corresponds to eigenvectors
$\m{v}$ orthogonal to the kernel of $\m{A}_1$ and to the eigenvalues $\lambda$
that satisfy
\[
\m{A_1}\m{v}=\lambda\m{M}\m{v}
\]
(see assumption iii) above), while the second family contains eigenvectors
$\m{w}$ in the kernel of $\m{A}_1$ and the eigenvalues $\lambda=\alpha\mu$
where $(\mu,\m{w})$ satisfies
\[
\m{A_2}\m{w}=\mu\m{M}\m{w}
\]
(see assumption ii) above). In particular the second family contains
eigenvalues growing linearly with $\alpha$ and includes, for $\alpha=0$, the
eigenvalue $\lambda=0$ with eigenspace equal to the kernel of $\m{A_1}$.

Figure~\ref{fg:case1} gives a graphical indication of the eigenvalues
of~\eqref{eq:geneig} as a function of $\alpha$ when $\beta>0$ is fixed: the
horizontal lines correspond to the eigenvalues of the first family, while the
linearly increasing curves starting at the origin correspond to the
eigenvalues of the second family.

\begin{figure}

\begin{center}
\includegraphics[width=.5\textwidth]{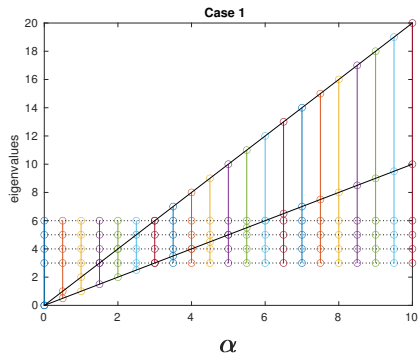}
\end{center}
\caption{Eigenvalues of~\eqref{eq:geneig} for $\beta>0$ fixed.}
\label{fg:case1}
\end{figure}

The second case that is useful to discuss is when $\alpha>0$ is fixed and
$\beta$ is varying. This case is analogous to the previous one with the roles
of $\m{A}$ and $\m{M}$ exchanged. This means that $\m{A}$ is positive
definite and the behavior of the solutions $(\omega,\m{u})$ to the following
eigenvalue problem
\[
\m{M}\m{u}=\omega\m{A}\m{u}
\]
is analogous to the one depicted in Figure~\ref{fg:case1} with $\alpha$
replaced by $\beta$. Since the original eigenmodes of
problem~\eqref{eq:geneig} are given by $(\lambda,\m{u})=(1/\omega,\m{u})$ for
$\omega\ne0$ and $(\infty,\m{u})$ for $\omega=0$, it turns out that the
eigenvalues in this case are split again into two families: the first one is
independent of $\beta$ and corresponds to eigenvectors $\m{v}$ orthogonal to
the kernel of $\m{M}_1$ and to the eigenvalues $\lambda=1/\omega$ that satisfy
\[
\m{M_1}\m{v}=\omega\m{A}\m{v}
\]
while the second family contains eigenvectors $\m{w}$ in the kernel of
$\m{M}_1$ and eigenvalues $\lambda=1/(\beta\omega)$ where $(\omega,\m{w})$
satisfies
\[
\m{M_2}\m{w}=\omega\m{A}\m{w}.
\]
In particular, the second family contains eigenvalues that behave like
decreasing hyperbolas tending to zero as $\beta$ goes to infinity.
For $\beta=0$ the eigenvalues of the second family degenerate to
$\lambda=\infty$ with eigenspace equal to the kernel of $\m{M}_1$.

Figure~\ref{fg:case2} shows the behavior of the eigensolution
of~\eqref{eq:geneig} for fixed $\alpha>0$ and varying $\beta$: the horizontal
lines correspond to the eigenvalues of the second family while the hyperbolas
represent those of the second one.

\begin{figure}

\begin{center}
\includegraphics[width=.5\textwidth]{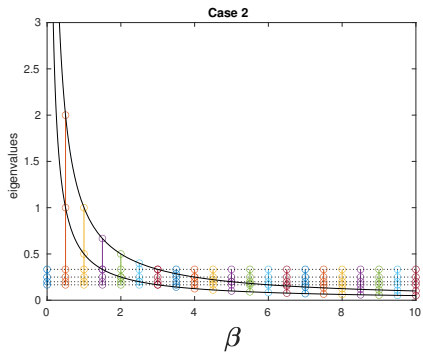}
\end{center}
\caption{Eigenvalues of~\eqref{eq:geneig} for $\alpha>0$ fixed.}
\label{fg:case2}
\end{figure}

\subsection{The role of the VEM stabilization parameters}

The general setting of the problems we have discussed in
Section~\ref{se:theory} does not fit exactly the simplified framework
presented in the previous subsection. In particular, the matrices $\m{A}$ and
$\m{M}$ do not satisfy the assumptions made above. Nevertheless, a similar
splitting into a matrix that is singular due to the presence of the
projection operator and a stabilizing matrix that takes care of the kernel of
the singular part is present both in $\m{A}$ and $\m{M}$. Actually, the
numerical experiments reveal a dependence on the parameters that resembles
pretty much the one presented in Figures~\ref{fg:case1} and~\ref{fg:case2}.

Redirecting the interested reader to~\cite{auto-stab} for more details, we
report here only some results related to the properties of the matrices
$\m{A}$ and $\m{M}$, together with some numerical experiments corresponding to
the simplified situations of Figures~\ref{fg:case1} and~\ref{fg:case2}. All
the presented results are taken from~\cite{auto-stab}.

We consider $\Omega$ as the square of side $\pi$ and introduce a sequence of
five  Voronoi meshes of polygons with increasing number of elements from
$N=50$ to $N=800$.  Table~\ref{tb:kernel} reports the dimension of the kernels
of the two matrices $\m{A_1}$ and $\m{M_1}$ as a function of the meshsize and
of the degree of approximation $k$. As it is known, both matrices are non
singular in the lowest order case for $k=1$; the matrix $\m{M_2}$ is non
singular for $k=2$ as well for the considered meshes.

\begin{table}
\centering
\begin{tabular}{c|c|c|c|c|c}\toprule
$k$ & $N=50$ & $N=100$ &$N=200$ &$N=400$ &$N=800$\\
\midrule
\multicolumn{6}{c}{Kernel of $\m{A}_1$}\\
\midrule
1 & 0 & 0 & 0 & 0 & 0\\
2 & 3 & 30 & 99 & 258 & 565\\
3 & 27 & 94 & 246 & 588 & 1312\\
\midrule
\multicolumn{6}{c}{Kernel of $\m{M}_1$}\\
\midrule
1 & 0 & 0 & 0 & 0 & 0\\
2 & 0 & 0 & 0 & 0 & 0\\
3 & 0 & 1 & 43 & 182 & 504\\
\bottomrule
\end{tabular} 
\caption{Dimension of the kernels of $\m{A}_1$ and $\m{M}_1$ with respect to
the degree $k$ and the number of elements $N$ in the mesh}
\label{tb:kernel}
\end{table}

For completeness, we also computed the lowest eigenvalue of the generalized
problem
\[
\m{A}_1\m{u}=\lambda\m{M}_1\m{u}
\]
that is reported in Table~\ref{tb:infsup}.
\begin{table}
\centering
\begin{tabular}{c|c|c|c|c}\toprule
$N=50$ & $N=100$ &$N=200$ &$N=400$ &$N=800$\\
\midrule
1.92654e+00 & 1.74193e+00 & 1.06691e+00 & 6.81927e-01 & 5.54346e-01\\
\bottomrule
\end{tabular}
\caption{First eigenvalues of $\m{A}_1\m{x}=\lambda\m{B}_1\m{x}$ for different
meshes}
\label{tb:infsup}
\end{table}
The fact that the value of the eigenvalue is decreasing as the number of
elements increases, confirms the need for a stabilization.

In the rest of this section we show and comment some results of our
computations of the eigenvalues of the Dirichlet problem for the Laplace
operator in the case of the mesh $N=200$.

Figure~\ref{fg:6(h)} reports the computed eigenvalues in the range $[0,40]$,
for $\beta=1$ and $\alpha\in[0,10]$ with elements of degree $k=3$.
\begin{figure}

\begin{center}
\includegraphics[width=.5\textwidth]{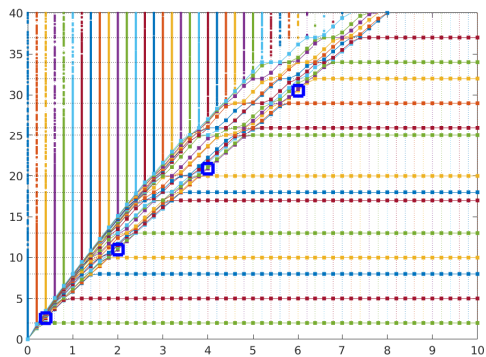}
\end{center}
\caption{Eigenvalues computed for $\beta=1$ and $\alpha\in[0,10]$ with $k=3$}
\label{fg:6(h)}
\end{figure}
Comparing these results with
Figure~\ref{fg:case1} it is easy to recognize the horizontal lines
corresponding to the \emph{good} eigenvalues we are interested in. At the same
time it is apparent the presence of some oblique lines that correspond to
\emph{spurious} modes introduced by the stabilization of $\m{A}$.

\begin{figure}

\begin{center}
\includegraphics[width=.45\textwidth]{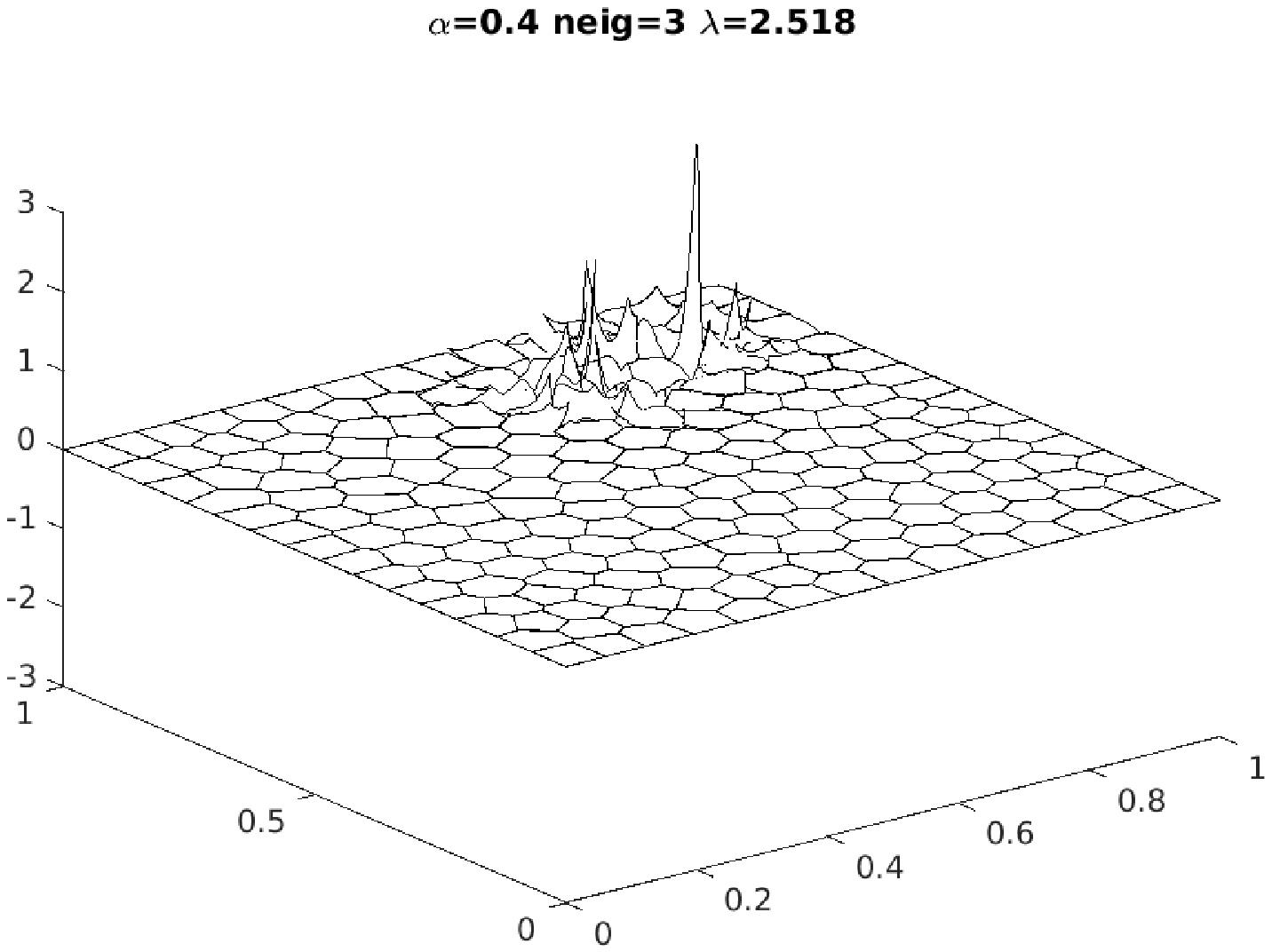}
\includegraphics[width=.45\textwidth]{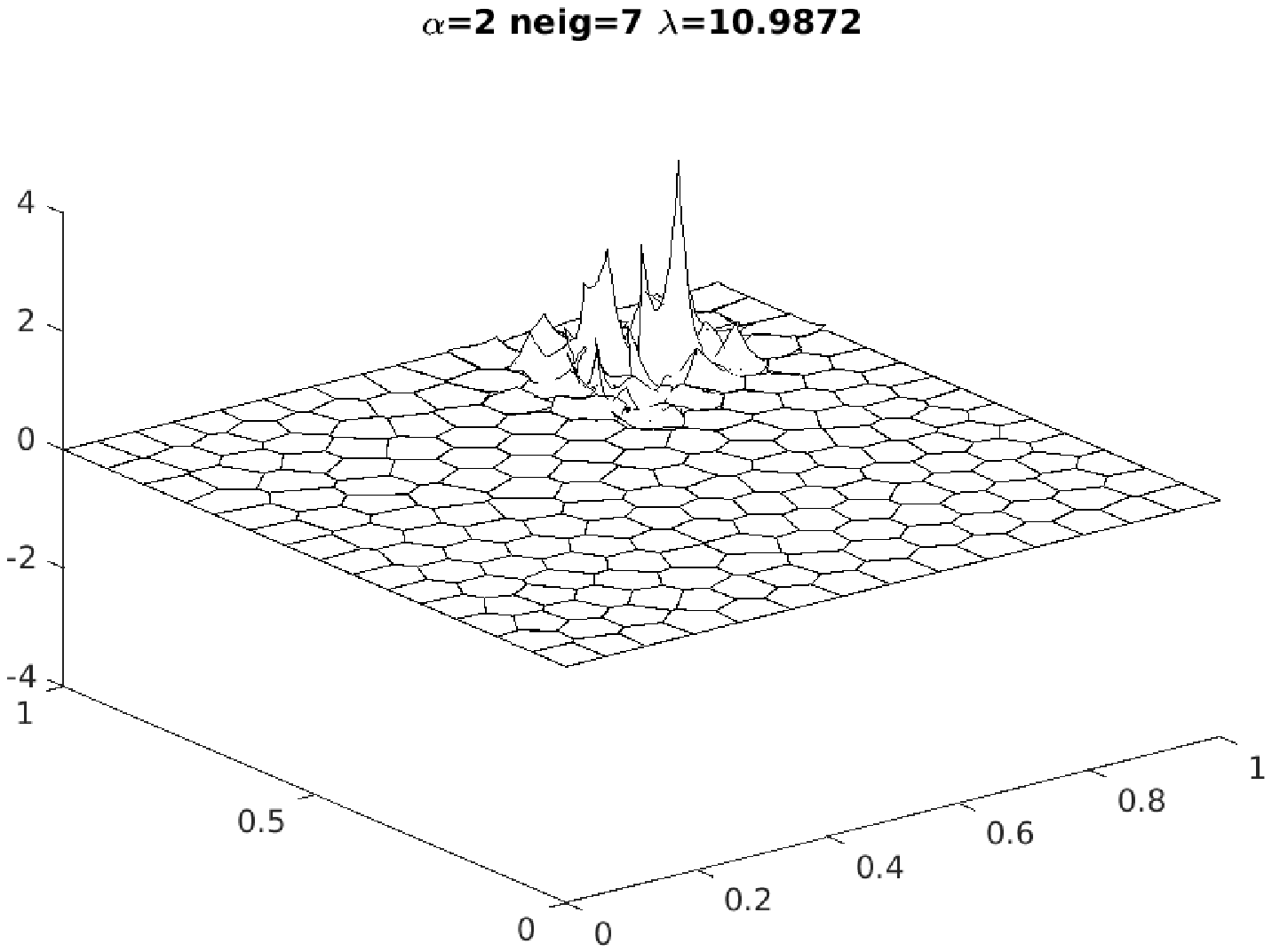}

\includegraphics[width=.45\textwidth]{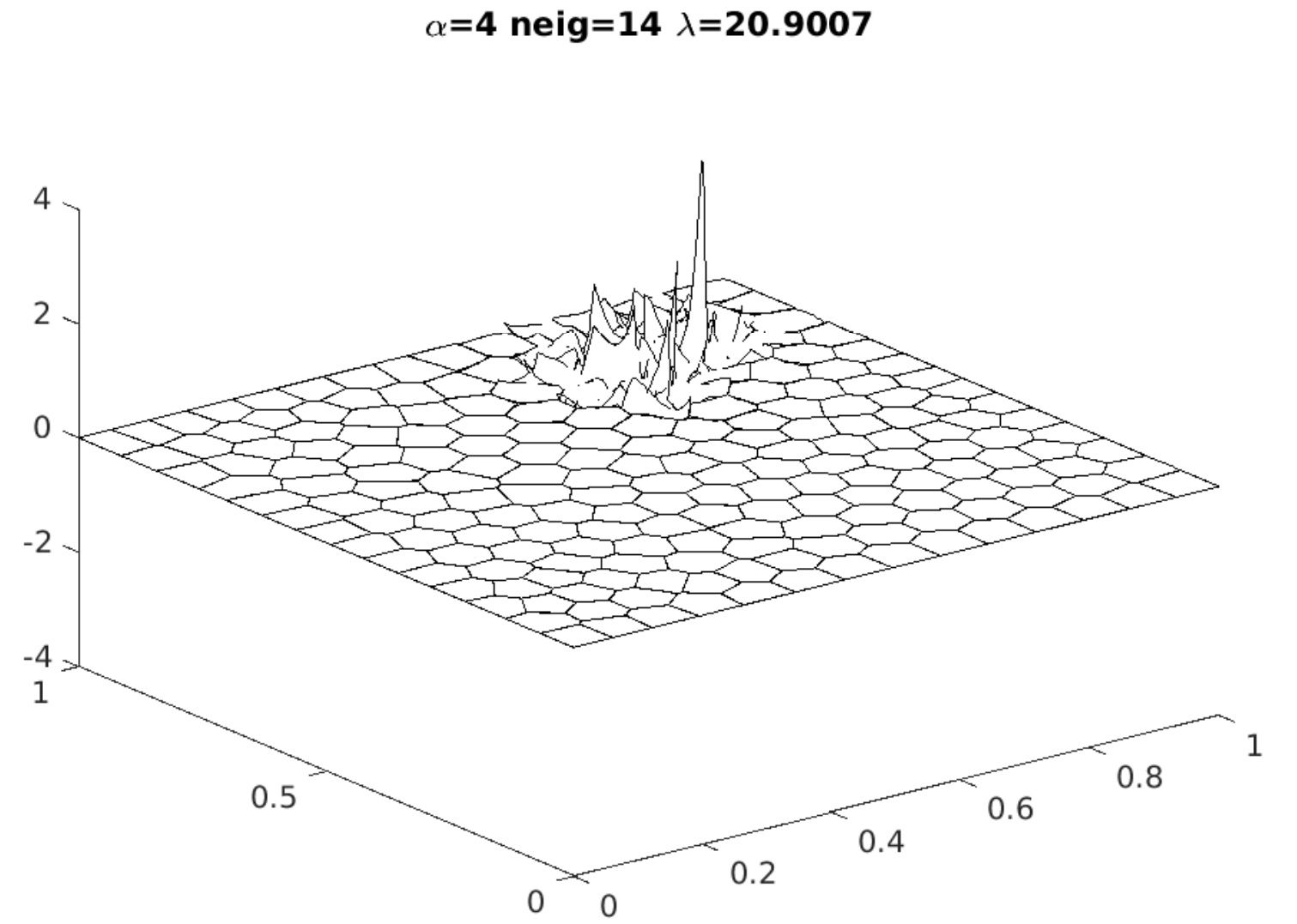}
\includegraphics[width=.45\textwidth]{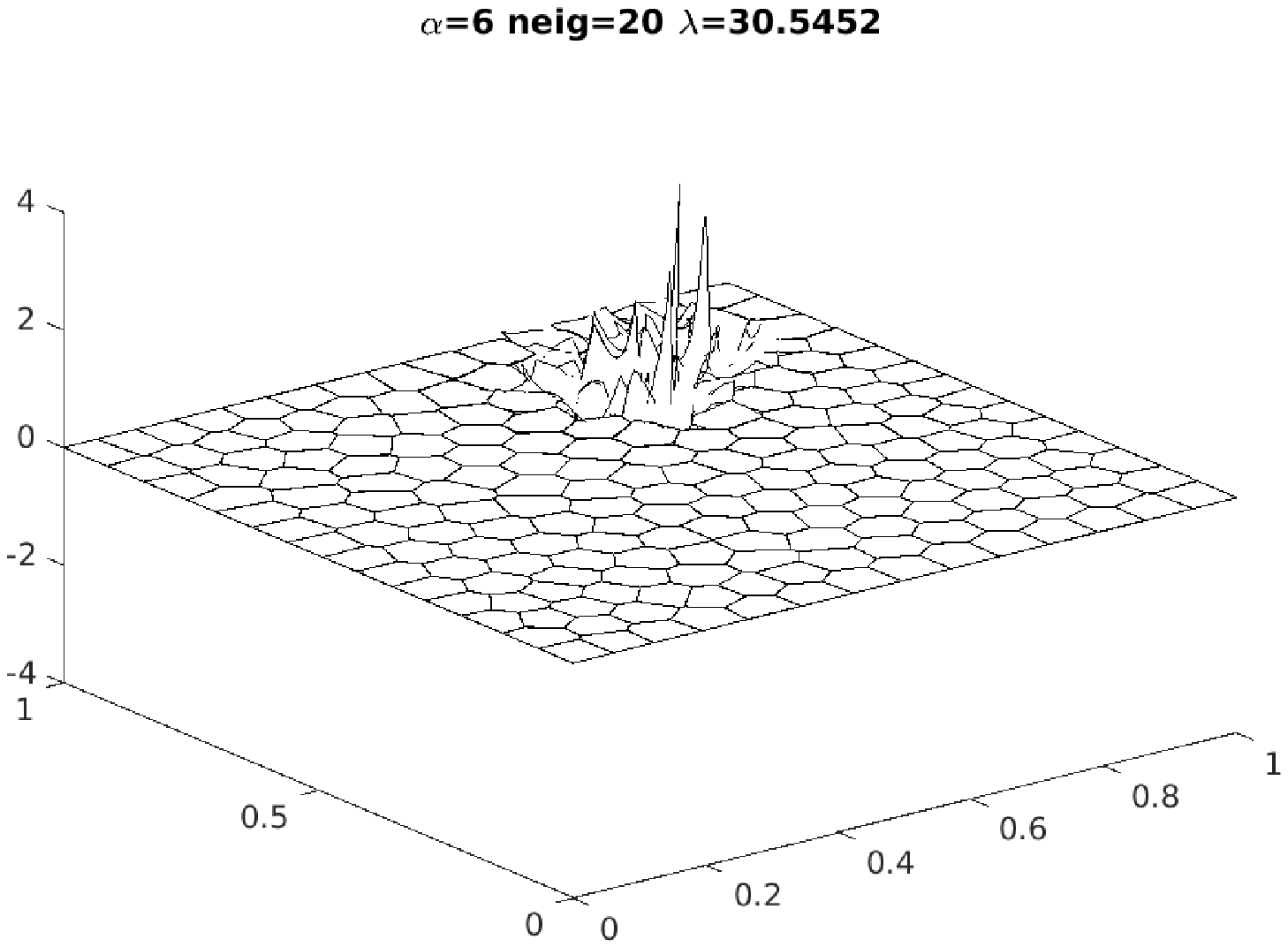}
\end{center}

\caption{Eigenfunctions associated with the marked eigenvalues of
Figure~\ref{fg:6(h)}}
\label{fg:marked}
\end{figure}

Figure~\ref{fg:marked} shows the eigenfunctions corresponding to the marked
values of Figure~\ref{fg:6(h)} that belong to the same oblique line. It can be
seen that the eigenfunctions look similar to each other, which confirms the
analogy with the situation presented in Figure~\ref{fg:case1}.

Several other cases when $\beta\ge0$ is fixed and $\alpha$ is varying are
reported in Figure~\ref{fg:alpha} for different values of $k$.

\begin{figure}

\begin{center}
\subfigure[\tiny{$k=1$, $\beta=0$, $\alpha\in[0,10]$}]
{\includegraphics[width=.32\textwidth]{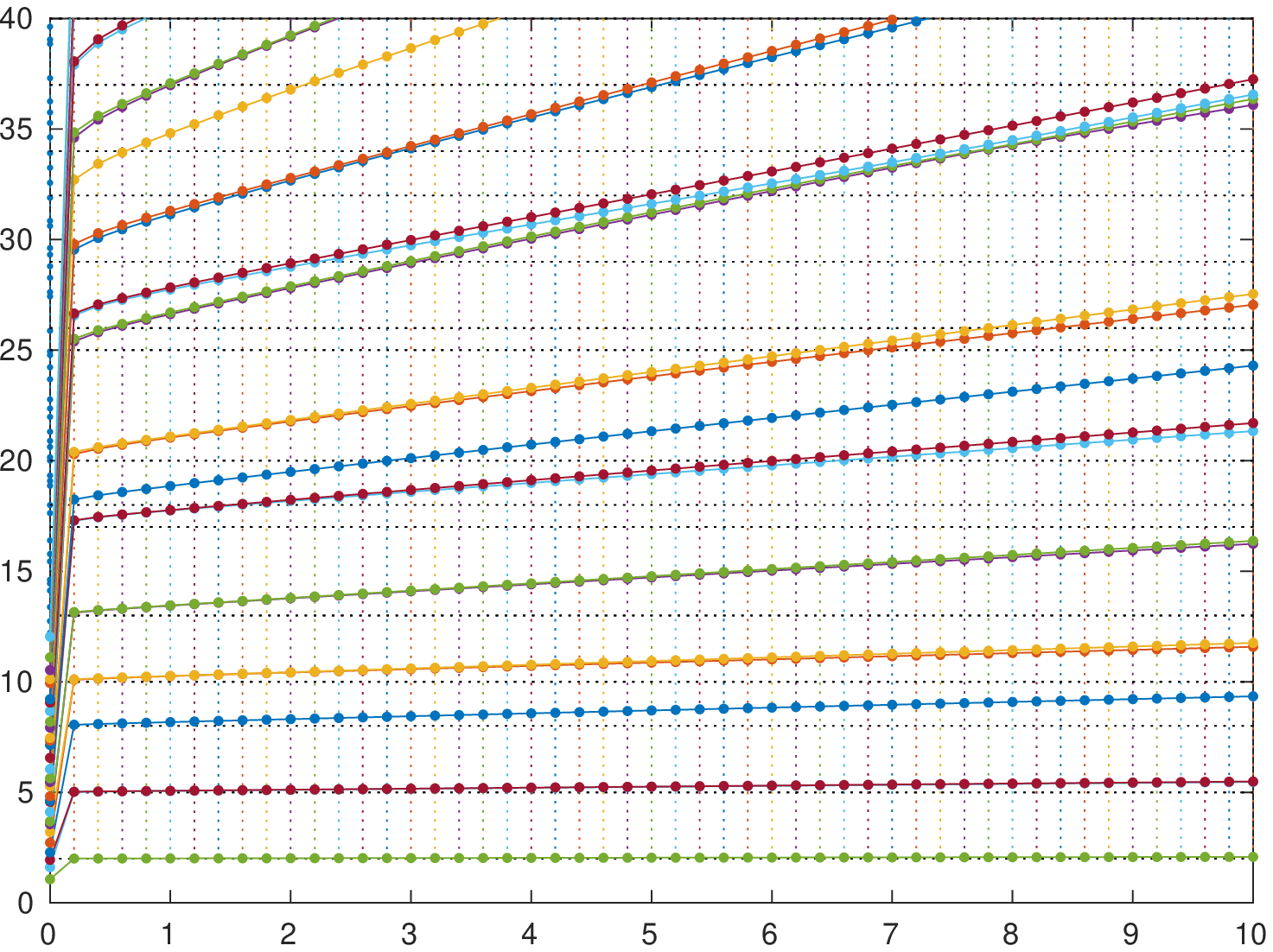}}
\subfigure[\tiny{$k=1$, $\beta=1$, $\alpha\in[0,10]$}]
{\includegraphics[width=.32\textwidth]{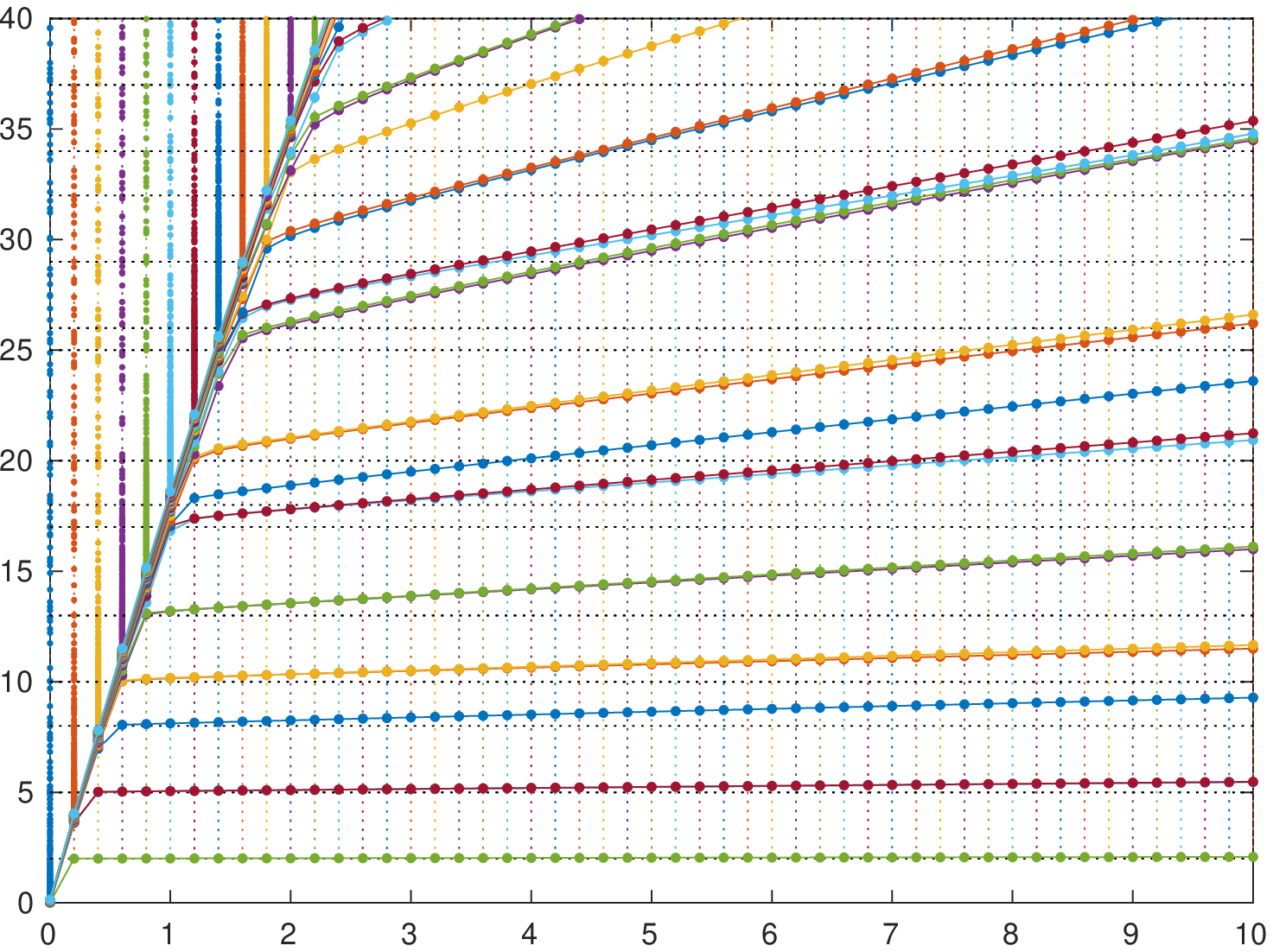}}
\subfigure[\tiny{$k=1$, $\beta=5$, $\alpha\in[0,10]$}]
{\includegraphics[width=.32\textwidth]{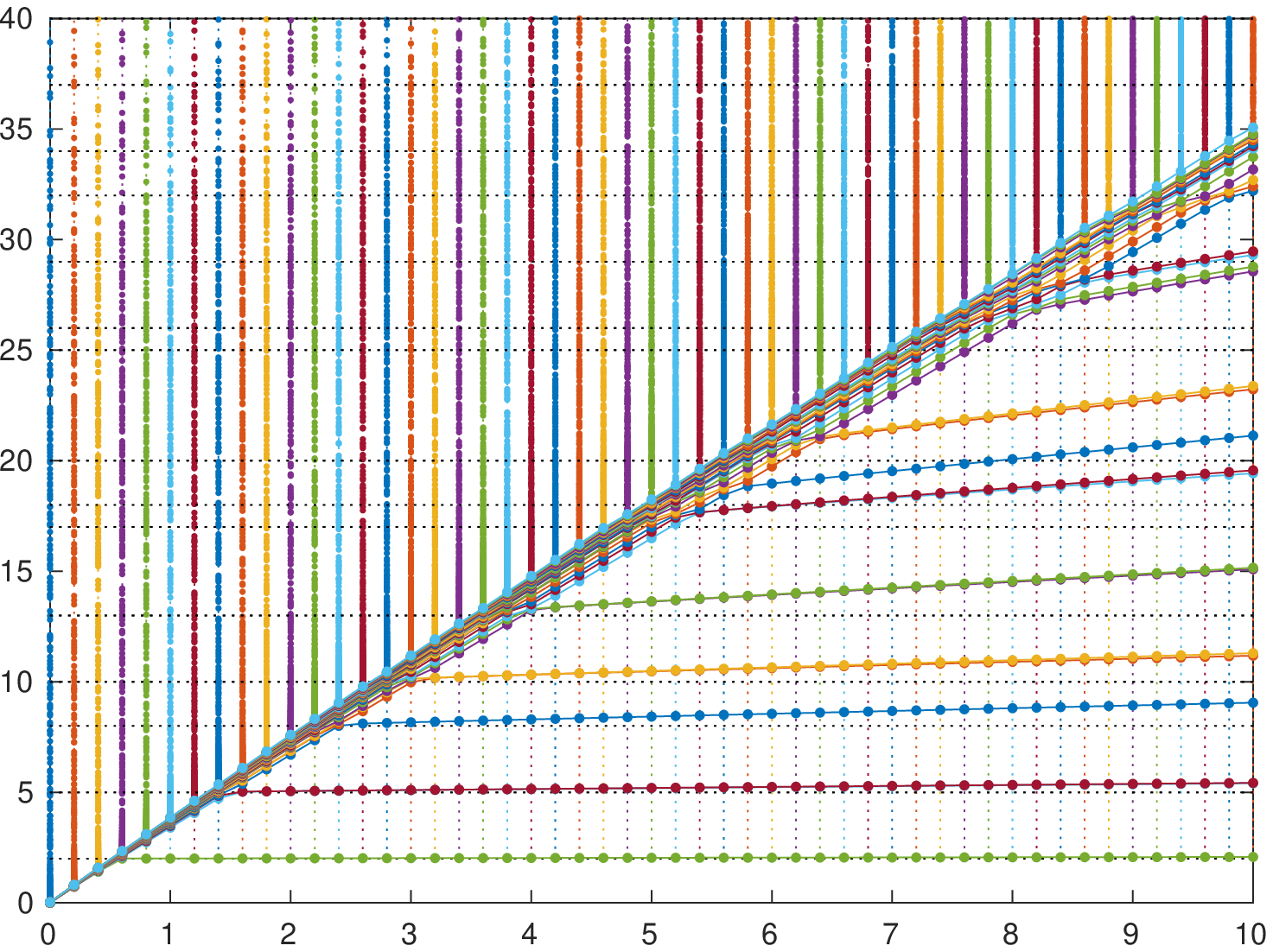}}

\subfigure[\tiny{$k=2$, $\beta=0$, $\alpha=[0,10]$}]
{\includegraphics[width=.32\textwidth]{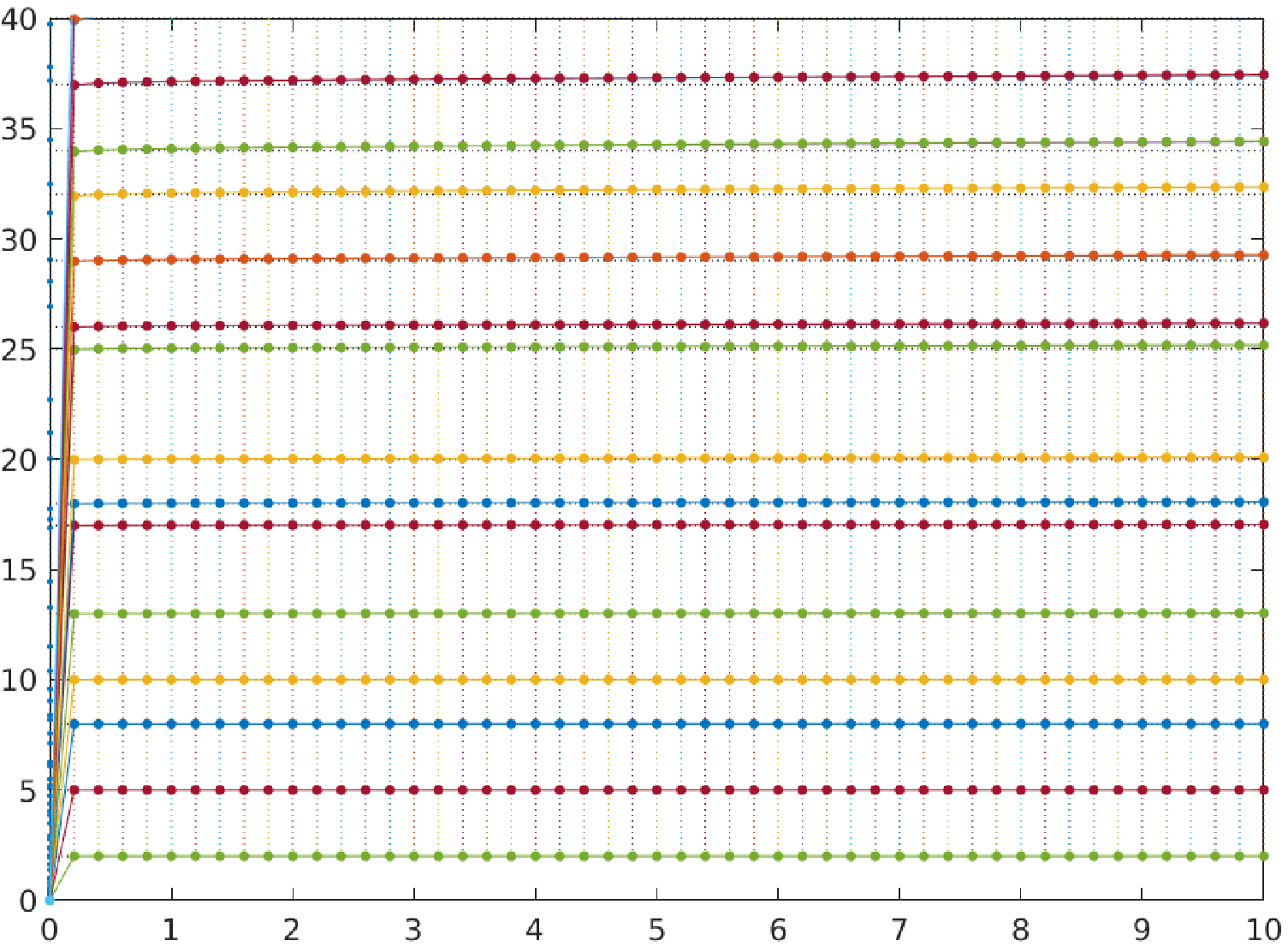}}
\subfigure[\tiny{$k=2$, $\beta=1$, $\alpha=[0,10]$}]
{\includegraphics[width=.32\textwidth]{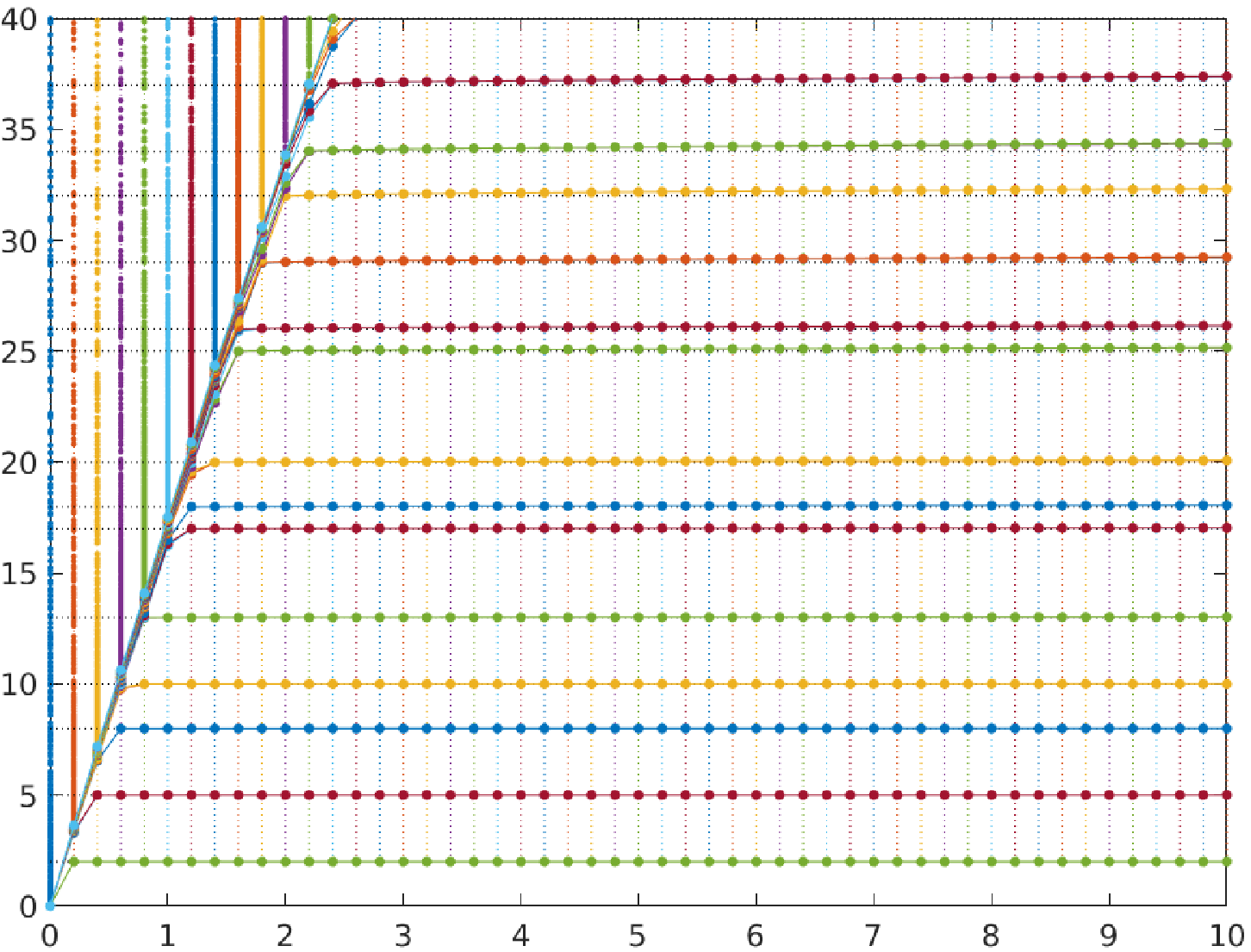}}
\subfigure[\tiny{$k=2$, $\beta=5$, $\alpha=[0,10]$}]
{\includegraphics[width=.32\textwidth]{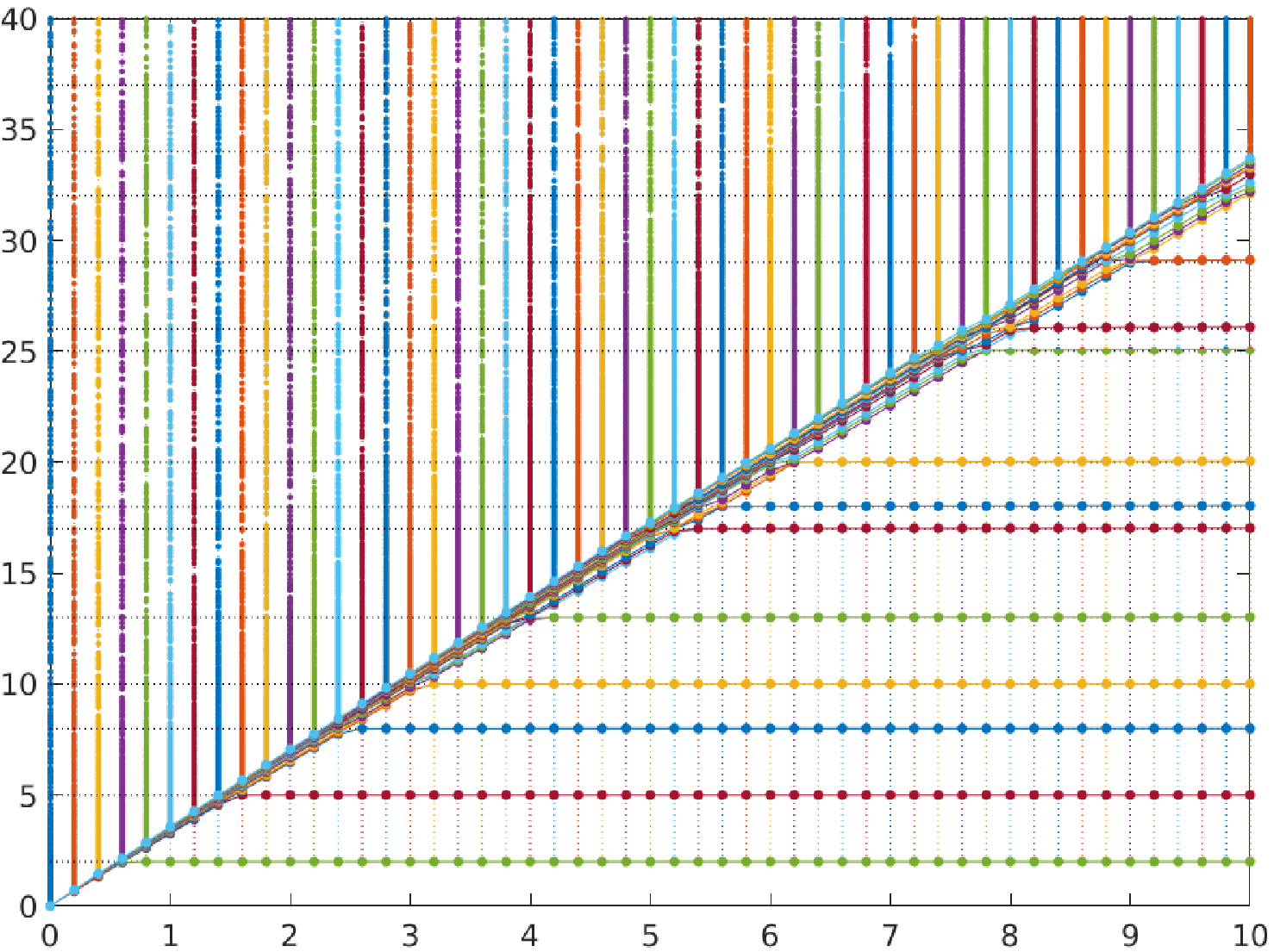}}

\subfigure[\tiny{$k=3$, $\beta=0$, $\alpha=[0,10]$}]
{\includegraphics[width=.32\textwidth]{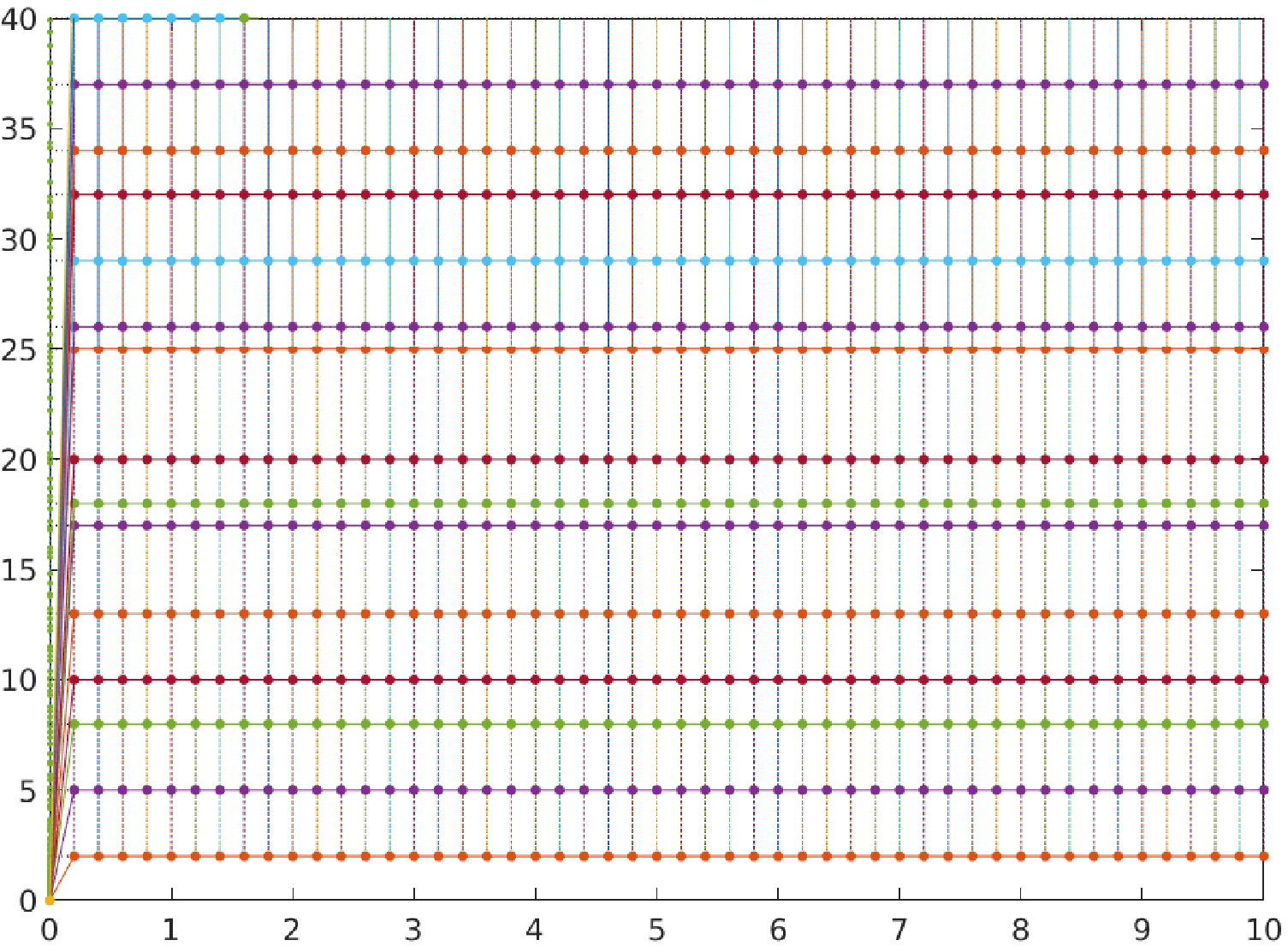}}
\subfigure[\tiny{$k=3$, $\beta=1$, $\alpha=[0,10]$}]
{\includegraphics[width=.32\textwidth]{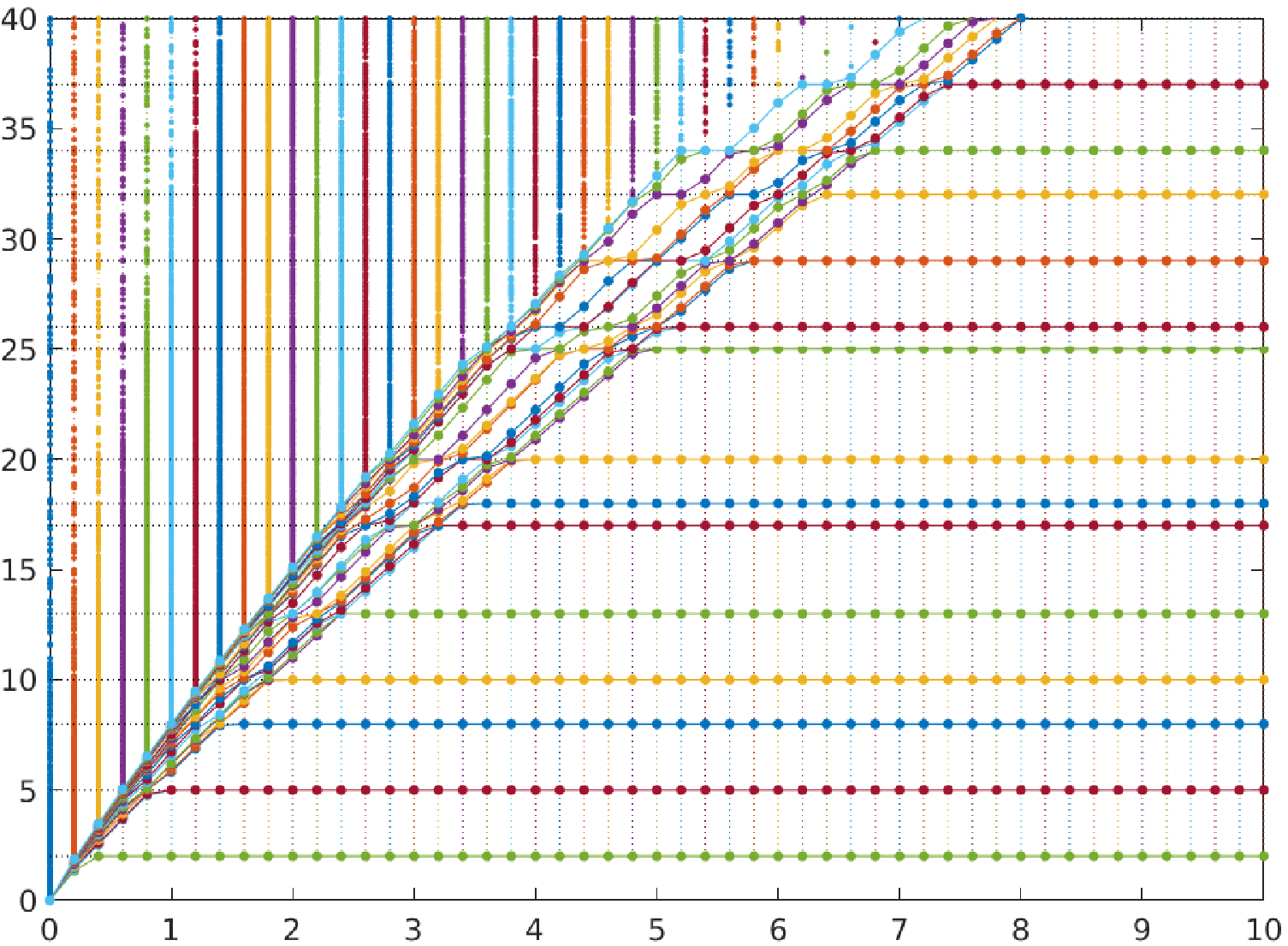}}
\subfigure[\tiny{$k=3$, $\beta=5$, $\alpha=[0,10]$}]
{\includegraphics[width=.32\textwidth]{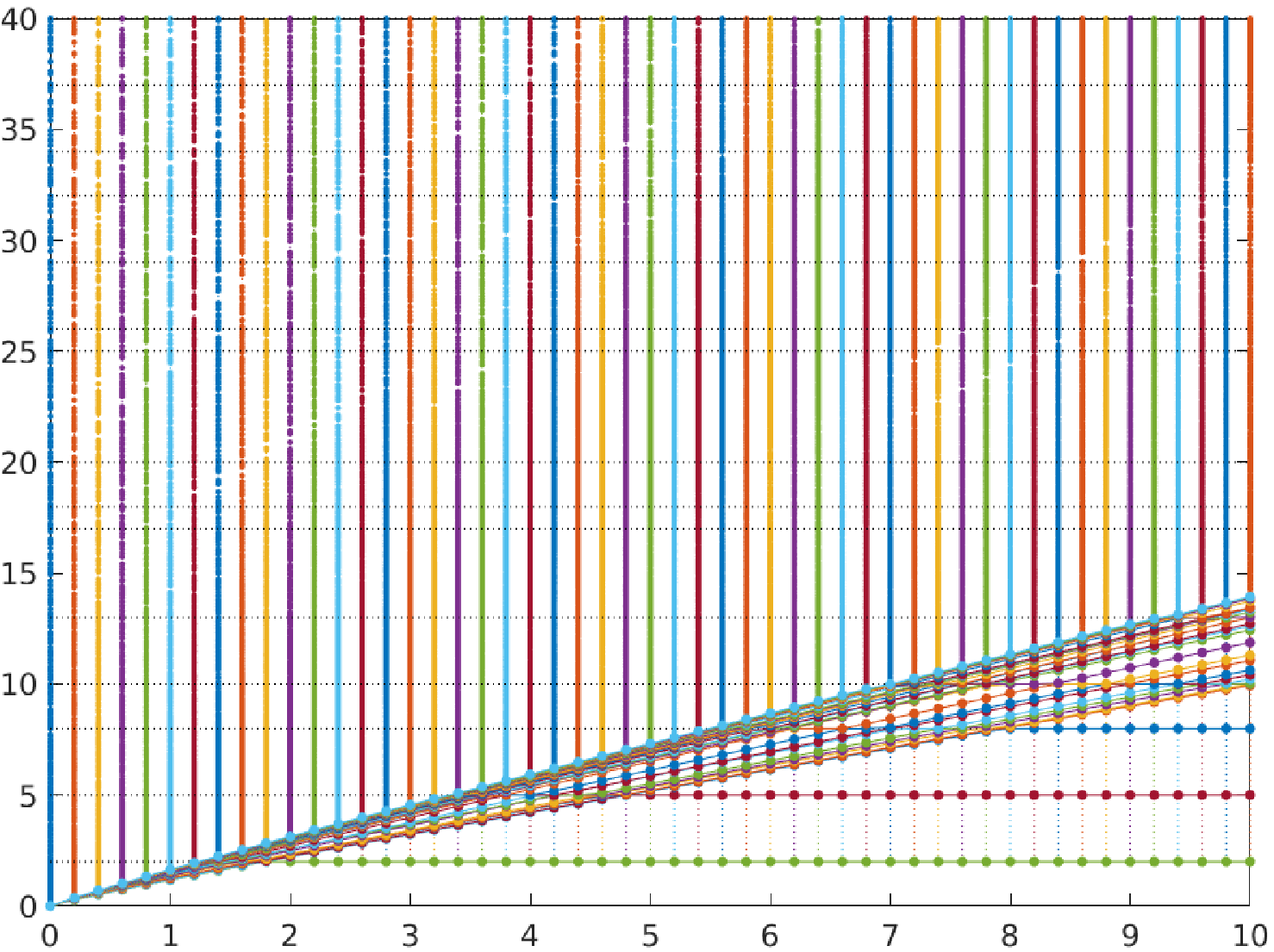}}
\end{center}
\caption{Eigenvalues with fixed $k$ and $\beta\ge0$ and varying
$\alpha\in[0,10]$}
\label{fg:alpha}
\end{figure}

The last computations of this section involve the situation when $\alpha>0$ is
fixed and $\beta$ is varying.
\begin{figure}

\begin{center}
\subfigure[\tiny{$k=1$, $\alpha=0.1$, $\beta\in[0,5]$}]
{\includegraphics[width=.32\textwidth]{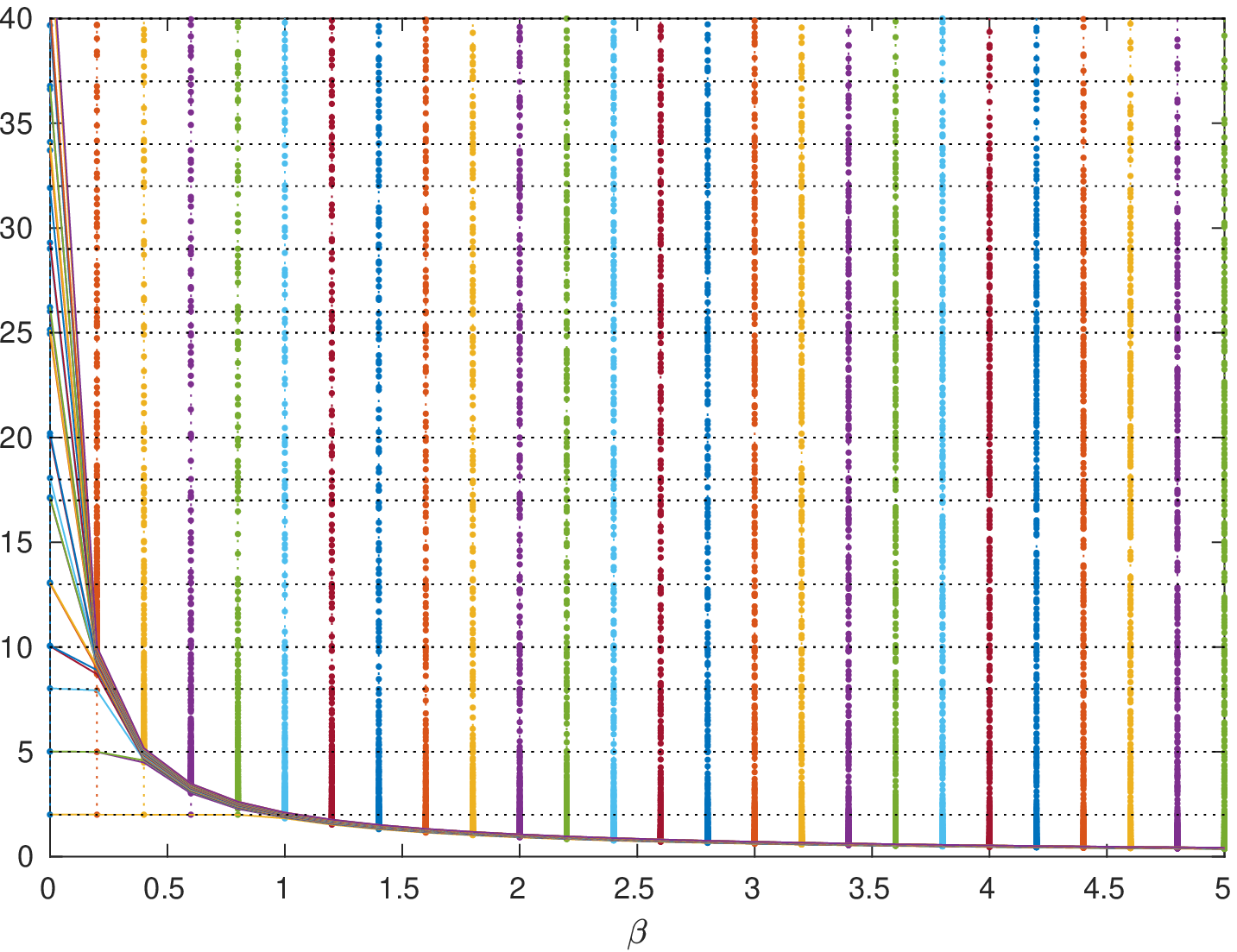}}
\subfigure[\tiny{$k=1$, $\alpha=1$, $\beta\in[0,5]$}]
{\includegraphics[width=.32\textwidth]{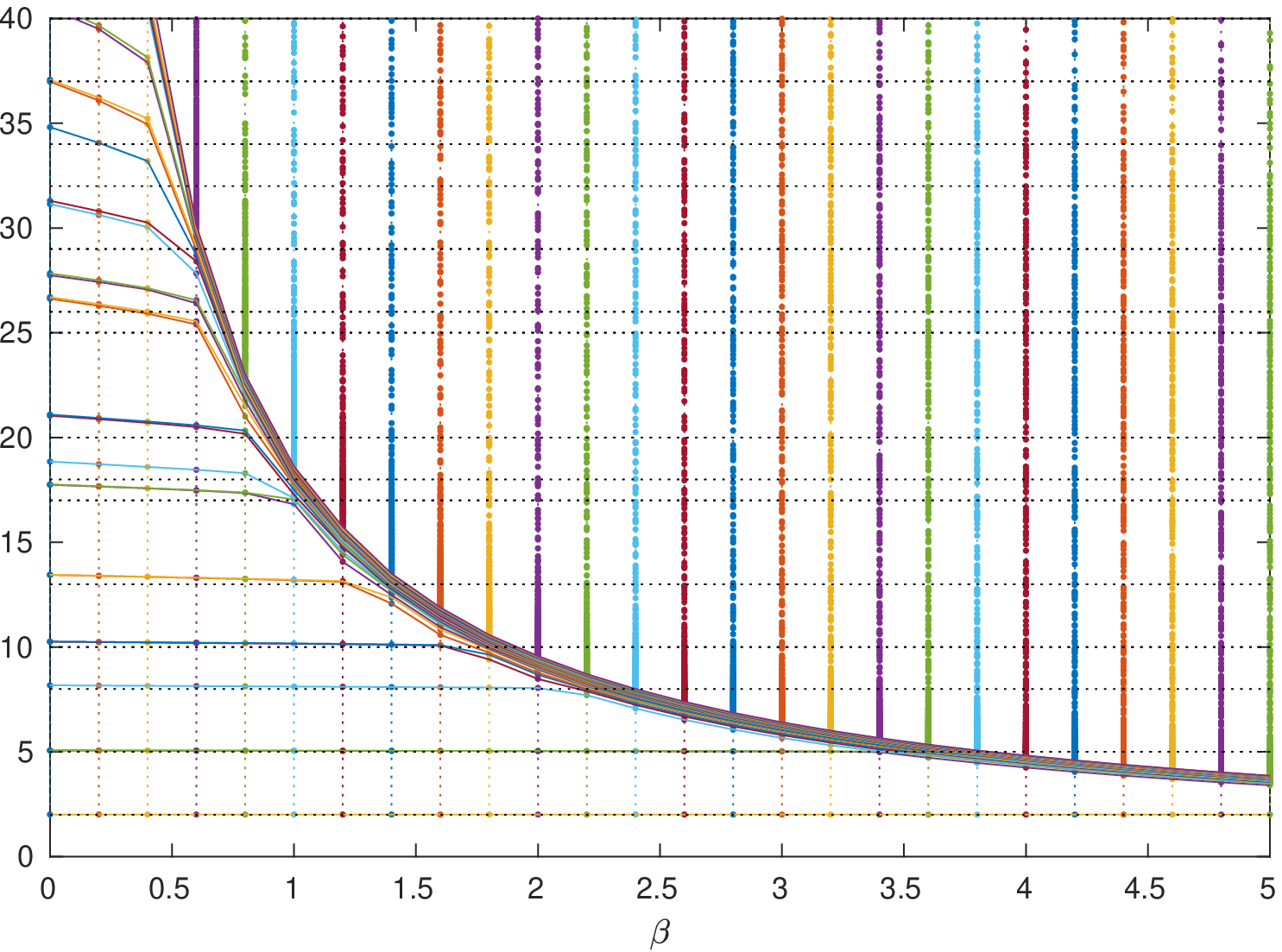}}
\subfigure[\tiny{$k=1$, $\alpha=10$, $\beta\in[0,5]$}]
{\includegraphics[width=.32\textwidth]{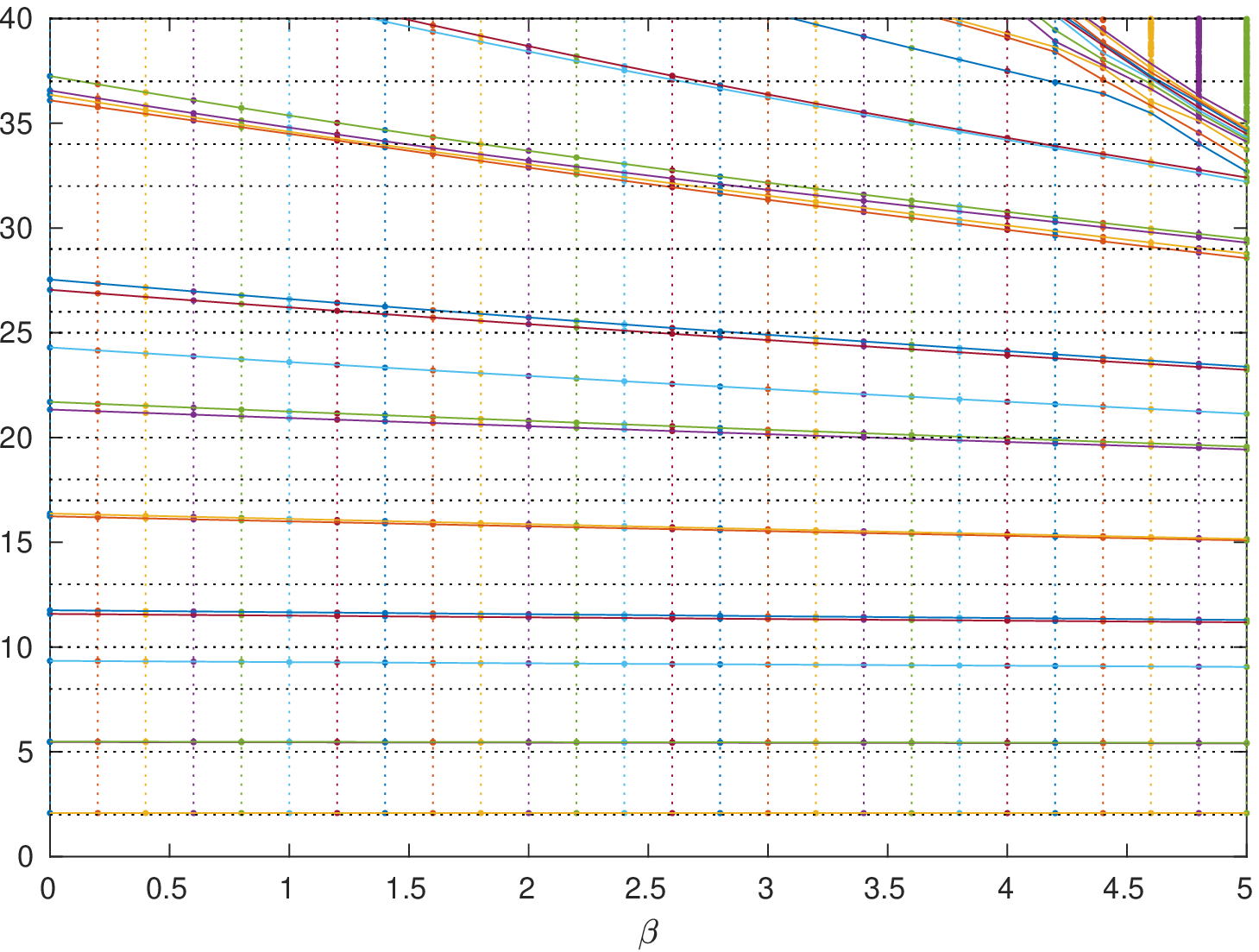}}

\subfigure[\tiny{$k=2$, $\alpha=0.1$, $\beta=[0,5]$}]
{\includegraphics[width=.32\textwidth]{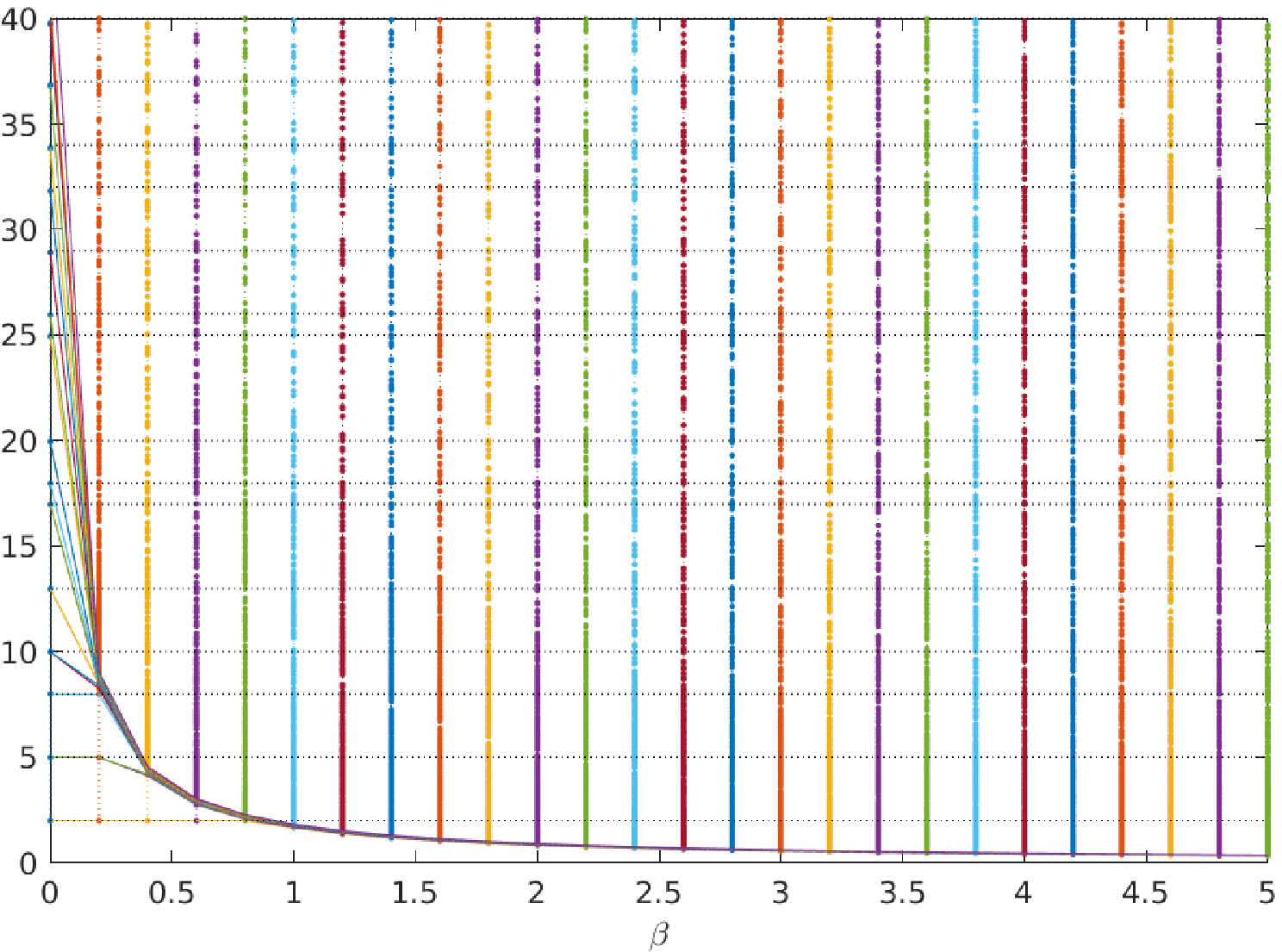}}
\subfigure[\tiny{$k=2$, $\alpha=1$, $\beta=[0,5]$}]
{\includegraphics[width=.32\textwidth]{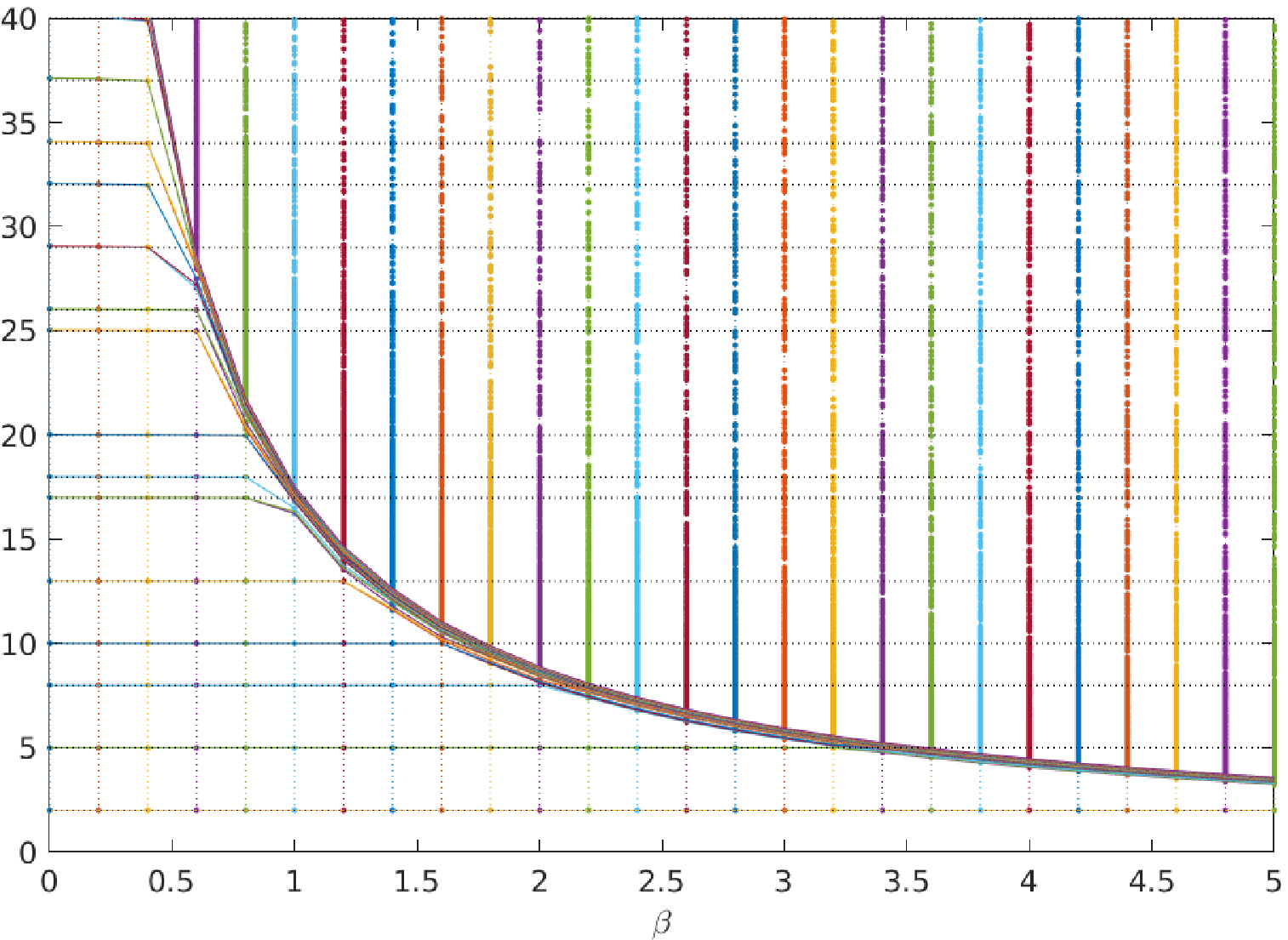}}
\subfigure[\tiny{$k=2$, $\alpha=10$, $\beta=[0,5]$}]
{\includegraphics[width=.32\textwidth]{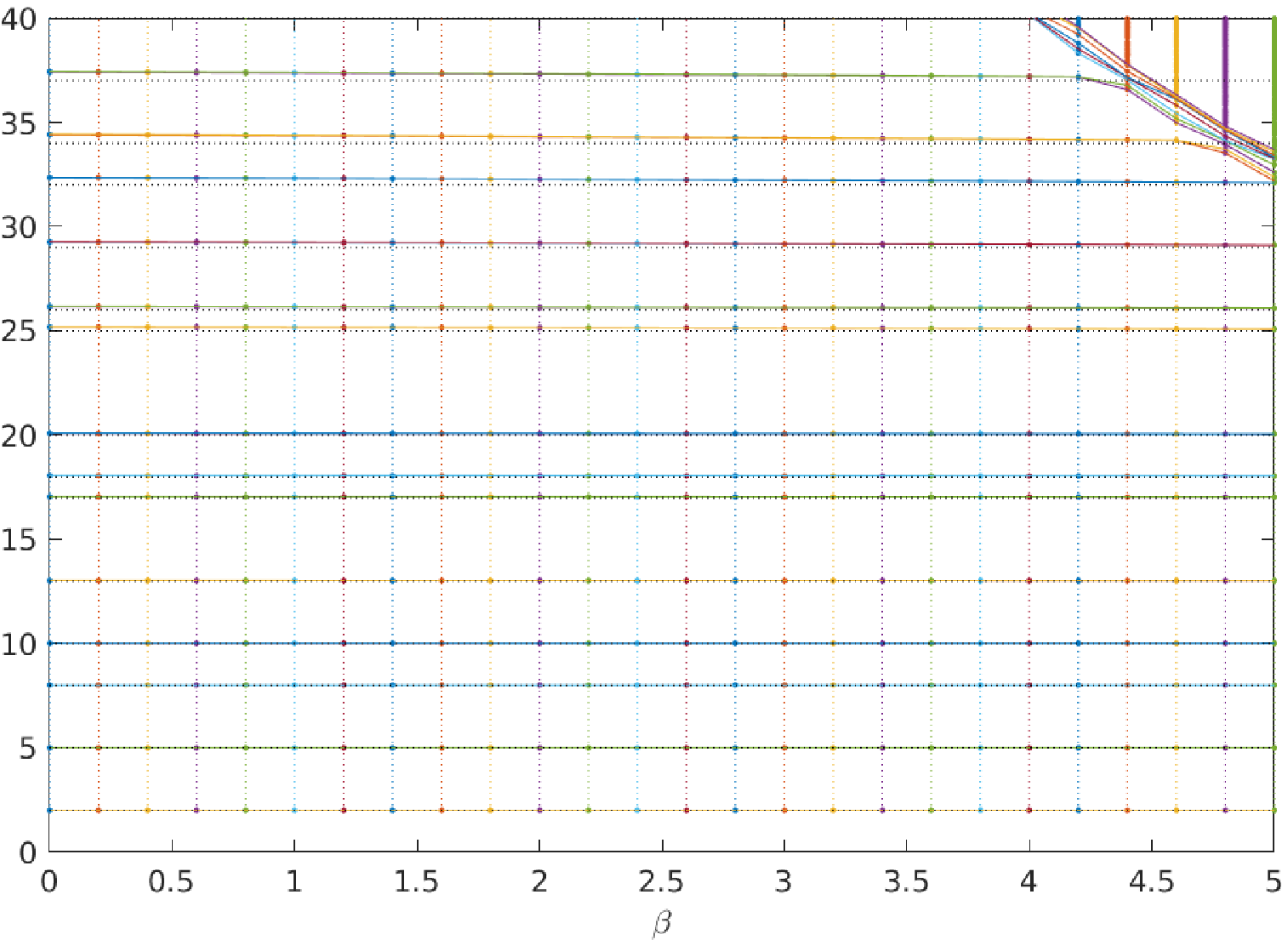}}

\subfigure[\tiny{$k=3$, $\alpha=0.1$, $\beta=[0,5]$}]
{\includegraphics[width=.32\textwidth]{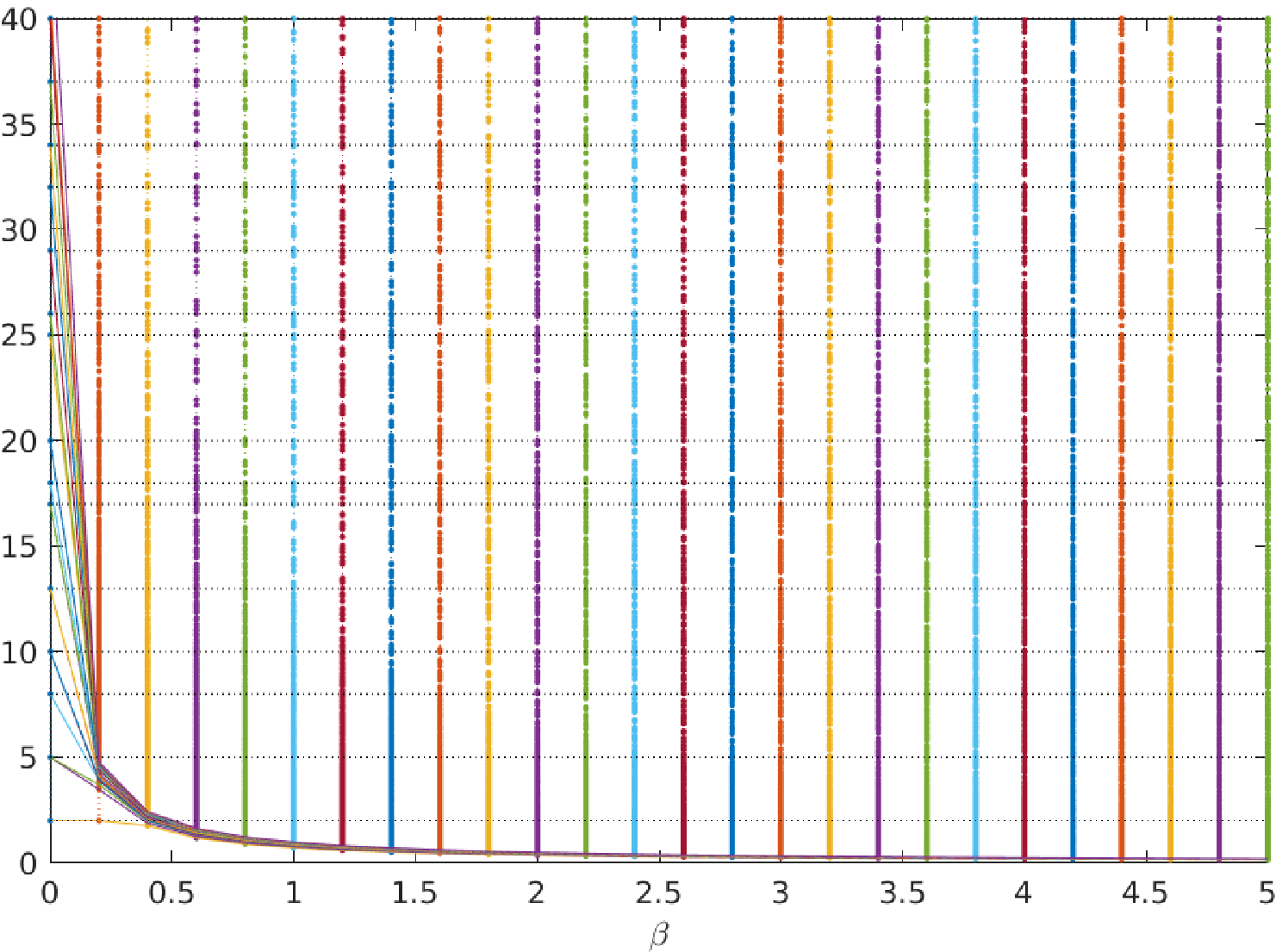}}
\subfigure[\tiny{$k=3$, $\alpha=1$, $\beta=[0,5]$}]
{\includegraphics[width=.32\textwidth]{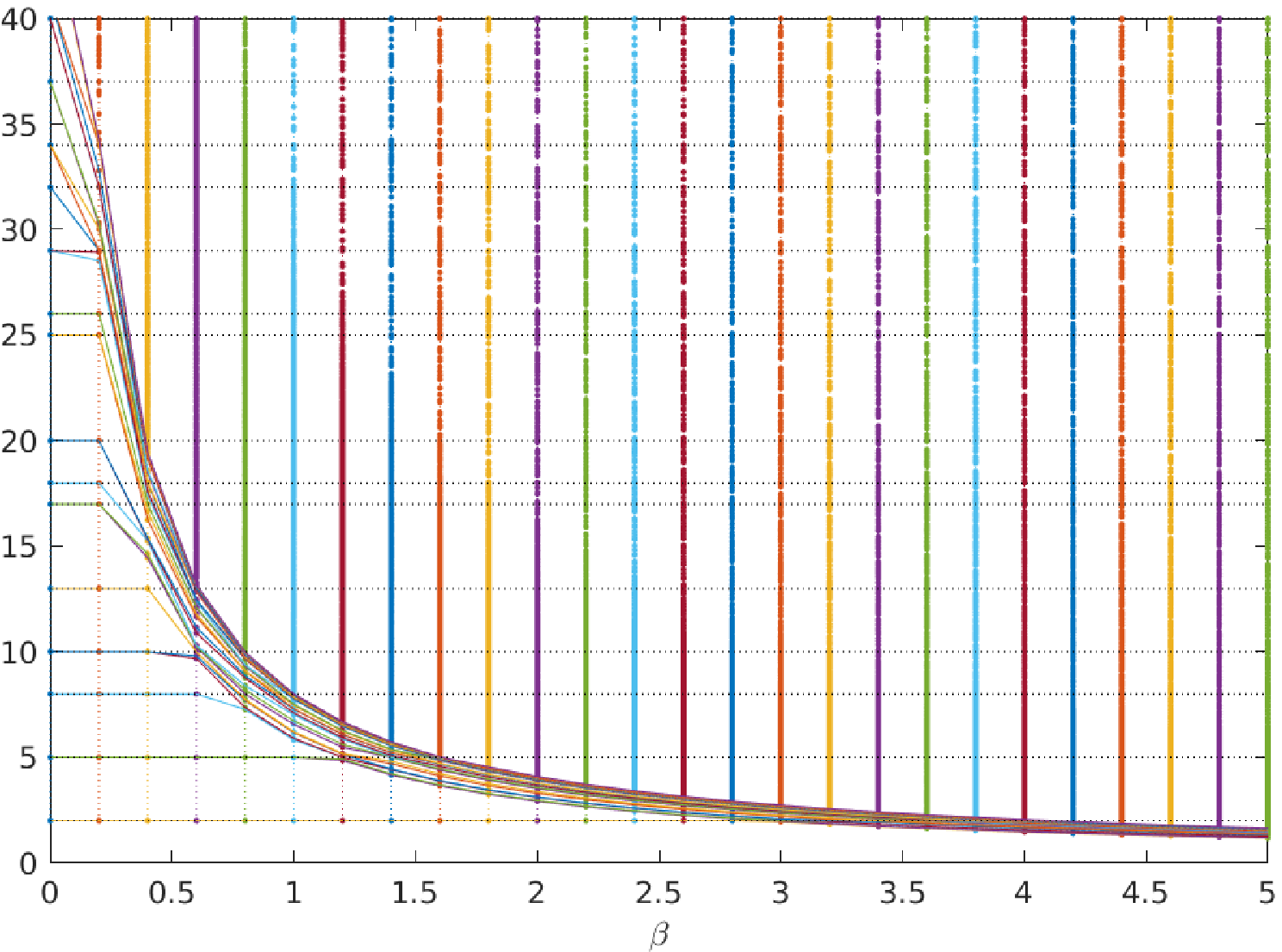}}
\subfigure[\tiny{$k=3$, $\alpha=10$, $\beta=[0,5]$}]
{\includegraphics[width=.32\textwidth]{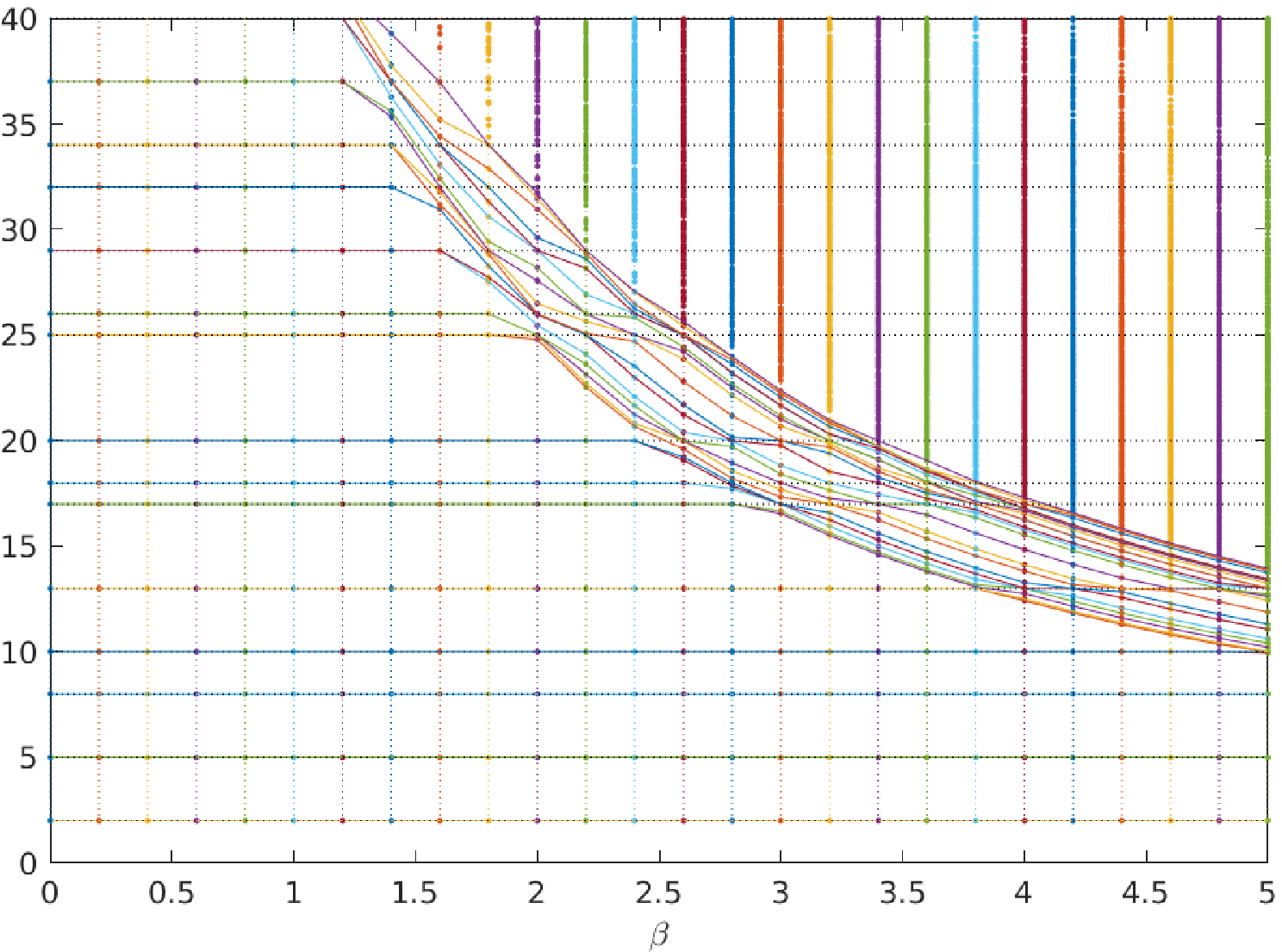}}
\end{center}

\caption{Eigenvalues with fixed $k$ and $\alpha>0$ and varying
$\beta\in[0,5]$}
\label{fg:beta}
\end{figure}
Figure~\ref{fg:beta} shows the behavior of the eigenvalues for different
values of $\alpha$ and $k$ and with $\beta$ varying in the interval $[0,5]$;
the analogies with the simplified case shown in Figure~\ref{fg:case2} are
evident.

In conclusion, the reader might wonder what is our suggestion for the choice
of the parameters. The answer is not immediate and is not definitive. One
clear message is that the dependence on $\beta$ may have more dangerous
effects than the one on $\alpha$. As a rule of thumb, our numerical tests
indicate that $\alpha$ should be chosen large enough, possibly using the same
criteria adopted for the source problem, while small values of $\beta$ (or
even $\beta=0$) are preferable. Clearly, a small value of $\beta$ leads to an
ill-conditioned problem: the limit case $\beta=0$ may corresponds to a
degenerate eigenvalue problem where the matrix $\m{M}$ is singular. As it is
typical for eigenvalue problems, if $h$ is not small enough the eigenvalues
are not yet well resolved and the dependence on the parameters might be more
complicated to analyze and may lead to useless results.
In order to better illustrate how large values of $\beta$ can produce wrong
results, in Figure~\ref{fg:start} we show the computation of the first four
eigenvalues with $\alpha=10$, $k=1$, and $\beta\in[0,400]$.
\begin{figure}
\begin{center}
\includegraphics[width=\textwidth]{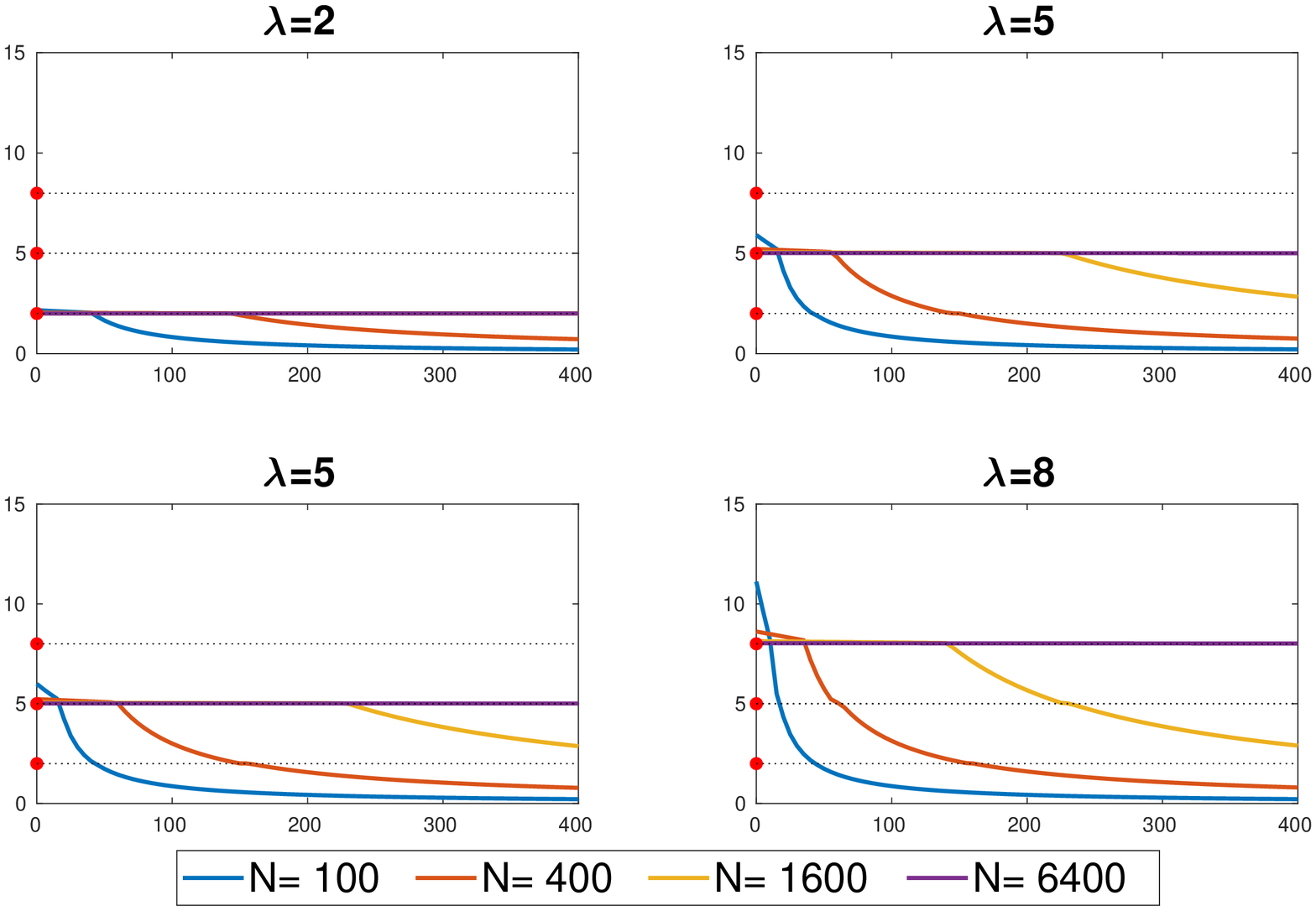}
\end{center}
\caption{First four eigenvalues depending on $\beta$ for $\alpha=10$ and
$k=1$}
\label{fg:start}
\end{figure}
Each subplot represents one of the first four eigenvalues as function of
$\beta$; different colors correspond to different meshes. These plots should
be compared to Figure~\ref{fg:beta}. In particular, Figure~\ref{fg:beta}(c)
shows computations for the same values of $\alpha$ and $k$ on the mesh
$N=200$. We comment in detail the behavior of the approximation to the
fourth eigenvalue shown in the bottom/right subplot of Figure~\ref{fg:start};
similar considerations apply to the other eigenvalues. First of all, it is
clear that if $h$ is small enough then the solution is correct: when $N=6400$
the fourth computed eigenvalue is independent on $\beta$ and provides a good
approximation of $\lambda=8$. On the coarsest mesh, however, the results
obtained are useless in order to predict the correct value of the fourth
eigenvalue: for large values of $\beta$ the computed eigenvalue is much
smaller than $8$ and tends to zero as $\beta$ increases; for a value of
$\beta$ between $10$ and $20$ the computed eigenvalue crosses the line
$\lambda=8$ and it becomes much larger for smaller values of $\beta$ reaching
a value of approximately $12$ for $\beta=0$. The results for the next mesh for
$N=400$ are similar for large $\beta$ but are significantly different after
the red line crosses the value $\lambda=8$. For values of $\beta$ smaller than
the crossing point, the computed eigenvalue remains closer to the correct
value of the exact solution. Another interesting phenomenon related to this
curve is that when it reaches $\lambda=2$ and, more evidently, $\lambda=5$,
the curve seems to bend and to be attracted for a while by the first and
second (double) eigenvalue. This behavior is due to the fact that the curve is
crossing other computed eigenvalues which are not plotted in this figure and,
after the crossing, it continues along another hyperbola.

\section{Applications}
\label{se:appl}

We conclude this paper with a discussion on the application of VEM to
eigenvalue problems arising from various models and formulations. We start
with a presentation of the mixed formulation for our model problem and then we
continue with other applications following the chronological order of the
related publications.

\subsection{The mixed Laplace eigenvalue problem}

In~\cite{mixed} the mixed formulation of the Laplace eigenvalue is considered.
The following usual mixed formulation is consider: $\Sigma=\Hdiv$,
$U=L^2(\Omega)$, so that we are seeking for $\lambda\in\RE$ and a non
vanishing $u\in U$ such that for $\bfsigma\in\Hdiv$ it holds
\[
\aligned
&(\bfsigma,\bftau)+(\div\bftau,u)=0&&\forall\bftau\in\Sigma\\
&(\div\bfsigma,v)=-\lambda(u,v)&&\forall v\in U.
\endaligned
\]

It is well known that the analysis of the approximation of the Laplace
eigenvalue in mixed form requires particular care because of the lack of
compactness of the solution operator with respect to the component $\bfsigma$.
The analysis of~\cite{mixed} extends the theory of~\cite{bbg2} to the virtual
element method.
The discrete spaces are constructed as follows: the space $\Sigma$ is
approximated by the following space described in~\cite{BFM}
\[
\aligned
\Sigma_h&=\{\bftau_h\in\Sigma:(\bftau_h\cdot\bfn)|_e\in\mathbb{P}_k(e)\
\forall e\in\edges,\ \div\bftau_h|_P\in\mathbb{P}_{k-1}(P),\\
&\qquad\rot\bftau_h|_P\in\mathbb{P}_{k-1}(P)\ \forall P\in\T\}
\endaligned
\]
while the space $U$ is approximated by a standard piecewise discontinuous
space
\[
U_h=\{v_h\in U:v_h|_P\in\mathbb{P}_{k-1}(P)\ \forall P\in\T\}.
\]

The matrix form of the approximate problem reads
\[
\left(
\begin{matrix}
\m{A}&\m{B}^\top\\
\m{B}&\m{0}
\end{matrix}
\right)
\left(
\begin{matrix}
\m{\sigma}\\
\m{u}
\end{matrix}
\right)
=-\lambda
\left(
\begin{matrix}
\m{0}&\m{0}\\
\m{0}&\m{M}
\end{matrix}
\right)
\left(
\begin{matrix}
\m{\sigma}\\
\m{u}
\end{matrix}
\right).
\]
In this case, the matrix $\m{M}$ does not contain any stabilization parameter
since the space $U_h$ is a standard mass matrix arising from discontinuous
finite elements. The same applies to the matrix $\m{B}$ that can be
constructed directly by using the degrees of freedom of the spaces $\Sigma_h$
and $U_h$.
Hence, the only matrix that requires a stabilization is $\m{A}$.

The analysis of~\cite{mixed} shows, using the tools of~\cite{bbg2} and the
properties of the spaces $\Sigma_h$ and $U_h$, the following uniform bound for
the solution of the source problem with right hand side $f\in L^2(\Omega)$:
\[
\|\bfsigma-\bfsigma_h\|_0+\|u-u_h\|_0\le Ch^t\|f\|_0
\]
where $t=\min(k,r)$, $r$ being such that $u\in H^{1+r}(\Omega)$
(see~\eqref{eq:regularity}). From this estimate the convergence of the
eigenvalues and eigenfunctions follows by standard arguments and an appropriate
definition of the solution operator (see Section~\ref{se:setting}
and~\cite{acta}).

\subsection{The Steklov eigenvalue problem}

Another formulation where the stabilization parameter enters only the matrix
on the left hand side of the discrete system, is given by the so called
Steklov eigenvalue problem
\[
\aligned
&\Delta u=0&&\text{in }\Omega\\
&\frac{\partial u}{\partial\bfn}=\lambda u&&\text{on }\partial\Omega.
\endaligned
\]
A more general setting could be actually considered by taking the second
equation only on a subset $\Gamma_0$ of $\partial\Omega$ and homogeneous
Neumann boundary conditions on the remaining part of the boundary.
We refer to~\cite{Steklov} for the analysis of its VEM discretization.

The variational formulation is obtained by considering $V=H^1(\Omega)$ and
seeks for $\lambda\in\RE$ and a non vanishing $u\in V$ such that
\[
\int_\Omega\nabla u\cdot\nabla v\,dx=
\lambda\int_{\partial\Omega}uv\,ds\qquad\forall v\in V.
\]
Since the bilinear form on the left hand side is not $\Hu$-elliptic, the
analysis is based on the use of a shift argument by adding the boundary
integral on both sides of the eigenvalue problem in weak form. Hence, the
resolvent operator and its discrete counterpart are
defined using the shifted problem. In this case, it is necessary to consider
$T:V\to V$ in order to give sense to the datum of the source problem.   

The approximation of the problem makes use of the standard VEM spaces
described in Section~\ref{se:theory}. As we have seen before, the discrete
problem requires to introduce a discretization of the bilinear form on the left
hand side, with the addition of a stabilization term. Whereas, the integral on
the right hand side can be computed using the degrees of freedom on the element
boundaries. 

For $f\in\Hu$, the solution to the associated source problem belongs
to $H^{1+r}(\Omega)$, with $r>\frac12$. This implies that the resolvent
operator is compact on $V$, so that the analysis in~\cite{Steklov} relies on
the standard theory for compact operators that we have recalled in
Section~\ref{se:setting}. In particular, it is based on the following uniform
convergence of $T_h$ to $T$:
\[
\|(T-T_h)f\|_1\le C h^r\|f\|_1\quad\forall f\in\Hu,
\]
which in turn gives the same rate of convergence for the gap between
eigenspaces and double rate of convergence for the eigenvalue error.

The matrix form of the discrete problem is
\[
\m{A}\m{u}=\lambda\m{M}\m{u},
\]
where $\m{A}$ depends on the stabilization parameter $\alpha_P$.
In~\cite{Steklov} one can find an interesting test case which models the
sloshing modes of a two dimensional fluid contained in a square. 
In this test case, $\Omega$ is the unit square and its boundary is
divided into two parts: $\Gamma_0$ the top side of the square
and $\Gamma_1$ the remaining three sides. On $\Gamma_0$ the Steklov condition
is imposed, while on $\Gamma_1$ homogeneous Neumann condition is enforced.
The convergence analysis with fixed parameter $\alpha_P=1$ for all $P\in\T$
confirms the theoretical results.
The paper contains, for the first time, the study of the effect of the
stability on the computed eigenvalues on a fixed mesh. It has been observed
that, when $\alpha_P$ is independent of the element $P$ and varies from
$4^{-3}$ to $4^3$ spurious eigenvalues appear among the correct ones. 
In Figure~\ref{fg:steklov}, we plot the first 16 eigenvalues reported
in~\cite[Table~3]{Steklov}. These eigenvalues were been obtained using a
triangular mesh, considering the middle point of each edge as a new vertex of
the polygon. Each side of the square has been divided into 8 parts and virtual
elements with $k=1$ are applied. The plot on the right is a zoom of the one
on the left.  The dotted black horiziontal lines represent the first 3 exact
eigenvalues.  One can see that for $\alpha_P>1$ the spurious eigenvalues seems
to be proportional to the stabilization parameter and this appears to be in
agreement with the analysis reported in Section~\ref{se:parameter}, while the
approximation of the three correct ones is independent of $\alpha_P$. 

\begin{figure}
\begin{center}
\includegraphics[width=.49\textwidth]{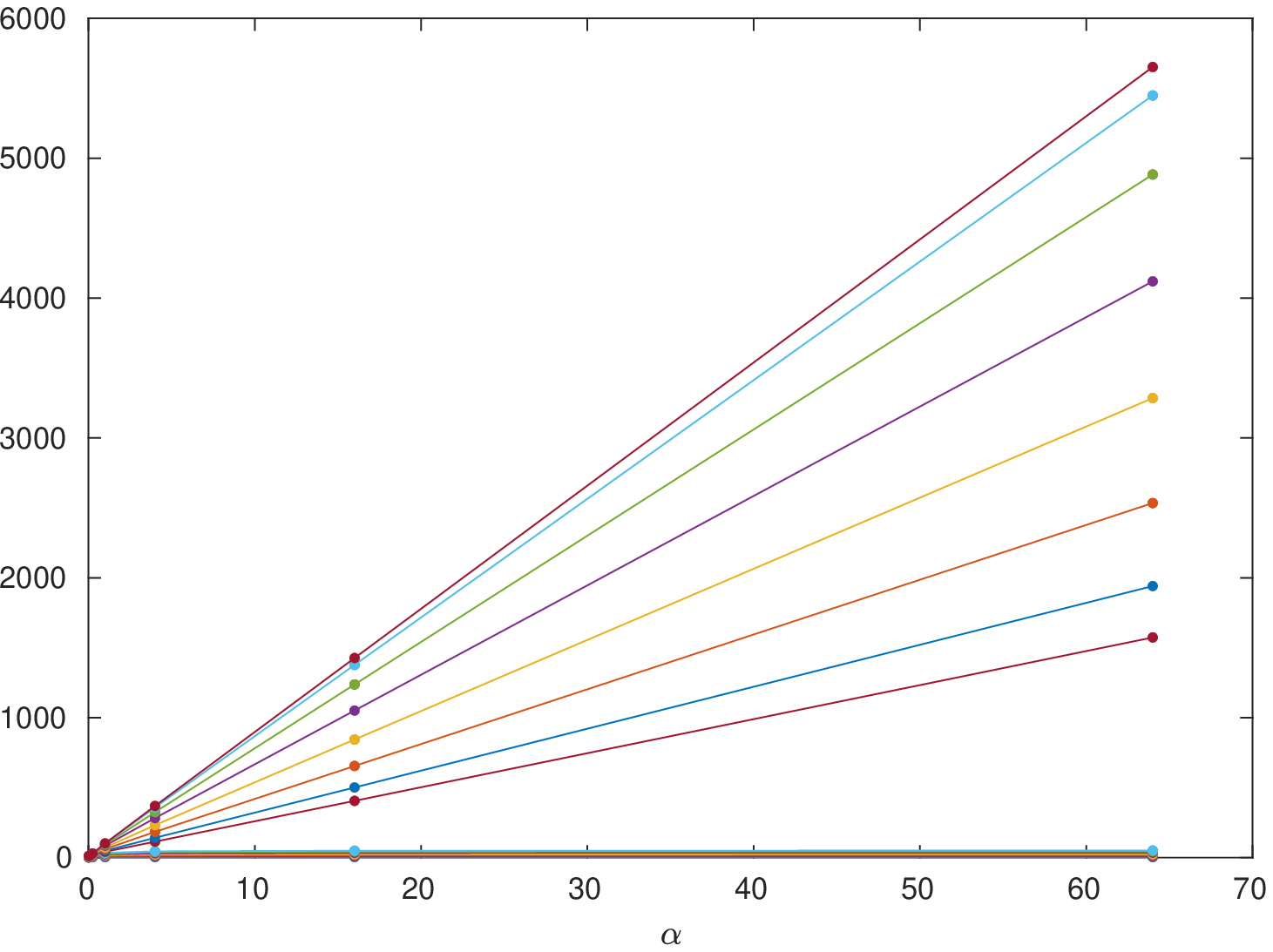}
\includegraphics[width=.49\textwidth]{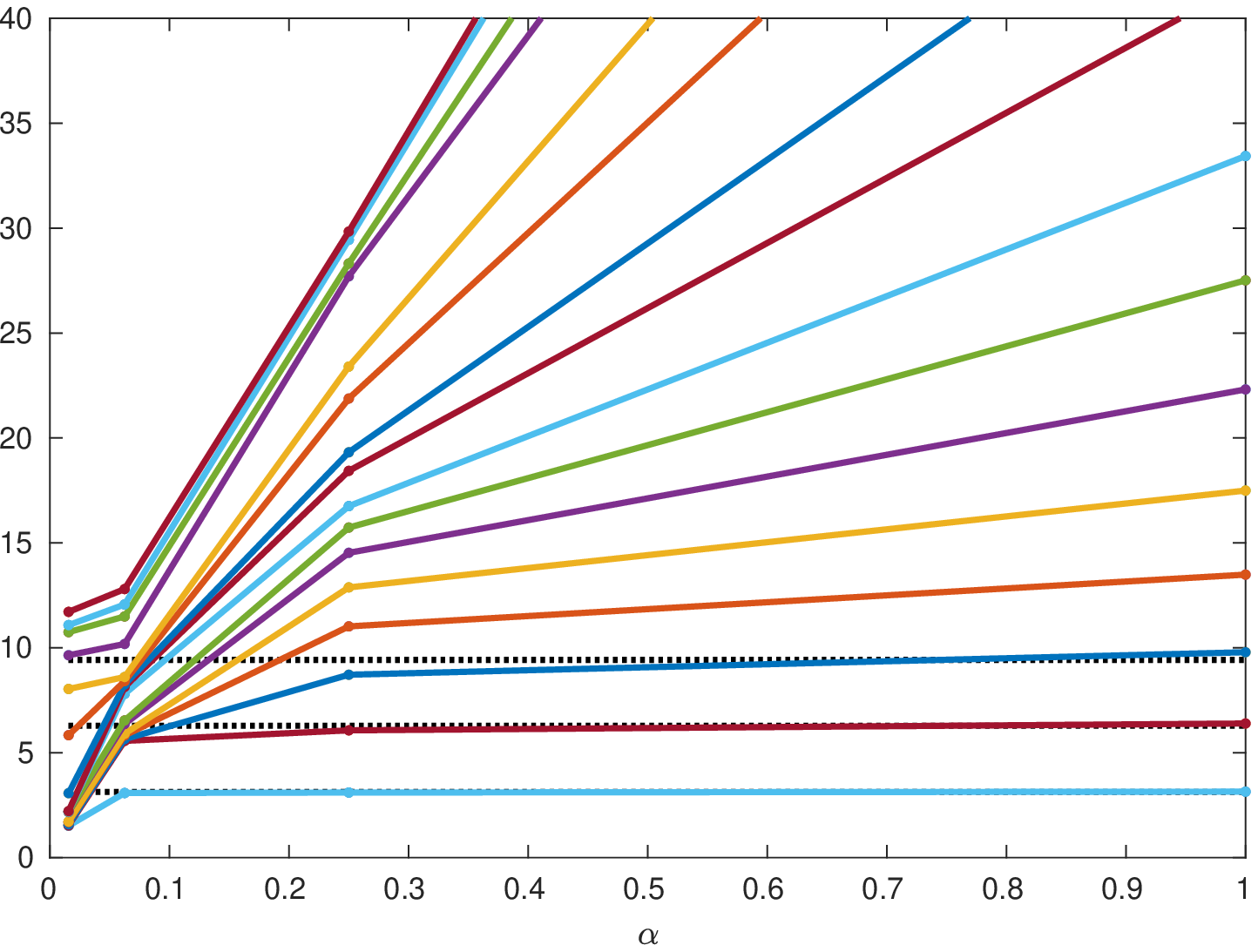}
\end{center}
\caption{First 16 eigenvalues for $4^{-3}\le\alpha_P\le4^3$}
\label{fg:steklov}
\end{figure}
\subsection{An acoustic vibration problem}

The next application has been considered in~\cite{acoustic} and is an example
where only the matrix on the right hand side of the discrete algebraic system
contains a parameter.

The model describes the free vibrations of an inviscid compressible fluid
within a bounded rigid cavity $\Omega$. The simplest model consists in finding
eigenvalues $\lambda\in\RE$ and non vanishing eigenfunctions $\bfu\in\Hodiv$
such that
\[
\int_\Omega\div\bfu\,\div\bfv\,dx=\lambda\int_\Omega\bfu\cdot\bfv\,dx
\qquad\forall\bfv\in\Hodiv.
\]
Here $\Hodiv$ is the subspace of $\Hdiv$ of functions with vanishing trace of
the normal component on $\partial\Omega$.

The approximation is performed by the following VEM space that has been
designed taking inspiration from~\cite{BBMR,BFM}:
\[
\aligned
V_h&=\{\bfv_h\in\Hodiv:(\bfv_h\cdot\bfn)|_e\in\mathbb{P}_k(e)\ \forall
e\in\edges,\ \div\bfv_h|_P\in\mathbb{P}_k(P),\\
&\qquad\rot\bfv_h|_P=0\ \forall
P\in\T\}.
\endaligned
\]

The discretized problem has the classical form
\[
\m{A}\m{u}=\lambda\m{M}\m{u}
\]
and the matrix $\m{A}$ can be formed directly using the degrees of freedom
since the divergence of elements in $V_h$ are piecewise polynomials. On the
other hand, the mass matrix $\m{M}$ should be constructed by using the typical
VEM approach and requires a stabilization in order to control the non
polynomial part of the space $V_h$.

The convergence analysis of~\cite{acoustic} is performed as in case of
standard finite elements~\cite{BDMRS} by shifting the solution operator in
order to deal with the kernel of the divergence operator, and by exploiting
the theory of~\cite{DNR} related to the analysis of non compact operators.
More precisely, the continuous solution operator $T:\Hodiv\to\Hodiv$ is defined
as the solution $T\bff\in\Hodiv$ of
\[
\int_\Omega\div T\bff\div\bfv\,dx+\int_\Omega T\bff\cdot\bfv\,dx=
\int_\Omega\bff\cdot\bfv\,dx\qquad\forall\bfv\in\Hodiv
\]
and its discrete counterpart $T_h:V_h\to V_h$ is defined as the solution
$T_h\bff_h\in V_h$ of
\[
\int_\Omega\div T_h\bff_h\div\bfv_h\,dx+\int_\Omega T_h\bff_h\cdot\bfv_h\,dx=
\int_\Omega\bff_h\cdot\bfv_h\,dx\qquad\forall\bfv_h\in V_h.
\]

The theory of~\cite{DNR} consists in showing the following mesh dependent
uniform convergence
\[
\lim_{h\to0}\sup_{\substack{\bff_h\in V_h\\\|\bff_h\|_{\div}=1}}
\|(T-T_h)\bff_h\|_{\div}=0
\]
together with the density of $V_h$ in $\Hodiv$. 
In particular, in~\cite{acoustic} it is proved the following estimate
\[
\sup_{\substack{\bff_h\in V_h\\\|\bff_h\|_{\div}=1}}
\|(T-T_h)\bff_h\|_{\div}\le C h^s,
\]
where $s\in]\frac12,1]$ is the regularity index of the solution to the
associated source problem. This results implies that for $h$ small enough
the virtual element method does not introduce spurious eigenvalues
interspersed among the physical ones we are interested in. Optimal error
estimates depending on the regularity of the eigenfunctions follow by
combining the technique of~\cite{DNR2} with the first Strang lemma.

The above theoretical results were obtained when the stability parameter
$\beta_P$ is positive. The numerical experiments show that for $k=0$ the
optimal second order rate of convergence for the eigenvalues can be achieved
for several choices of meshes. The effect of the stability
constant has also been investigated. In Figure~\ref{fg:acoustic}, we plot the 
first eigenvalue for $\beta_P$ ranging from $2^{-6}$ to $2^6$ for different
square meshes as reported in~\cite[Table~4]{acoustic}.
It can be seen that as the mesh gets finer the results are less sensible to the
value of $\beta_P$, while for coarse meshes the behavior looks similar to that
shown in Figure~\ref{fg:start} corresponding to Case~$2$ in
Section~\ref{se:parameter}. Moreover, although the choice $\beta_P=0$ is not
covered by the theory, the computational results appear to be very accurate in
this case, and also this is in agreement with the observations reported in
Section~\ref{se:parameter}.
\begin{figure}
\begin{center}
\includegraphics[width=.9\textwidth]{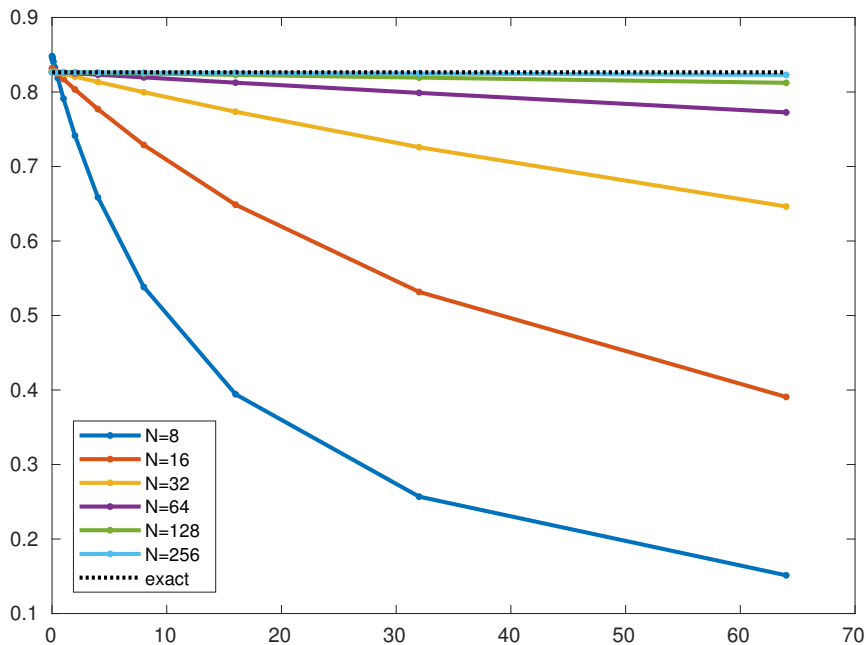}
\end{center}
\caption{First eigenvalue with $2^{-6}\le\beta\le2^6$ and different meshes}
\label{fg:acoustic}
\end{figure}
\subsection{Eigenvalue problems related to plate models}
\label{se:plates}
The computation of vibration frequencies and modes of an elastic solid is a
very important topic in engineering applications. This and the following
sections are devoted to report on the results obtained
in~\cite{Kirchhoff,buckling,MRelasticity} concerning the use of VEM to
approximate the eigenvalues of the Kirchhoff--Love model for plates and of
linear elasticty equations. 

We consider a plate whose mean surface, in its reference configuration,
occupies a polygonal bounded domain $\Omega\subset\RE^2$. The plate is clamped
on its whole boundary. Let $u$ denote the transverse displacement and $\lambda$
the vibration frequency, then the plate vibration problem, modeled by
Kirchhoff--Love equations reads: find $\lambda\in\RE$ and a non vanishing
$u$ such that:
\[
\aligned
&\Delta^2 u=\lambda u &&\quad\text{in }\Omega\\
&u=\frac{\partial u}{\partial\bfn}=0&&\quad\text{on }\partial\Omega.
\endaligned
\]
The corresponding weak form reads: find $(\lambda,u)\in\RE\times\Hd$ with
$u\ne0$ such that
\begin{equation}
\label{eq:plate}
\int_\Omega D^2u:D^2v\,dx=\lambda\int_\Omega uv\,dx\qquad\forall v\in\Hd,
\end{equation}
where $D^2 v$ denotes the Hessian matrix of $v$. The associated resolvent
operator $T$ is defined for all $f\in\Hd$ as the solution $Tf$ of the
corresponding source problem. It results to be self adjoint and compact thanks
to the fact that $Tf$ belongs to $H^{2+s}(\Omega)$ for some $s\in]\frac12,1]$
being the Sobolev regularity for the biharmonic equation with homogeneous
Dirichlet boundary conditions.

We recall here the construction of the virtual elements space proposed
in~\cite{Kirchhoff} which differs from those presented in the previous sections
since $C^1$-regularity is required for approximating functions in $\Hd$.
Given a sequence $\{\T\}_h$ of decompositions of $\Omega$ into polygons $P$,
we introduce first the following finite dimensional space:
\[
\aligned
V_h(P)&=\Big\{v_h\in H^2(P): \Delta^2v_h\in\mathbb{P}_2(P),
\ v_h|_{\partial P}\in C^0(\partial P),\\
&\quad\quad 
\nabla v_h|_{\partial P}\in C^0(\partial P)^2,\ 
v_h|_e\in\mathbb{P}_3(e),
\ \frac{\partial v_h}{\partial\bfn}\Big|_e\in\mathbb{P}_1(e)
\ \forall e\in\partial P\Big\}.
\endaligned
\]
The corresponding degrees of freedom are the values of $v_h$ and $\nabla v_h$ at
the vertices of $P$. Using these degrees of freedom it is possible to define a
projection operator $\Pidelta:V_h(P)\to\mathbb{P}_2(P)\subseteq V_h(P)$ by
solving for each $v\in V_h(P)$ the following problem:
\[
\aligned
&\adelta(\Pidelta v,q)=\adelta(v,q)&&\quad\forall q\in\mathbb{P}_2(P)\\
&((\Pidelta v, q))_P=((v,q))_P&&\quad\forall q\in\mathbb{P}_1(P),
\endaligned
\]
where 
\[
\adelta(u,v)=\int_P D^2 u:D^2 v\,dx,
\quad ((v,q))_P=\sum_{i=1}^{N_P}u(V_i)v(V_i)\qquad
\forall u,v\in H^2(P),
\] 
and $V_i$ $i=1,\dots,N_P$ are the vertices of $P$.

Then the local virtual space is:
\[
W_h(P)=\left\{v_h\in V_h(P): \int_P (\Pidelta v_h-v_k)q\,dx=0\ 
\forall q\in\mathbb{P}_2(P)\right\}.
\]
Thanks to this characterization of $W_h(P)$, the $L^2(P)$-projection operator
onto $\mathbb{P}_2(P)$ coincides with $\Pidelta$.

With the above definitions, the global virtual space is defined as:
\[
W_h=\{v_h\in\Hd: v_h|_P\in W_h(P)\}
\]
and the discrete bilinear forms are given for all $u_h,v_h\in W_h$ by  
\[
a_h(u_h,v_h)=\sum_{P\in\T} \adh(u_h,v_h),\quad 
b_h(u_h,v_h)=\sum_{P\in\T} \bdh(u_h,v_h),
\]
with
\[
\aligned
&\adh(u_h,v_h)=\adelta(\Pidelta u_h,\Pidelta v_h)
+s_{\Delta,a}^P(u_h-\Pidelta u_h,v_h-\Pidelta v_h)\\
&\bdh(u_h,v_h)=\int_P\Pidelta u_h\Pidelta v_h\,dx
+s_{\Delta,b}^P(u_h-\Pidelta u_h,v_h-\Pidelta v_h).
\endaligned
\]
As usual the stabilization terms $s_{\Delta,a}^P$ and $s_{\Delta,b}^P$ have been
added in order to guarantee consistency and stability.

With the above definitions, the discrete counterpart of
equation~\eqref{eq:plate} reads: find $(\lambda_h,u_h)\in\RE\times W_h$ with
$u\ne0$ such that
\[
a_h(u_h,v_h)=\lambda_hb_h(u_h,v_h)\quad\forall v_h\in W_h.
\]
After introducing the resolvent operator $T_h:W_h\to W_h$ as the mapping that
associates to any $f_h\in W_h$ the solution $T_hf_h$ to the corresponding source
problem, the convergence analysis follows from the theory developed
in~\cite{DNR,DNR2} and gives optimal rate of convergence for the gap between
continuous and discrete eigenspaces according to the regularity of the
eigenfunctions and double rate of convergence for the eigenvalue error. 

The numerical results require that the choice for the stabilization terms is
made precise, therefore using the degrees of freedom of the local space one can
set
\[
\aligned
&s_{\Delta,a}^P(u_h,v_h)=\alpha_P\sum_{i=1}^{N_P}
(u_h(V_i)v_h(V_i)+h_{V_i}^2\nabla u_h(V_i)\cdot\nabla v_h(V_i))\\
&s_{\Delta,b}^P(u_h,v_h)=\beta_P\sum_{i=1}^{N_P}
(u_h(V_i)v_h(V_i)+h_{V_i}^2\nabla u_h(V_i)\cdot\nabla v_h(V_i))
\endaligned
\]
for all $u_h,v_h\in W_h(P)$. In the numerical experiments presented
in~\cite{Kirchhoff}, the stabilization parameters $\alpha_P$ and $\beta_P$ are
chosen as the mean value of the eigenvalues of the local matrices
$\adelta(\Pidelta u_h,\Pidelta v_h)$ and $\int_P\Pidelta u_h \Pidelta v_h\,dx$,
respectively. Hence, the matrices $\m{A}$ and $\m{M}$ of the associated
algebraic eigenvalue problem $\m{A}\m{u}=\lambda\m{M}\m{u}$ depend on
stabilization parameters and we shall discuss later on their choice.

The numerical results reported in~\cite{Kirchhoff} show that the first four
eigenvalues converge to the exact or extrapolated values with quadratic rate in
agreement with the theoretical results. 

In order to analyze the effect of the parameters on the computed eigenvalues,
in~\cite{Kirchhoff} it has been done the choice to multiply the stabilizing forms
$s_{\Delta,a}^P$ and $s_{\Delta,b}^P$ by the same value of the parameter. 
This makes very difficult to compare the results with the behavior described in
Section~\ref{se:parameter}, since the better choice would be to take large
values of $\alpha_P$ and small ones of $\beta_P$. Coherently, the conclusions
of~\cite{Kirchhoff} is that
the optimal choice is for value of the parameter in the middle. 
 
We end this subsection with some consideration relative to the buckling
problem of a clamped plate, which can be formulated as a spectral problem of
fourth order as follows: given a plane stress tensor field
$\bfeta:\Omega\to\RE^{2\times2}$, find the non vanishing deflection of the
plate $u$ and the eigenvalue (the buckling coefficient) $\lambda$ such that
\[
\aligned
&\Delta^2u=-\lambda\div(\bfeta\nabla u)&&\quad\text{in }\Omega\\
&u=\frac{\partial u}{\partial\bfn}=0&&\quad\text{on }\partial\Omega.
\endaligned
\]
Setting $a_\Delta(u,v)=\int_\Omega D^2 u: D^2v\,dx$ and 
$b_b(u,v)=\int_\Omega \bfeta\nabla u\nabla v\,dx$ the weak problem can be written
as~\eqref{eq:plate} with $b_b(u,v)$ on the right hand side instead of the
$L^2(\Omega)$-scalar product. 

The virtual element discretization of this problem has been discussed
in~\cite{buckling} using virtual element spaces of any degree $k\ge2$
introduced in~\cite{BM}:
\[
V_h(P)=\Big\{v_h\in\widetilde{V}_h(P):\int_P(v_h-\Pi_k^\Delta v_h)q\,dx=0\ 
\forall q\in\mathbb{P}^*_{k-3}(P)\cup\mathbb{P}^*_{k-2}(P)\Big\}
\]
where $\mathbb{P}^*_{\ell}(P)$ denotes homogeneous polynomials of degree $\ell$
with the convention that $\mathbb{P}^*_{-1}(P)=\{0\}$.
In the definition above the following notation have been used: for
$r=\max(3,k)$ and $s=k-1$
\[
\aligned
\widetilde{V}_h(P)&=\Big\{v_h\in H^2(P): \Delta^2v_h\in\mathbb{P}_{k-2}(P),
\ v_h|_{\partial P}\in C^0(\partial P),\\
&\qquad\nabla v_h|_{\partial P}\in C^0(\partial P)^2,\
v_h|_e\in\mathbb{P}_r(e),
\ \frac{\partial v_h}{\partial\bfn}\Big|_e\in\mathbb{P}_s(e)
\ \forall e\in\partial P\Big\}.
\endaligned
\]
and $\Pi_k^\Delta: H^2(P)\to \mathbb{P}_k(P)\subseteq\widetilde{V}_h(P)$ is
the computable projection operator solution of local problems:
\[
\aligned
&\adelta(\Pi_k^\Delta v,q)=\adelta(v,q)\quad\forall q\in\mathbb{P}_k(P)
\quad \forall v\in H^2(P)\\
&\widehat{\Pi_k^\Delta v}=\widehat{v}, \quad 
\widehat{\nabla\Pi_k^\Delta v}=\widehat{\nabla v}
\endaligned
\]
with $\widehat{v}=\frac1{N_P}\sum_{i=1}^{N_P}v(V_i)$, and $V_i$ the
$N_P$ vertices of $P$.
Then the discrete bilinear form $a_h$ is the sum of local terms composed by
a consistency and a stabilizing part
\[
a_h^P(u_h,v_h)=\adelta(\Pi_k^\Delta u_h,\Pi_k^\Delta v_h)
+s_{\Delta,a}^P(u_h-\Pi_k^\Delta u_h,v_h-\Pi_k^\Delta v_h).
\]
On the other hand the discrete right hand side can be constructed using the
degrees of freedom without the necessity of stabilizing
\[
b_b^P(u_h,v_h)=\int_P \bfeta\Pi_{k-1}^\Delta u_h \Pi_{k-1}^\Delta v_h\,dx.
\]
The analysis of this problem follows the same lines as the one for the
vibrating plate reported above. In~\cite{buckling} the theoretical results
are confirmed by several numerical experiments showing that the
method produces accurate solutions.

Another eigenvalue problem that involves a fourth order operator and whose
approximation has been analyzed when VEM spaces are used, can be found
in~\cite{transmission} where a transmission problem is considered.

\subsection{Eigenvalue problems related to linear elasticity models}

In this section we report on the analysis presented in~\cite{MRelasticity}
about the VEM approximation of the linear elasticity eigenvalue problem in two
space dimensions.
We consider the functional space $V=H^1_{\Gamma_D}(\Omega)^2$, where
$\Gamma_D$ is the part of $\partial\Omega$ where homogeneous Dirichlet
boundary condition are imposed, that is where the displacement of the
structure is prescribed and equal to zero.
The equation we are interested in is: find $\lambda\in\RE$ and
$\bfu\in V$ different from zero such that
\[
\int_{\Omega}\mathcal{C}\symgrad(\bfu):\symgrad(\bfv)\,d\bfx=\lambda
\int_\Omega\rho\bfu\cdot\bfv\,d\bfx,
\]
where $\rho>0$ is the density of the material, $\symgrad$ is the symmetric
gradient, and the compliance tensor $\mathcal{C}$ is defined as
\[
\mathcal{C}\bftau=2\mu\bftau+\lambda\mathrm{tr}(\bftau)\underline{\mathbf{I}},
\]
$\mu$ and $\lambda$ being the Lam\'e constants.

It is well known that the problem under consideration is associated with a
compact resolvent operator, so that the analysis of its approximation relies
on the standard Babu\v ska--Osborn theory.

The VEM discretization makes use of the natural vectorial generalization
$V_h\subset V$ based on the local standard spaces~\eqref{eq:vemP} and of the
projection operator defined as in~\eqref{eq:pinabla} with the due
modifications.

The discrete problem has the standard form of a generalized eigenvalue problem
$\m{A}\m{u}=\lambda\m{M}\m{u}$, where both matrices $\m{A}$ and $\m{M}$
contain a stabilization parameter. The analysis, both a priori and a
posteriori, is performed for a given choice of the stabilizing parameters,
which is analogous to the one described in Section~\ref{se:test-conf}
As in~\cite{Kirchhoff}, see also Section~\ref{se:plates},
in~\cite{MRelasticity} a study on the dependence of the stabilizing parameters
is performed when the parameters of $\m{A}$ and $\m{M}$ are chosen equal to
each other. It turns out that when the parameter is not too large or not too
small, the first eigenvalues of the considered example are not polluted by
spurious modes. This is compatible with what we have found in
Section~\ref{se:parameter}, although a safer choice would imply to take a
larger parameter for stabilizing the matrix $\m{A}$ and a smaller one (or even
zero) for $\m{M}$.

\bibliography{ref.bib}
\bibliographystyle{plain}

\end{document}